%% file: main.tex
\documentclass[a4paper]{amsart}
\usepackage{sty/preamble}

\begin{document}

\title[\texorpdfstring{$\nArtG \cong \nDualArtG$}{A isomorphic to dual A} for XL Artin groups]{The dual Artin isomorphism for Artin groups of XL type}

\author{Sean O'Brien}
\address{University of Glasgow}
\email{2812920O@student.gla.ac.uk}

\input{sections/pre/abstract}

\maketitle

\input{sections/introduction/introduction_main}

\input{sections/hurwitz_groups/hurwitz_groups_main}

\input{sections/geometry_of_Q/geometry_of_Q_main}

\input{sections/lemma/lemma_main}

\input{sections/theorem/theorem_main}

\bibliographystyle{alpha}
\bibliography{bib/refs}

\typeout{get arXiv to do 4 passes: Label(s) may have changed. Rerun}
\end{document}

%% file: sections/pre/abstract.tex
\begin{abstract}
	We show that an Artin group $A_\Gamma$ of XL type (with all defining integers satisfying $m_{ij}\geq 4$) is isomorphic to the corresponding dual Artin group for any choice of Coxeter element.
	Our proof involves the set of Hurwitz words~$Q$ which arise from the Hurwitz action on tuples of elements of the free group.
	We show that the canonical epimorphism from $A_\Gamma$ to the dual Artin group is an isomorphism if and only if the projection from $A_\Gamma$ to the Coxeter group is injective on the image of $Q$.
	Then, using geometric properties of $Q$ and the solution to the word problem for Coxeter groups, we show this to be the case when all $m_{ij} \geq 4$.
\end{abstract}

%% file: sections/introduction/introduction_main.tex
\section{Introduction}\label{sec:introduction}
\input{sections/introduction/introduction}

\input{sections/introduction/acknowledgements}

%% file: sections/introduction/introduction.tex
Artin groups are a broad family of groups that generalise braid groups and are closely related to Coxeter groups, yet remain relatively poorly understood.
Godelle and Paris \cite{godelle_paris_basic_2012} highlighted four open questions about them: torsion-freeness, triviality of the centre (for irreducible groups of infinite type), solvability of the word problem, and the $K(\pi,1)$ conjecture.
In 2026, OpenAI announced a solution to the word problem \cite{openai_parabolic_intersections_2026} and a proof of the $K(\pi,1)$ conjecture \cite{openai_Kpi1_2026}, which together with known results implies torsion-freeness and triviality of the centre.
While these announced proofs are likely to be valid, it is valuable to seek other proofs and machinery, and many fundamental problems remain open: the conjugacy and isomorphism problems are unsolved in general, and it is unknown whether all Artin groups are linear, or which are CAT(0).

Coxeter groups and Artin groups are determined by a Coxeter graph $\Gamma$, which encodes a number $m_{ij}$ for each pair of generators $\Set{\nCoxGenElt_i,\nCoxGenElt_j}$.
We denote the corresponding groups $\nCoxG$ and $\nArtG$ respectively.
Dual Artin groups are defined using certain presentations related to a Coxeter group $\nCoxG$ and are conjectured to be isomorphic to the corresponding Artin group $\nArtG$, as we discuss below.
See \Cref{sec:definitions_of_groups} for definitions.

In the last decade, dual Artin groups have played a central role in several breakthroughs concerning Artin groups.
In \cite{mccammond_sulway_artin_2017}, McCammond and Sulway showed that when~$\nCoxG$ is irreducible and affine, the corresponding Artin group~$\nArtG$ is torsion-free, has solvable word problem, and has trivial centre.
Subsequently, in \cite{paolini_salvetti_proof_2021}, Paolini and Salvetti used dual Artin groups to prove the $K(\pi,1)$ conjecture for the same class of Artin groups.
They were also central in \cite{delucchi_etal_dual_2024}, where Delucchi, Paolini and Salvetti proved the $K(\pi,1)$ conjecture for $\nArtG$ whose associated Coxeter group~$\nCoxG$ acts on~$\Hyp^2$, completing all rank $3$ cases.
Building on work of Birman, Ko and Lee \cite{birman_etal_new_1998}, dual Artin groups were introduced for finite~$\nCoxG$ of rank $3$ by Brady in \cite{brady_artin_2000}, then independently extended to all finite~$\nCoxG$ by Bessis in \cite{bessis_dual_2003} and Brady and Watt in \cite{brady_watt_kp_2002a}.

In contrast with \cite{paolini_salvetti_proof_2021} and \cite{delucchi_etal_dual_2024}, the OpenAI-announced proof of the $K(\pi,1)$ conjecture in \cite{openai_Kpi1_2026} does not use dual Artin groups.
Dual Artin groups offer a different approach to the $K(\pi,1)$ conjecture, as well as approaches to several problems which remain open.

Dual presentations are a principal source of Garside structures on Artin groups beyond the finite case.
When the interval below a Coxeter element in~$\nCoxG$, ordered by reflection length, is a lattice, the corresponding dual Artin group is a Garside group.
Even when the lattice property fails, as it does for many affine~$\nCoxG$ \cite{mccammond_dual_2015}, dual presentations can be used to embed the Artin group in a larger Garside group \cite{mccammond_sulway_artin_2017}.
Garside structures give a powerful approach to the word and conjugacy problems.

Dual presentations are also a source of non-positively curved geometry.
The dual braid complex, a classifying space for the braid group built from its dual presentation, is CAT(0) for braid groups on at most $7$ strands \cite{brady_mccammond_braids_2010, haettel_etal_6strand_2016, jeong_sevenstrand_2023}, and is conjectured to be CAT(0) in general.
Dougherty and McCammond showed that it is a deformation retract of a compactification of the space of monic centred polynomials with distinct roots, giving a topological bridge between the standard and dual presentations of the braid group \cite{dougherty_mccammond_geometric_2026}, and it is natural to ask whether other dual Artin groups admit similar connections.

The dual perspective also connects Artin groups to the Hurwitz action on reflection factorisations, an active area of research in algebraic combinatorics.
In any Coxeter group, the Hurwitz action is transitive on reduced reflection factorisations of a Coxeter element \cite{igusa_schiffler_exceptional_2010}, and Wegener recently showed the same for every element below a Coxeter element \cite{wegener_prefixes_2026}.

Finally, whether a dual Artin group is isomorphic to the corresponding Artin group is a natural question in its own right, and is needed to transfer any of the above techniques to Artin groups, so we now discuss it.

Given a Coxeter group $\nCoxG$ and Coxeter element $c \in \nCoxG$, we denote the corresponding dual Artin group $\nDualArtG(c)$.
There is a canonical epimorphism $\nArtG \to \nDualArtG(c)$, constructed in \Cref{sec:hurwitz_groups}.
This epimorphism is the precise sense in which dual Artin groups are conjectured to be isomorphic to Artin groups, and it is known to be an isomorphism in all established cases: finite $\nCoxG$ in \cite{brady_watt_kp_2002a} and \cite{bessis_dual_2003}, with a uniform proof by Chapuy and Douvropoulos in \cite{chapuy_douvropoulos_counting_2022} following \cite{bessis_finite_2015}, affine $\nCoxG$ in \cite{mccammond_sulway_artin_2017} and reproved in \cite{paolini_salvetti_proof_2021}, rank $3$ $\nCoxG$ in \cite{delucchi_etal_dual_2024}, and certain combinations of these due to Resteghini in \cite{resteghini_free_2024}.
In all except \cite{resteghini_free_2024}, the results hold irrespective of the choice of Coxeter element $c$.

In the cases above, we say that the dual Artin group is \emph{canonically isomorphic} to the Artin group.
It is believed that this holds in full generality, but systematic approaches to this problem are limited.
Extending the cases for which this is known and developing new approaches will be the focus of this paper.
The main result is the following theorem, which is proven as part of \Cref{thm:final_theorem}.
{\renewcommand*{\thetheorem}{\Alph{theorem}}
\begin{theorem}
	\label{thm:intro_main_theorem}
	For any Coxeter group $\nCoxG$ with all $m_{ij} \geq 4$, and any Coxeter element~$c \in \nCoxG$, the dual Artin group $\nDualArtG(c)$ is canonically isomorphic to the Artin group~$\nArtG$.
\end{theorem}
}
Such Coxeter groups are referred to as \emph{XL Coxeter groups}, and their corresponding Artin groups are \emph{Artin groups of XL type}, as in \cite{blufstein_etal_homomorphisms_2025}.

In \cite[Theorem 4.11]{mccammond_introduction_2005}, it is claimed that when all $m_{ij} \geq 6$, the dual Artin group $\nDualArtG(c)$ has a Garside structure for all Coxeter elements $c$.
However, a complete proof of this result has not yet appeared in the literature.
If this holds, then in combination with \Cref{thm:intro_main_theorem} it would give a Garside structure to Artin groups with all $m_{ij} \geq 6$.

\subsection*{Outline of proof}
We start by defining the Hurwitz action, which is needed to state \Cref{thm:intro_Q_theorem}, a key step towards \Cref{thm:intro_main_theorem}.
With standard braid generators $\Set{\sigma_1,\ldots,\sigma_{\nRank - 1}} $, the Hurwitz action is defined as follows.
\[
	\sigma_i \cdot (g_1,\ldots,g_\nRank) \; = \; ( g_1,\ldots, g_{i-1}, \; g_{i+1} ,\; g_{i+1}^{-1}g_i g_{i+1}, \; g_{i+2},\ldots ,g_\nRank).
\]
This was introduced by Hurwitz in \cite{hurwitz_ueber_1891} for the symmetric group, and has since been extensively studied, for example by Brieskorn \cite{brieskorn_automorphic_1988}.

The combinatorics of the Hurwitz action relevant to us are encoded in the set of all elements that occur at any position somewhere in the orbit of the Hurwitz action, called \emph{Hurwitz elements}, and denoted as follows.
\[
	\hurelt(\nTupleG) \coloneq \Set{r^\prime \given \exists (r_1,\ldots,r^\prime,\ldots,r_\nRank) \in B_\nRank \cdot (\nTupleG)}.
\]

In \cref{sec:hurwitz_groups}, we develop a generalisation of $\nDualArtG(c)$ based on the Hurwitz action.
This is used to give a presentation for $\nDualArtG(c)$ relating to the set of \emph{Hurwitz words}
\[
	Q \coloneq \hurelt(\nTupleFree),
\]
where $\Set{\nTupleFree}$ denotes the generating set of the free group $\nFree$.
We use this setup to state some known facts about the dual Artin group, including the following theorem.
Let $\Set{\nTupleCox}$ and $\Set{\nTupleArt}$ denote the standard generating sets for~$\nCoxG$ and~$\nArtG$ respectively, and let $\nPiArt \colon \nFree \to \nArtG$ be the epimorphism respecting generator indexing.
The following is proved as part of \Cref{cor:main_hurwitz_results_applied_to_dual_artin}.
{\renewcommand*{\thetheorem}{\Alph{theorem}}
\begin{theorem}
	\label{thm:intro_Q_theorem}
	Given a Coxeter group $\nCoxG$, with Coxeter element $c \coloneq \nCoxGenElt_{1}\cdots \nCoxGenElt_{\nRank} \in \nCoxG$, the Artin group $\nArtG$ is canonically isomorphic to the dual Artin group $\nDualArtG(c)$ if and only if the epimorphism $\nArtG \to \nCoxG$ is injective on $\hurelt(\nTupleArt) = \nPiArt(Q)$.
\end{theorem}
}
Though the statement of \Cref{thm:intro_Q_theorem} is new, related statements have previously appeared in the literature.
Resteghini \cite[Proposition 3.9]{resteghini_free_2024} gives a similar criterion involving two equivalence relations on the braid group, and Bessis \cite[Section~6, top of p.~13]{bessis_dual_2004} previously observed the Hurwitz group construction, given in \Cref{sec:hurwitz_groups} here, which leads to the proof.
Note that we cite the preprint~\cite{bessis_dual_2004}, not the published article, since Section~6 does not appear in the published version.

We now expand on how \Cref{thm:intro_Q_theorem} is used to prove \Cref{thm:intro_main_theorem}.
We define $\nPiCox \colon \nFree \to \nCoxG$ similarly to $\nPiArt$.
There is a commuting triangle of epimorphisms
\begin{equation}
	\label{eqn:commuting_triangle}
	\begin{tikzcd}[row sep=scriptsize]
		               & \nFree \ar[dl, "\nPiArt"'] \ar[dr, "\nPiCox"] &        \\
		\nArtG \ar[rr] &                                               & \nCoxG
	\end{tikzcd}.
\end{equation}
Using \Cref{thm:intro_Q_theorem}, we see that $\nArtG$ is canonically isomorphic to $\nDualArtG(c)$ precisely when, for every pair of Hurwitz words $q_1,q_2 \in Q$, the equality $\nPiCox(q_1) = \nPiCox(q_2)$ implies $\nPiArt(q_1) = \nPiArt(q_2)$ (see \Cref{cor:equality_in_W_implies_equality_in_A_implies_theorem}).
Furthermore, there is a solution to the word problem in~$\nCoxG$ due to Tits, which we use to control when $\nPiCox(q_1) = \nPiCox(q_2)$.

This control relies on the geometry of Hurwitz words~$Q$.
The Hurwitz action is closely related to the geometric action of $B_\nRank$ on $\nFree$ by automorphisms, denoted by $\star$, and this is used to show that $Q$ is equal to the orbit $B_\nRank \star \nFreeGenElt_1$ (the specific generator $\nFreeGenElt_1$ here is not important).
Then, realising $\nFree$ as the fundamental group of the disk with $\nRank$ punctures $\nPuncturedD$, we show that $Q$ consists of elements which correspond to simple loops (embeddings of $S^1$) which intersect the boundary of $\D$, and this is enough to prevent certain subwords relevant to the word problem in $\nCoxG$ from appearing in any $q \in Q$, which is used to control when $\nPiCox(q_1) = \nPiCox(q_2)$.
The main new idea in this paper is that geometric restrictions on simple loops in $\nPuncturedD$ constrain the Coxeter word problem, and hence the dual Artin isomorphism problem.

For example, a consequence of \Cref{lem:paired_exclusion} is that if $i \neq j$, then $\nFreeGenElt_i^{\pm 1}\nFreeGenElt_k^{\pm 1}\nFreeGenElt_i^{\pm 1}$ and $\nFreeGenElt_j^{\pm 1}\nFreeGenElt_k^{\pm 1}\nFreeGenElt_j^{\pm 1}$ cannot coexist as two subwords of any Hurwitz word $q \in Q$.
This means that any two alternating subwords $u_1,u_2$ of $q$ with $\ell(u_1),\ell(u_2) \geq 4$ cannot have exactly one generator in common, which greatly restricts how these subwords can interact under word operations in $\nCoxG$.
This is used to show that if all $m_{ij} \geq 4$, certain pairs $q_1,q_2 \in Q$ are equal in~$\nCoxG$ only if they are related by a relatively simple set of commuting word operations.

Given a Coxeter graph $\Gamma$, we define the subgroup~$\nBraidSubgrpG \leq B_\nRank$ generated by all~$\sigma_{ij}^{m_{ij}}$, where $\sigma_{ij}$ denotes a Birman--Ko--Lee generator in $B_\nRank$.
We show that for all $\hat{\beta} \in \nBraidSubgrpG$ and $x \in \nFree$, we have $\nPiArt(x) = \nPiArt(\hat{\beta} \star x)$.
Each $\nBraidSubgrpG$ provides a canonical way in which two $q_1,q_2 \in Q$ can be equal in $\nArtG$, and the study of how these subgroups sit inside $B_\nRank$ could be a fruitful approach to the dual Artin isomorphism problem.

Finishing our argument, with further restrictions on $Q$ derived from its geometry in $\nPuncturedD$, we show that if all $m_{ij} \geq 4$, the commuting word operations which equate a pair $q_1,q_2 \in Q$ in $\nCoxG$ must be realised by acting by elements of $\nBraidSubgrpG$ on $Q$, showing that if $q_1,q_2$ are equal in $\nCoxG$, then they are equal in $\nArtG$, and thus that $\nArtG$ is canonically isomorphic to $\nDualArtG(c)$.

After a preliminary version of this paper appeared on the arXiv, the author was made aware of unpublished work of Kluitmann \cite{kluitmann_isotropy_1991} which contains arguments that overlap with some of those presented here.
Specifically, the results comprising \Cref{thm:summary_of_Q_restrictions} appear in Kluitmann's preprint, and they are applied similarly to show that certain word reductions can be achieved within $Q$, resulting in a more general version of \Cref{cor:bhat_B} which applies to large type Coxeter groups.

%% file: sections/introduction/acknowledgements.tex
\subsection*{Acknowledgements}
I would like to thank my supervisors Jim Belk and Rachael Boyd for their invaluable guidance and insightful discussions, and specifically thank Michael Dougherty and Mark Powell for providing comments on drafts.
I also thank Michael Dougherty for bringing to my attention the unpublished preprint by Paul Kluitmann \cite{kluitmann_isotropy_1991} mentioned above.

I also thank Simeon Hellsten, Alicja Pietrzak, Lorna Richardson and Susanna Terron for their work on a software tool \cite{simple_loop_subwords_algorithm} which computes non-crossing drawability of free subwords.
This was an invaluable tool in the completion of this paper.

Claude AI tools were used to improve implementations and documentation of the software mentioned above, for which the main algorithms were developed by the named contributors.
I also acknowledge that ChatGPT 6 and Claude Opus 5.5 were useful for eliminating typos and minor errors in this paper.

This work was supported by the EPSRC under the Centre for Doctoral Training in Algebra, Geometry and Quantum Fields, grant number EP/Y035232/1.

%% file: sections/hurwitz_groups/hurwitz_groups_main.tex
\section{Coxeter groups, Artin groups, Hurwitz groups and dual Artin groups}\label{sec:hurwitz_groups}
\input{sections/hurwitz_groups/preamble}
\input{sections/hurwitz_groups/definitions}
\input{sections/hurwitz_groups/hurwitz_groups}

%% file: sections/hurwitz_groups/preamble.tex
In this section we will define most of the groups used in this paper, most importantly dual Artin groups, which we view as a special case of the more general class of Hurwitz groups.
We will then show that Hurwitz groups admit a presentation that depends on the set of Hurwitz words $Q$, and using this presentation we will show that Hurwitz groups inherit any relations that occur in Artin groups.
Thus, we construct an epimorphism from the Artin group to the dual Artin group, which is seen as the image of $\nArtG \to \nCoxG$ under a functor.
Properties of this functor show this epimorphism is an isomorphism if and only if $\nArtG \to \nCoxG$ is injective on $\nPiArt(Q)$.

%% file: sections/hurwitz_groups/definitions.tex
\subsection{Definitions of relevant groups}
\label{sec:definitions_of_groups}
A \emph{Coxeter graph} $\Gamma$ is a simplicial graph with vertices labelled $\Set{1,\ldots,\nRank}$ and edge labels $m_{ij} \in \Set{3,4,\ldots} \cup \Set{\infty}$.
We set $m_{ij} = 2$ whenever the vertices $i$ and $j$ are not joined by an edge, and call~$\nRank$ the \emph{rank} of $\Gamma$.
Given $k \in \N$ and group elements~$g$ and~$h$, let $\nPi{g}{h}{k}$ denote the alternating product of~$g$ and~$h$ beginning with~$g$ and of total length $k$.
For example, $\nPi{g}{h}{3} = ghg$.
The \emph{Artin group} $\nArtG$ associated to the Coxeter graph $\Gamma$ is
\[
	\nArtG = \GroupPres{\nTupleArt \mid \nPi{\nArtGenElt_i}{\nArtGenElt_j}{m_{ij}} = \nPi{\nArtGenElt_j}{\nArtGenElt_i}{m_{ij}}\text{ for all $i \neq j$ with $m_{ij}\ne \infty$} }.
\]
The \emph{braid group} $B_\nRank$ is the Artin group $\nArtG$ where $\Gamma$ is $\nRank-1$ vertices connected in a line, with each edge labelled $m_{ij} = 3$.
The $\nRank - 1$ generators in this description are the \emph{standard braid generators}, denoted $\Set{\sigma_1,\ldots,\sigma_{\nRank -1 }}$.

The \emph{Coxeter group} $\nCoxG$ associated to $\Gamma$ is the group generated by $s_1,\dots,s_n$ with relations corresponding to those in $\nArtG$ as well as $s_i^2 = 1$ for all $i$.
There is a canonical epimorphism $\nArtG \to \nCoxG$ mapping $\nArtGenElt_i \mapsto \nCoxGenElt_i$ for all $i$.

Given a group $G$ and a tuple $(\nTupleG) \in G^\nRank$, recall that the Hurwitz elements $\hurelt(\nTupleG)$ are all $g \in G$ which occur at any position in an element of the Hurwitz action orbit $B_\nRank \cdot (\nTupleG)$.
Note that all $r\in \hurelt(\nTupleG)$ actually appear as the \textit{first} coordinate of some tuple in $B_\nRank \cdot (\nTupleG)$.
In particular, if $r$ is the $i$th coordinate of $\beta\cdot (\nTupleG)$, then it is the first coordinate of
\[
	\sigma_1\cdots \sigma_{i-1}\beta\cdot(\nTupleG).
\]

\begin{definition}
	\label{def:hurwitz_group}
	Let $G$ be a group, and let $(g_1,\ldots,g_n)\in G^n$.
	The associated \emph{Hurwitz group} $H(g_1,\ldots,g_n)$ is the group with one generator $[r]$ for each $r\in\hurelt(\nTupleG)$, and relations
	\[
		[r_2]^{-1}[r_1][r_2]=[r_2^{-1}r_1r_2]
	\]
	for all $(r_1,r_2,\ldots,r_n)\in B_n\cdot (g_1,\ldots,g_n)$.
\end{definition}
\begin{remark}
	\label{rmk:bring_pairs_to_start_of_tuple}
	It is not important that the elements involved in the relations in \Cref{def:hurwitz_group} occur at the \emph{start} of a tuple.
	If $i<j$ with $r$ and $r^\prime$ occurring at coordinates~$i$ and~$j$ of $\beta \cdot (\nTupleG)$, then $r$ and $r^\prime$ occur at coordinates~$1$ and~$2$ of
	\[
		(\sigma_2\cdots\sigma_{j-1})(\sigma_1\cdots\sigma_{i-1})\beta \cdot (\nTupleG).
	\]
\end{remark}
Let $\mathcal{S}_\nRank$ denote the symmetric group on $\nRank$ letters.
A \emph{Coxeter element} of $\nCoxG$ is an element of the form $c=\nCoxGenElt_{\mu(1)}\cdots\nCoxGenElt_{\mu(\nRank)}$ for some $\mu\in\mathcal{S}_\nRank$.

The following definition is not the standard one, but it is equivalent.
Usually, dual Artin groups are defined as interval groups associated to Coxeter elements \cite[Definitions 2.6 and 2.7]{mccammond_sulway_artin_2017}, which we compare with \Cref{def:hurwitz_group}.
The equivalence follows from \cite[Proposition 3.5]{mccammond_dual_2015} and the fact that the Hurwitz action on any reduced reflection factorisation $(\nCoxGenElt_{\mu(1)},\ldots,\nCoxGenElt_{\mu(\nRank)})$ is transitive on the set of reduced reflection factorisations of the corresponding Coxeter element, proved in the general case by Igusa and Schiffler in \cite{igusa_schiffler_exceptional_2010}, with a shortened proof by Baumeister, Dyer, Stump and Wegener in \cite{baumeister_etal_note_2014}.
\begin{definition}
	\label{def:dual_artin_group}
	Let $\nCoxG$ be a Coxeter group and let $c = \nCoxGenElt_{\mu(1)}\cdots\nCoxGenElt_{\mu(\nRank)} \in \nCoxG$ be a Coxeter element.
	We define the \emph{dual Artin group} $\nDualArtG(c)$ to be the Hurwitz group $H(s_{\mu(1)},\ldots,s_{\mu(\nRank)})$.
\end{definition}
A priori, the isomorphism class of any dual Artin group $\nDualArtG(c)$ depends on the choice of Coxeter element $c$.
However, in this paper, we will only consider the Coxeter element $\nCoxGenElt_1\cdots\nCoxGenElt_\nRank$, and accordingly drop the appearance of $c$ in our notation.
Now, any choice of Coxeter element is hidden in the arbitrary indexing of the generators, and choosing a different Coxeter element is equivalent to appropriately permuting the $m_{ij}$.
The arguments in this paper will not depend on any permutation of the $m_{ij}$, and thus hold for any choice of Coxeter element.

%% file: sections/hurwitz_groups/hurwitz_groups.tex
\subsection{Abstract properties of Hurwitz groups}
\label{sec:abstract_hurwitz_groups}
Consider the category $\cngrp$ of~$n$-generated groups, whose objects are pairs $(G,(\nTupleG))$ with $G$ generated by $\Set{\nTupleG}$, and whose morphisms are group epimorphisms respecting the tuples.

The Hurwitz action commutes with homomorphisms, i.e.~if $\beta \cdot (g_1,\ldots,g_\nRank) = (r_1,\ldots,r_n)$, then $\beta \cdot (\phi(g_1),\ldots,\phi(g_\nRank)) = (\phi(r_1),\ldots,\phi(r_\nRank))$ for any relevant homomorphism~$\phi$.
This gives the braid group an action on $\cngrp$, i.e.~a homomorphism $B_\nRank \to \aut(\cngrp)$ mapping $\beta \mapsto F_\beta$, which is defined by
\[
	\begin{tikzcd}
		 & (G,(\nTupleG)) \ar[r, mapsto,"F_\beta"] & (G, \beta \cdot (\nTupleG)).
	\end{tikzcd}
\]
We refer to this observation as the \emph{functoriality of the Hurwitz action}.
\begin{lemma}
	\label{lem:hurwitz_action_on_hurwitz_group}
	Let $G$ be a group and let $(\nTupleG) \in G^\nRank$.
	For all $\beta \in B_\nRank$, if $\beta \cdot (\nTupleG) = (r_1,\ldots,r_\nRank)$, then $\beta \cdot (\nTupleSqG) = ([r_1],\ldots,[r_\nRank])$.
\end{lemma}
\begin{proof}
	The first equation is in $G^\nRank$, and the second is in $H(\nTupleG)^\nRank$.
	We prove this by induction on the word length of $\beta$ in the standard generators of $B_\nRank$.
	The base case where $\beta$ is the identity holds trivially.
	Now assume the claim for $\beta$, so $\beta \cdot (\nTupleG) = (r_1,\ldots,r_\nRank)$ and $\beta \cdot (\nTupleSqG) = ([r_1],\ldots,[r_\nRank])$.
	We see that
	\begin{align*}
		\sigma_i\beta \cdot (\nTupleG)   & = \left(r_1,\ldots,r_{i-1},\;r_{i+1},\;r_{i+1}^{-1}r_{i}r_{i+1},\;r_{i+2},\ldots,r_\nRank\right), \quad \text{and} \\
		\sigma_i\beta \cdot (\nTupleSqG) & = \left([r_1],\ldots,[r_{i-1}],\;[r_{i+1}],\;[r_{i+1}]^{-1}[r_{i}][r_{i+1}],\;[r_{i+2}],\ldots,[r_\nRank]\right).
	\end{align*}
	By \Cref{rmk:bring_pairs_to_start_of_tuple}, the relation $[r_{i+1}]^{-1}[r_{i}][r_{i+1}] = [r_{i+1}^{-1}r_{i}r_{i+1}]$	holds in $H(\nTupleG)$, so the claim holds for $\sigma_i\beta$.
	Relations of the form $[r_{i}][r_{i+1}][r_i]^{-1} = [r_{i}r_{i+1}r_i^{-1}]$ also hold in $H(\nTupleG)$, so the claim holds for~$\sigma_i^{-1}\beta$.
\end{proof}
The pair $(\nFree,(\nTupleFree))$ is initial in $\cngrp$, and for each $(G,(\nTupleG))$, we denote the relevant homomorphism as follows.
\[
	\pi_G \colon (\nFree,(\nTupleFree)) \to (G,(\nTupleG)).
\]
We may omit the tuples when they are clear from context.
Given a word $w\in \nFree$, we can denote its image $\pi_G(w)$ by $w(\nTupleG)$.
For example, $x_1x_2x_1^{-1}(g_1,g_2) = g_1g_2g_1^{-1}$.
\begin{lemma}
	\label{lem:q_relations_in_hurwitz_group}
	Let $G$ be a group and let $(\nTupleG) \in G^\nRank$.
	The relation
	\[
		q(\nTupleSqG) = [q(\nTupleG)]
	\]
	holds in $H(\nTupleG)$ for all $q \in Q$.
\end{lemma}
\begin{proof}
	Let $q \in Q$ with $\beta \in B_\nRank$ such that $q$ appears in $\beta \cdot (\nTupleFree) = (q_1, \ldots,q_\nRank)$.
	Considering the homomorphism $\pi_{H(\nTupleG)} \colon \nFree \to H(\nTupleG)$, by the functoriality of the Hurwitz action we have
	\begin{align*}
		\beta \cdot \left( \nTupleSqG \right) = \left( q_1(\nTupleSqG),\ldots,q_\nRank(\nTupleSqG) \right).
	\end{align*}
	Similarly, considering the homomorphism $\pi_G \colon \nFree \to G$, we have $\beta \cdot (\nTupleG) = (q_1(\nTupleG),\ldots,q_\nRank(\nTupleG))$.
	Therefore, by \Cref{lem:hurwitz_action_on_hurwitz_group}, we have
	\[
		\beta \cdot (\nTupleSqG) = \left( [q_1(\nTupleG)],\ldots,[q_\nRank(\nTupleG)] \right),
	\]
	and combining this with the first equation proves the result.
\end{proof}
\begin{corollary}
	\label{cor:hurwitz_group_finite_gen_set}
	The Hurwitz group $H(\nTupleG)$ is generated by $\Set{\nTupleSqG}$.
\end{corollary}
\begin{proof}
	The defining generating set for $H(\nTupleG)$ is $[\hurelt(\nTupleG)]$, which by the functoriality of the Hurwitz action is $[\pi_G(Q)]$.
	\Cref{lem:q_relations_in_hurwitz_group} equates each element of $[\pi_G(Q)]$ to a word in  $\Set{\nTupleSqG}$.
\end{proof}
It is now simple to see that $H \colon \cngrp \to \cngrp$ defined as below is a functor.
\begin{equation*}
	\begin{tikzcd}[row sep=tiny]
		 & (G,(\nTupleG)) \ar[r, mapsto, "H"]          & (H(\nTupleG),(\nTupleSqG))  \\
		 & (g_i \mapsto \phi(g_i)) \ar[r, mapsto, "H"] & ([g_i] \mapsto [\phi(g_i)])
	\end{tikzcd}
\end{equation*}
We can interpret \Cref{lem:hurwitz_action_on_hurwitz_group} as establishing equivariance of this functor with respect to the Hurwitz action on $\cngrp$, i.e.~$H \circ F_\beta = F_\beta \circ H$ for all $\beta \in B_\nRank$.
\begin{lemma}
	\label{lem:hurwitz_group_alt_presentation}
	Let $G$ be a group and let $(\nTupleG) \in G^\nRank$.
	Identifying each~$[g_i]$ with~$\nFreeGenElt_i$, the Hurwitz group $H(\nTupleG)$ has presentation
	\[
		\GroupPres{\nTupleFree \relations q_1 = q_2 \text{ if $q_1,q_2 \in Q$ and $\pi_G(q_1) = \pi_G(q_2)$}}.
	\]
\end{lemma}
\begin{proof}
	Let $J$ denote the group with presentation as above.
	By \Cref{lem:q_relations_in_hurwitz_group} and \Cref{cor:hurwitz_group_finite_gen_set}, the map $x_i \mapsto [g_i]$ induces an epimorphism $J \to H(\nTupleG)$.

	In the other direction, let the relation $[r_2]^{-1}[r_1][r_2] = [r_2^{-1}r_1r_2]$ in $H(\nTupleG)$ correspond to $\beta \cdot (\nTupleG) = (r_1,\ldots,r_\nRank)$.
	Let $\beta \cdot (\nTupleFree) = (q_1,\ldots,q_\nRank)$ so that $[r_i] = [q_i(\nTupleG)]$ for each $i$.
	By \Cref{lem:q_relations_in_hurwitz_group}, we have $[r_i] = q_i(\nTupleSqG)$, so mapping $[g_i] \mapsto x_i$ maps the relation to an equation which freely holds in $J$.
\end{proof}
Mapping $[g_i] \mapsto g_i$ for all $i$ induces an epimorphism $H(\nTupleG) \to G$, and this describes a natural transformation $H \Rightarrow \id_{\cngrp}$.
The resulting epimorphisms $H(\nTupleG) \to G$ are injective on $[\hurelt(\nTupleG)]$, so by restricting, we get the following commuting squares for each $\phi \in \cngrp$ with bijective top and bottom maps.
\begin{equation*}
	\begin{tikzcd}
		{[\hurelt(\nTupleG)]} \ar[d, "H(\phi)"'] \ar[r, "\sim"] & {\hurelt(\nTupleG)}   \ar[d,"\phi"] \\
		{[\hurelt(\nTuplePhiG)]} \ar[r, "\sim"]                 & {\hurelt(\nTuplePhiG)}
	\end{tikzcd}
\end{equation*}
\begin{corollary}
	\label{cor:equivalent_conditions_H_iso}
	Let $\phi \colon (G,(\nTupleG)) \to (G^\prime,(\nTuplePhiG)) \in \cngrp$.
	Consider $H(\phi) \colon H(\nTupleG) \to H(\nTuplePhiG)$.
	The following are equivalent.
	\begin{enumerate}
		\item $H(\phi)$ is an isomorphism.
		\item $H(\phi)$ is injective on  $[\hurelt(\nTupleG)]$.
		\item $\phi$ is injective on  $\hurelt(\nTupleG)$.
	\end{enumerate}
\end{corollary}
\begin{proof}
	$(1) \Rightarrow (2)$ is immediate and $(2) \Rightarrow (3)$ follows from the diagram above.
	Now, suppose $\phi \colon G \to G^\prime$ restricts to a bijection between $\hurelt(\nTupleG) = \pi_G(Q)$ and $\hurelt(\nTuplePhiG) = \pi_{G^\prime}(Q)$.
	Then, $\pi_{G^\prime}(q_1) = \pi_{G^\prime}(q_2)$ implies $\pi_{G}(q_1) = \pi_{G}(q_2)$, and $(3) \Rightarrow (1)$ follows by \Cref{lem:hurwitz_group_alt_presentation}.
\end{proof}
\subsection{Hurwitz groups applied to dual Artin groups}
We first show that any Artin group $\nArtG$ can be realised as a Hurwitz group, and then apply results of \Cref{sec:abstract_hurwitz_groups}.
Recall $\Set{\nTupleArt}$ denotes the defining generators of $\nArtG$.
\begin{lemma}
	\label{lem:artin_group_is_hurwitz_group}
	The epimorphism $H(\nTupleArt) \to \nArtG$ mapping $[t_i] \mapsto t_i$ is an isomorphism.
\end{lemma}
\begin{proof}
	It suffices to show that mapping $t_i \mapsto [t_i]$ induces a homomorphism $\nArtG \to H(\nTupleArt)$.
	Consider the function $\nFp{a}{b}{p} \coloneq \nPi{a}{b}{p+1}\nPi{a}{b}{p}^{-1} $.
	One can compute that
	\[
		\nFp{a}{b}{p}\nFp{a}{b}{p-1}\nFp{a}{b}{p}^{-1}=\nFp{a}{b}{p+1}.
	\]
	Assuming $i<j$, there exists a tuple in $B_n \cdot (\nTupleFree)$ which begins with  $x_i,x_j$.
	Acting once by $\sigma_1^{-1}$ on the pair $(\nFreeGenElt_i,\nFreeGenElt_j)$, we get the pair $(\nFp{\nFreeGenElt_i}{\nFreeGenElt_j}{1},\nFp{\nFreeGenElt_i}{\nFreeGenElt_j}{0})$.
	By repeatedly acting by $\sigma_1^{-1}$ we see that $\nFp{\nFreeGenElt_i}{\nFreeGenElt_j}{m_{ij}} \in Q$ for any $m_{ij}$.

	If $\nPi{\nArtGenElt_i}{\nArtGenElt_j}{m_{ij}} = \nPi{\nArtGenElt_j}{\nArtGenElt_i}{m_{ij}}$ in $\nArtG$, then by \Cref{lem:q_relations_in_hurwitz_group}, $\nFp{[\nArtGenElt_i]}{[\nArtGenElt_j]}{m_{ij}} = [\nArtGenElt_i]$ holds in $H(\nTupleArt)$, which is equivalent to $\nPi{[\nArtGenElt_i]}{[\nArtGenElt_j]}{m_{ij}} = \nPi{[\nArtGenElt_j]}{[\nArtGenElt_i]}{m_{ij}}$.
\end{proof}
We now highlight and apply some results to the groups $\nArtG$, $\nCoxG$ and $\nDualArtG$.
\begin{corollary}
	\label{cor:main_hurwitz_results_applied_to_dual_artin}
	Let $\Gamma$ be a Coxeter graph and consider the Artin group $\nArtG$, Coxeter group $\nCoxG$, and canonical epimorphism $p \colon \nArtG \to \nCoxG$.
	\begin{enumerate}
		\item Identifying each $[\nCoxGenElt_i]$ with $x_i$, the dual Artin group $\nDualArtG = H(\nTupleCox)$ has presentation
		      \[
			      \GroupPres{\nTupleFree \relations q_1 = q_2 \text{ if $q_1,q_2 \in Q$ and $\pi_W(q_1) = \pi_W(q_2)$}}.
		      \]
		\item Mapping $[s_i] \mapsto s_i$ induces an epimorphism $\phi \colon \nDualArtG \to \nCoxG$.
		      \smallskip
		\item Mapping $t_i \mapsto [s_i]$ induces an epimorphism $\psi \colon \nArtG \to \nDualArtG$ which is an isomorphism if and only if $p$ is injective on  $\nPiArt(Q)$.
	\end{enumerate}
\end{corollary}
\begin{proof}
	Statement~$(1)$ is \Cref{lem:hurwitz_group_alt_presentation}.
	Statement~$(2)$ is the morphism $H(\nTupleCox) \to \nCoxG$.
	By \Cref{lem:artin_group_is_hurwitz_group}, mapping $t_i \mapsto [t_i]$ induces an isomorphism $H(\nTupleArt) \cong \nArtG$, and statement~$(3)$ follows by \Cref{cor:equivalent_conditions_H_iso}.
\end{proof}
The epimorphism $\psi \colon \nArtG \to \nDualArtG$ from \Cref{cor:main_hurwitz_results_applied_to_dual_artin} is the candidate \emph{canonical} epimorphism that we wish to show is an isomorphism.
The dual Artin isomorphism problem consists of showing $\psi$ is injective, which we have shown is equivalent to showing that $\nArtG \to \nCoxG$ is injective on $\nPiArt(Q) = \hurelt(\nTupleArt)$.
We note this is also equivalent to $\psi$ being injective on~$\nPiArt(Q)$.
The maps from \Cref{cor:main_hurwitz_results_applied_to_dual_artin} fit into the following commutative diagram of epimorphisms, cf.~\eqref{eqn:commuting_triangle}.
\begin{equation}
	\label{eqn:commuting_triangles_with_dual}
	\begin{tikzcd}[row sep=scriptsize]
		       &  & {\nFree}                 \\
		       &  &                          \\
		\nArtG &  & {\nDualArtG} &  & \nCoxG
		\arrow["{\nPiArt}"', from=1-3, to=3-1]
		\arrow["{\nPiDualArt}", from=1-3, to=3-3]
		\arrow["{\nPiCox}", from=1-3, to=3-5]
		\arrow["\psi", from=3-1, to=3-3]
		\arrow["\phi", from=3-3, to=3-5]
		\arrow[from=3-1, to=3-5, bend right=20, "p"']
	\end{tikzcd}
\end{equation}
We finish with a simple restatement of (3) from \Cref{cor:main_hurwitz_results_applied_to_dual_artin}.
\begin{corollary}
	\label{cor:equality_in_W_implies_equality_in_A_implies_theorem}
	The Artin group $\nArtG$ is canonically isomorphic to the dual Artin group $\nDualArtG$ if and only if
	\[
		\nPiCox(q_1) = \nPiCox(q_2) \implies \nPiArt(q_1) = \nPiArt(q_2)
	\]
	for all pairs of Hurwitz words $q_1,q_2 \in Q$.
\end{corollary}
\begin{proof}
	This follows by the commutative triangle \eqref{eqn:commuting_triangle}.
\end{proof}

%% file: sections/geometry_of_Q/geometry_of_Q_main.tex
\section{The geometry of Hurwitz words}
\label{sec:bessis_Q_description}
\input{sections/geometry_of_Q/preamble}
\input{sections/geometry_of_Q/hurwitz_action_and_artin_action}
\input{sections/geometry_of_Q/restrictions_on_Q}

%% file: sections/geometry_of_Q/preamble.tex
Given \Cref{cor:equality_in_W_implies_equality_in_A_implies_theorem}, an understanding of $Q$ and when $\nPiCox(q_1) = \nPiCox(q_2)$ would be useful to make progress with the dual Artin isomorphism problem.
To that end, in this section we explore the geometry of elements of $Q$, considered as homotopy class representatives in $\pi_1(\nPuncturedD,-i) \cong \nFree$.
A thorough treatment of reduced reflection factorisations in the free group is given in \cite{bessis_dual_2006}.
There, Bessis also shows that Hurwitz words are \emph{non-crossing}, a property of words in $\nFree$ (generated by $\Set{\nTupleFree}$ here) which we will define in this section.
However, most of what is needed for this paper can be shown using classical results for the braid group, and we make those arguments here.

Let $\D$ be the closed unit disk in $\C$ centred at the origin.
Let
\[
	-1<p_1<\cdots <p_n<1
\]
be points on the real axis, and let $P_n=\{p_1,\ldots,p_n\}$.
We know that $\nFree$ is isomorphic to $\pi_1(\nPuncturedD, -i)$.
For each $\nPoint_j$, let $\lambda_j$ be the line segment from $\nPoint_j$ to $+i$.
A picture of this setup is below.
\[
	\begin{tikzpicture}[,every node/.append style={scale=1}, inner sep=0pt, outer sep=0pt]
		\input{figs/points_on_C_simplified.tex}

	\end{tikzpicture}%

\]
Let $\gamma \colon S^1 \to \nPuncturedD$ be a loop based at $-i$ such that all crossings with the $\lambda_j$'s are transverse and $\gamma(S^1) \cap \partial \D = \Set{-i}$.
We can define the \emph{word representing $\gamma$} by the following procedure.
As we traverse $\gamma$, we may encounter the lines $\lambda_j$.
Every time we cross $\lambda_j$ from left to right, this counts as a factor of $\nFreeGenElt_j$ in the word.
Crossing from right to left counts as a factor of $\nFreeGenElt_j^{-1}$.
We traverse all of $\gamma$, appending elements of $\Set{\nTupleFree}$ or inverses, until we arrive back at $-i$.
Given a word $w$, we define a \emph{loop representing $w$} to be any $\gamma$ as above with representing word~$w$.
\begin{remark}
	These definitions give an isomorphism $\pi_1(\nPuncturedD, -i) \xrightarrow{\sim} \nFree$ which is opposite to the orientation in $\C$, i.e.~the generators of $\nFree$ correspond to clockwise loops, and opposite to that used in \cite{bessis_dual_2006}.
	We made this choice to improve left-to-right readability.
\end{remark}
With this setup, we say a word $w$ over $\Set{\nTupleFree} $ is \emph{non-crossing} if there exists a loop $\gamma \colon S^1 \to \nPuncturedD$ based at $-i$ which is \emph{simple} ($\gamma$ is injective), and whose corresponding word is $w$.
Similarly, we say some $w \in \nFree$ is non-crossing if the reduced word corresponding to $w$ is non-crossing.

%% file: figs/points_on_C_simplified.tex
\definecolor{c819d43}{RGB}{129,157,67}
\definecolor{ca2a2a2}{RGB}{162,162,162}
\definecolor{c979797}{RGB}{151,151,151}

  \path[draw=c819d43,line width=0.02cm] (5.6203, 22.8362) -- (4.7126, 24.3302);

  \path[draw=c819d43,line width=0.02cm] (3.7777, 22.8362) -- (4.7126, 24.3302);

  \path[draw=c819d43,line width=0.02cm] (4.4304, 22.8362) -- (4.7126, 24.3302);

  \path[draw=ca2a2a2,line width=0.02cm,dash pattern=on 0.06cm off 0.02cm,cm={ 0.0,1.0,1.0,0.0,(-29.7, 29.7)}] (-6.8616, 34.4126) ellipse (1.4918cm and 1.492cm);

  \node[text=black,line width=0.005cm,xscale=1.0001,yscale=0.9999,anchor=south west] (text27967) at (5.5024, 22.5391){$\scriptstyle \nPoint_\nRank$};

  \node[text=black,line width=0.005cm,xscale=1.0001,yscale=0.9999,anchor=south west] (text57948) at (4.2454, 22.5391){$\scriptstyle \nPoint_2$};

  \node[text=black,line width=0.005cm,xscale=1.0001,yscale=0.9999,anchor=south west] (text57988) at (4.4858, 21.5323){$\scriptstyle -i$};

  \node[text=black,line width=0.005cm,xscale=1.0001,yscale=0.9999,anchor=south west] (text57948-6) at (3.6434, 22.5391){$\scriptstyle \nPoint_1$};

  \node[text=black,line width=0.005cm,xscale=1.0001,yscale=0.9999,anchor=south west] (text66017) at (4.7852, 22.8156){$\ldots$};

  \node[text=black,line width=0.005cm,xscale=1.0001,yscale=0.9999,anchor=south west] (text75234) at (4.1563, 23.1156){$\scriptscriptstyle\lambda_2$};

  \node[text=black,line width=0.005cm,xscale=1.0001,yscale=0.9999,anchor=south west] (text75238) at (5.514, 23.1155){$\scriptscriptstyle\lambda_\nRank$};

  \node[text=black,line width=0.005cm,xscale=1.0001,yscale=0.9999,anchor=south west] (text75242) at (3.6269, 23.1149){$\scriptscriptstyle\lambda_1$};

  \path[draw=black,fill=c979797,line width=0.0487cm,cm={ 0.0,1.0,1.0,0.0,(-29.7, 29.7)}] (-6.8638, 34.1304) circle (0.0458cm);

  \path[draw=black,fill=c979797,line width=0.0487cm,cm={ 0.0,1.0,1.0,0.0,(-29.7, 29.7)}] (-6.8638, 33.4777) circle (0.0458cm);

  \path[draw=black,fill=c979797,line width=0.0487cm,cm={ 0.0,1.0,1.0,0.0,(-29.7, 29.7)}] (-6.869, 35.3288) circle (0.0458cm);

  \path[draw=black,fill=c979797,line width=0.0487cm,cm={ 0.0,1.0,1.0,0.0,(-29.7, 29.7)}] (-8.3535, 34.4126) circle (0.0458cm);

%% file: sections/geometry_of_Q/hurwitz_action_and_artin_action.tex
\subsection{Hurwitz words are non-crossing}
\label{sec:hurwitz_action_and_artin_action}
We now consider the braid group as the mapping class group of $\nPuncturedD$.
Let $\mu_i$ be the horizontal line segment from~$\nPoint_i$ to~$\nPoint_{i+1}$.
Let $\hat{\mu}_i$ be a loop which follows close to $\mu_i$ and which surrounds $\nPoint_i$ and $\nPoint_{i+1}$.
In the braid group action on $\nPuncturedD$, each generator $\sigma_i$ acts as a clockwise half twist about the loop $\hat{\mu}_i$.
This induces an action of $B_\nRank$ on $\nFree$ by post-composing homotopy class representatives in $\pi_1(\nPuncturedD, -i)$ with mapping class representatives in $B_\nRank$.
This action on $\nFree$, which we denote by $\star$, is described by the following equations.
\[
	\sigma_i \star \nFreeGenElt_j =
	\begin{cases}
		\nFreeGenElt_{j}\nFreeGenElt_{j+1}\nFreeGenElt_{j}^{-1} & \text{ if } j=i    \\
		\nFreeGenElt_{j-1}                                      & \text{ if } j=i+1  \\
		\nFreeGenElt_j                                          & \text{ otherwise.}
	\end{cases}
\]
We call this the \emph{Artin action} \cite{artin_theorie_1925}.
The following lemma relates the Artin action on the generators to the Hurwitz action.
Note that it only holds for $(\nTupleFree)$, where the implicit ordering of the generators of $\nFree$ is important and encodes elements of $\Hom(\nFree,\nFree)$.
This is also proven in \cite[Corollary 4.7]{resteghini_free_2024}.
\begin{lemma}
	\label{lem:artin_vs_hurwitz_action}
	Let $\star$ and $\cdot$ respectively denote the Artin action on $\nFree$ and the Hurwitz action on $\nFree^\nRank$.
	For any $\beta \in B_\nRank$, we have
	\[
		(\beta \star \nFreeGenElt_1,\ldots,\beta \star \nFreeGenElt_\nRank) = \beta^{-1} \cdot (\nFreeGenElt_1,\ldots,\nFreeGenElt_\nRank) .
	\]
\end{lemma}
\begin{proof}
	We identify $\nFree^\nRank$ with the set of homomorphisms $f \colon \nFree \to \nFree$, denoted $\Hom(\nFree, \nFree)$.
	Let $\beta \in B_\nRank$.
	The Artin action on $\nFree$ gives $B_\nRank$ a left action on $\Hom(\nFree, \nFree)$ by
	\[
		(\beta \cdot f)(x) = f(\beta^{-1} \star x).
	\]
	There is another action, given by
	\[
		(\beta \star f)(x) = \beta \star f(x),
	\]
	where the $\star$ on the right-hand side is the Artin action.
	One can verify using the standard generators for $B_\nRank$ that the action denoted $\beta \cdot f$ is the Hurwitz action of $B_\nRank$ on $\nFree^\nRank$, and the action denoted $\beta \star f$ is the application of the Artin action at each coordinate in $\nFree^\nRank$.
	We prove the result by taking $f$ to be the identity in $\Hom(\nFree,\nFree)$, corresponding to the tuple $(\nTupleFree)$.
\end{proof}
\begin{corollary}
	\label{cor:hurref_is_image_of_artin_action}
	Let $\star$ denote the Artin action on $\nFree$.
	Then $Q = B_\nRank \star \nFreeGenElt_1$.
\end{corollary}
\begin{proof}
	All elements of $Q$ occur as the first coordinate of $\beta \cdot (\nTupleFree)$ for some $\beta \in B_\nRank$.
	The result follows by \Cref{lem:artin_vs_hurwitz_action}.
\end{proof}
\begin{corollary}
	\label{cor:Q_is_noncrossing}
	Every $q \in Q$ is non-crossing.
\end{corollary}
\begin{proof}
	The generator $\nFreeGenElt_1$ is non-crossing.
	The action of $B_\nRank$ on $\nPuncturedD$ which defines~$\star$ is by homeomorphisms, so if $w \in \nFree$ is non-crossing, then $\beta \star w$ is non-crossing for all $\beta \in B_\nRank$.
	The result follows by \Cref{cor:hurref_is_image_of_artin_action}.
\end{proof}

%% file: sections/geometry_of_Q/restrictions_on_Q.tex
\subsection{Restrictions imposed on Hurwitz words by the non-crossing property}
\label{sec:Q_restrictions}
In this subsection, we will use \Cref{cor:Q_is_noncrossing} and plane drawings to impose restrictions on elements of $Q$.
We define an \emph{alternating word} to be any word whose letters alternate between two generators.
If $w$ is an alternating word in which the generators $\nFreeGenElt_i$ and $\nFreeGenElt_j$ appear, then we say that $w$ is an \emph{$(\nFreeGenElt_i,\nFreeGenElt_j)$-alternating word}.
For example, both $\nFreeGenElt_i\nFreeGenElt_j\nFreeGenElt_i^{-1}\nFreeGenElt_j^{-1}$ and $\nFreeGenElt_j\nFreeGenElt_i\nFreeGenElt_j\nFreeGenElt_i^{-1}$ are $(\nFreeGenElt_i,\nFreeGenElt_j)$-alternating words.
We say that an alternating subword $u$ is \emph{maximal} if it is not contained in any longer alternating subword.
A \emph{change of sign} occurs after the $k$th letter of an $(\nFreeGenElt_i,\nFreeGenElt_j)$-alternating word if the $k$th and $(k+1)$st letters are $\nFreeGenElt_i^{\pm 1}\nFreeGenElt_j^{\mp 1}$ or $\nFreeGenElt_j^{\pm 1}\nFreeGenElt_i^{\mp 1}$.

The main results are summarised in the following theorem.
Statement~(1) is proven in \Cref{lem:paired_exclusion}, statement (2) is proven in \Cref{lem:crossing_diagonals_and_one_change_of_sign} and \Cref{lem:change_sign_close_to_middle}, and statement~(3) is proven in \Cref{lem:possible_lengths_other_crossings}.
Similar results are obtained in \cite{kluitmann_isotropy_1991}.
\begin{theorem}
	\label{thm:summary_of_Q_restrictions}
	Let $q \in Q$.
	\begin{enumerate}
		\item Let $u \coloneq \nFreeGenElt_i^{\pm 1} w \nFreeGenElt_i^{\pm 1}$ and $v \coloneq \nFreeGenElt_j^{\pm 1} w^\prime \nFreeGenElt_j^{\pm 1}$ be two subwords of $q$, where $i \neq j$ and neither $w$ nor $w^\prime$ contains any $\nFreeGenElt_i^{\pm 1}$ or $\nFreeGenElt_j^{\pm 1}$ letters.
		      Then the set of letters which occur an odd number of times in $w$ is disjoint from the set of letters which occur an odd number of times in $w^\prime$.
		      \smallskip
		\item All maximal alternating subwords $u$ of $q$ with $\ell(u) \geq 3$ have exactly one change of sign and:
		      \begin{itemize}
			      \item If $\ell(u) = 2n$, the change of sign occurs after the $n$th letter in  $u$.
			      \item If $\ell(u) = 2n + 1$, the change of sign occurs after the $n$th or $(n+1)$st letter in $u$.
		      \end{itemize}
		      \smallskip
		\item If  $u$ and  $v$ are two distinct maximal alternating subwords in  $q$  consisting of the same generators and  $\ell(u),\ell(v) \geq 3$, then $\Abs{\ell(u) - \ell(v)} \leq 2$.
		      Moreover, $\Abs{\ell(u) - \ell(v)} \leq 1$ if either $\ell(u)$ or $\ell(v)$ is even.
	\end{enumerate}
\end{theorem}
Note that all of the $\pm$ signs are independent in statement~(1), e.g.~the subwords $x_5x_3^{-1}x_5^{-1}$ and $x_2^{-1}x_3x_2^{-1}$ cannot both occur in an element of~$Q$.

We now define some notation and assert some simple results that allow us to make useful drawings of simple loops corresponding to elements of $Q$.
For all $j\ne k$, let $\mu_{jk}$ denote the arc that connects the points $p_j$ and $p_k$ and passes below any intermediate punctures.
Note that for each $j$, we have $\mu_{j,j+1}=\mu_j$, defined in \Cref{sec:hurwitz_action_and_artin_action}.
\Cref{fig:upside_down_ice_cream_cone} shows all of the arcs $\lambda_i$ and $\mu_{jk}$ in the case where~$n=5$.
\begin{figure}[t]
	\begin{tikzpicture}[baseline=(baseline_text.center),every node/.append style={scale=1}, inner sep=0pt, outer sep=0pt]
		\input{figs/all_arcs.tex}
	\end{tikzpicture}%

	\caption{All $\lambda_{i}$ (green) and $\mu_{ij}$ (black) arcs.}
	\label{fig:upside_down_ice_cream_cone}
\end{figure}
Let~$w$ be a non-crossing free word.
We define a \emph{minimal simple loop representing~$w$} to be any simple loop representing~$w$ which crosses each $\lambda_j$ and $\mu_{jk}$ arc a minimal number of times in its isotopy class.
By a similar argument to that in \cite[Section 1.2]{farb_margalit_primer_2011}, a minimal simple representing loop can always be constructed from a simple representing loop.
So, by \Cref{cor:Q_is_noncrossing}, there is a minimal simple loop representing every $q \in Q$.

In order to pass freely between subwords of $q \in Q$ and sub-arcs of minimal simple loops representing $q$, we define the following notation.
Let $\gamma$ be a minimal simple loop representing $q$.
We can associate points on the image of $\gamma$ to letters in~$q$, and, given any subword $u$ of $q$, we can also associate the corresponding sub-arc of~$\gamma$.
We denote this sub-arc by~$\gamma_u$.

We now prove a simple but very useful property of words in $Q$, also observed in \cite[Lemma 3.4]{bessis_dual_2006}.
We say a word is \emph{square-free} if it is free of squares of any generator or inverse generator.
An element of $\nFree$ is square-free if it is square-free as a reduced word.
The proof here is typical of the kind of arguments one can make about $Q$ using its geometry in $\nPuncturedD$.
\begin{lemma}
	\label{lem:Q_square_free}
	Every element of $Q$ is square-free.
\end{lemma}
\begin{proof}
	Let $q \in Q$ and let $\gamma$ be a minimal simple loop representing $q$.
	Now, suppose there is a factor $u = x_i^{2}$ in $q$.
	One of the possible pictures for $\gamma_u$ follows.
	\[
	\begin{tikzpicture}[,every node/.append style={scale=1}, inner sep=0pt, outer sep=0pt]
		\input{figs/square_impossible.tex}
	\end{tikzpicture}%

	\]
	The region $R$ shaded in grey contains the head of  $\gamma_u$.
	Since $\gamma$ is based at a point of~$\partial \D$, the region $R$ does not contain the base point.
	So, $\gamma$ must next cross a $\lambda$ line, and it can only cross $\lambda_i$.
	Since $q$ is reduced, the letter after $u$ is not $x_i^{-1}$, so we have shown~$q$ contains the factor $x_i^3$.
	We can repeat this argument for the new~$x_i^2$ factor which is a suffix of the $x_i^3$ factor, and can repeat this indefinitely, but $q$ is finite.
	The same argument works for the other pictures and for the factor $x_i^{-2}$.
\end{proof}
\Cref{lem:Q_square_free} tells us that Hurwitz words are determined by a sequence of non-repeating generators, each with exponent $\pm 1$.
\begin{lemma}
	\label{lem:paired_exclusion}
	Let $q \in Q$.
	Suppose that $u = \nFreeGenElt_i^{\pm 1}w\nFreeGenElt_i^{\pm 1}$ and $v = \nFreeGenElt_j^{\pm 1}w^\prime \nFreeGenElt_j^{\pm 1}$ are subwords of $q$, where $i \neq j$ and neither $w$ nor $w^\prime$ contains any $\nFreeGenElt_i^{\pm 1}$ or $\nFreeGenElt_j^{\pm 1}$ letters.
	Define the following for $w$, and similarly for $w^\prime$:
	\begin{align*}
		\mathcal{L}_w^{\text{odd}} & \coloneq \Set{\text{Generators which occur an odd number of times in } w}.
	\end{align*}
	Then $\mathcal{L}_w^{\text{odd}} \cap \mathcal{L}_{w^\prime}^{\text{odd}} = \emptyset$.
\end{lemma}
\begin{proof}
	Let $\gamma$ be a minimal simple loop representing $q$.
	There is a closed region $R$ which is bounded by $\gamma_u$ and the segment of $\lambda_i$ between the first and last crossings of $\lambda_i$.
	Similarly, there is a closed region $R^\prime$ corresponding to $v$.

	Since the $\lambda$ lines extend to $\partial \D$, the region $R$ contains all points $\nPoint_k$ corresponding to letters in $\mathcal{L}^\text{odd}_w$, and $R^\prime$ contains all points corresponding to letters in $\mathcal{L}^\text{odd}_{w^\prime}$.
	So, if $\mathcal{L}_w^{\text{odd}} \cap \mathcal{L}_{w^\prime}^{\text{odd}} \neq \emptyset$, then $R$ and $R^\prime$ intersect.
	However, their boundaries do not cross, so they must be nested.
	If $R \subseteq R^\prime$, then $\lambda_i$ must pass through the boundary of $R^\prime$, so there must be an $\nFreeGenElt_i^{\pm 1}$ letter in $v$, a contradiction.
	We get a similar contradiction if $R^\prime \subseteq R$.
\end{proof}
\begin{lemma}
	\label{lem:exists_change_of_sign}
	Let $q \in Q$.
	Any maximal alternating subword $u$ of $q$ with $\ell(u) \geq 3$ contains a change of sign.
\end{lemma}
\begin{proof}
	Let $\gamma$ be a minimal simple loop representing $q$.
	We can assume $u$ is an $(\nFreeGenElt_i,\nFreeGenElt_k)$-alternating word which begins with $\nFreeGenElt_i$ or $\nFreeGenElt_i^{-1}$.
	So, we need to show that $u$ is not $\nFreeGenElt_i\nFreeGenElt_k\nFreeGenElt_i\ldots$ or $\nFreeGenElt_i^{-1}\nFreeGenElt_k^{-1}\nFreeGenElt_i^{-1}\ldots$.
	Since $\ell(u) \geq 3$, there are two letters in $u$ which correspond to the generator $\nFreeGenElt_i$.
	If those two letters are of opposite sign then we are done.
	So we can assume they are of the same sign.
	With this assumption, if $i < k$, there is a length 3 prefix $u^\prime$ of $u$ such that $\gamma_{u^\prime}$ looks like the picture below, and we can make a symmetric picture if $k < i$.
	\begin{equation*}
	\begin{tikzpicture}[baseline=(current bounding box.center),every node/.append style={scale=1}, inner sep=0pt, outer sep=0pt]
		\input{figs/iji_turn_02.tex}
	\end{tikzpicture}%

	\end{equation*}
	In this picture, there is a section of~$\lambda_i$, denoted~$\rho$, between the two intersections of~$\gamma_{u^\prime}$ with~$\lambda_i$.
	The region~$R$, coloured grey, is enclosed by~$\gamma_{u^\prime}$ and~$\rho$.
	There is one end of~$\gamma_{u^\prime}$ which is inside~$R$.
	Inside~$R$, the relevant sub-arc of~$\gamma$ can only cross~$\lambda_i$ or~$\lambda_k$, so everything inside~$R$ contributes to~$u$.
	The only way this end can escape~$R$ (which it must) is by passing through $\rho$.
	Doing so would correspond to a letter $\nFreeGenElt_i^{\pm 1}$ with sign opposite to the sign of the two crossings at the ends of $\rho$.
\end{proof}
\begin{lemma}
	\label{lem:crossing_diagonals_and_one_change_of_sign}
	Let $q \in Q$.
	\begin{enumerate}
		\item For all $i \neq j$ the sets $Z_+ = \Set{\nFreeGenElt_i\nFreeGenElt_j^{-1}, \nFreeGenElt_j\nFreeGenElt_i^{-1}}$ and $Z_- = \Set{\nFreeGenElt_i^{-1}\nFreeGenElt_j, \nFreeGenElt_j^{-1}\nFreeGenElt_i}$ are mutually exclusive as subwords of $q$, i.e.~if any two-letter subword of $q$ is in $Z_+$, then no two-letter subword of $q$ is in $Z_-$, and vice versa.
		      \smallskip
		\item All maximal, length $\geq 3$, alternating subwords of $q$ have exactly one change of sign.
	\end{enumerate}
\end{lemma}
\begin{proof}
	Statement (1) follows by the diagram below.
	\begin{equation*}
	\begin{tikzpicture}[baseline=(current bounding box.center),every node/.append style={scale=1}, inner sep=0pt, outer sep=0pt]
		\input{figs/crossing_diagonals.tex}
	\end{tikzpicture}%

	\end{equation*}
	Without specifying direction, the solid line corresponds to the subwords in $Z_+$ and the dashed line to those in $Z_-$.

	For (2), let $u$ denote such a subword, which we may assume is an $(\nFreeGenElt_i,\nFreeGenElt_j)$-alternating subword of $q$.
	Assume the first letter of $u$ has exponent $+1$.
	\Cref{lem:exists_change_of_sign} tells us there exists a change of sign in $u$.
	We are free to choose $i$ and $j$ such that this first change of sign is $\nFreeGenElt_i\nFreeGenElt_j^{-1}$.
	If the alternating word continues, then by (1), the next letter after $\nFreeGenElt_i\nFreeGenElt_j^{-1}$ is not $\nFreeGenElt_i$.
	So the subword continues $\nFreeGenElt_i\nFreeGenElt_j^{-1}\nFreeGenElt_i^{-1}$.
	If the alternating sequence continues still, the next letter after $\nFreeGenElt_i\nFreeGenElt_j^{-1}\nFreeGenElt_i^{-1}$ is not $\nFreeGenElt_j$, and so on.
	We make a symmetric argument if the first letter of $u$ has exponent $-1$.
\end{proof}
The $\pm$ in $Z_\pm$ above refers to the exponent of the first letter in those changes of sign.
We will retain the names of these sets as part of our nomenclature.
For instance, a change of sign as in $Z_+$ is a \emph{$Z_+$ change of sign}, and similarly for $Z_-$.

Recall that for every pair of points $\nPoint_i,\nPoint_j$ in $\Set{\nPoint_1,\ldots,\nPoint_\nRank}$, there is an arc~$\mu_{ij}$ connecting those points that passes beneath other points in $\Set{\nPoint_1,\ldots,\nPoint_\nRank}$.
\begin{corollary}
	\label{cor:single_crossing_of_line_in_alternating_subword}
	Let $q \in Q$ and let $u$ be a maximal $(\nFreeGenElt_i,\nFreeGenElt_j)$-alternating subword of $q$ with $\ell(u) \geq 3$.
	Let $\gamma$ be a minimal simple loop representing $q$.
	The arc $\gamma_u$ crosses $\mu_{ij}$ exactly once, corresponding to where $u$ changes sign.
\end{corollary}
\begin{proof}
	\Cref{lem:crossing_diagonals_and_one_change_of_sign} tells us there is exactly one change of sign in $u$.
	Recall that $\gamma$ crosses $\mu_{ij}$ minimally and thus forms no bigons with $\mu_{ij}$.
	It is simple to see from a picture that $\gamma_u$ does not cross $\mu_{ij}$ except at the change of sign, where it must cross $\mu_{ij}$.
\end{proof}
\begin{definition}
	\label{def:+-_factorisation}
	Let $q \in Q$ and let $\gamma$ be a minimal simple loop representing $q$.
	Let $u$ be a maximal $(\nFreeGenElt_i,\nFreeGenElt_j)$-alternating subword of $q$ with $\ell(u) \geq 3$.
	We factor $u$ into two parts, splitting at the unique change of sign in $u$, which corresponds to the unique crossing of $\mu_{ij}$ in $\gamma_u$.
	The factor corresponding to the sub-arc of $\gamma_u$ which emerges above $\mu_{ij}$ is denoted $u_+$, and the other factor, which emerges below $\mu_{ij}$, is denoted~$u_-$.
\end{definition}
\begin{remark}
	In the above definition, we do not know if $u = u_+u_-$ or if $u=u_-u_+$, but one of these factorisations is valid.
	This depends on the direction of $\gamma$ at the relevant crossing of $\mu_{ij}$, as demonstrated in the following pictures.
	\begin{equation*}
	\begin{tikzpicture}[,every node/.append style={scale=1}, inner sep=0pt, outer sep=0pt]
		\input{figs/change_of_sign_demo_01.tex}
	\end{tikzpicture}%

		\quad
	\begin{tikzpicture}[,every node/.append style={scale=1}, inner sep=0pt, outer sep=0pt]
		\input{figs/change_of_sign_demo_02.tex}
	\end{tikzpicture}%

		\quad
	\begin{tikzpicture}[,every node/.append style={scale=1}, inner sep=0pt, outer sep=0pt]
		\input{figs/change_of_sign_demo_03.tex}
	\end{tikzpicture}%

		\quad
	\begin{tikzpicture}[,every node/.append style={scale=1}, inner sep=0pt, outer sep=0pt]
		\input{figs/change_of_sign_demo_04.tex}
	\end{tikzpicture}%

	\end{equation*}
\end{remark}
The following lemma restricts \emph{where} in a maximal alternating subword (of some $q \in Q$) a change of sign can occur.
The change of sign occurs at the halfway point of the maximal subword if the maximal subword has even length, and approximately at the halfway point if the length is odd.
\begin{lemma}
	\label{lem:change_sign_close_to_middle}
	Let $q \in Q$ with maximal $(\nFreeGenElt_i,\nFreeGenElt_j)$-alternating subword $u$ satisfying $\ell(u) \geq 3$.
	\begin{enumerate}
		\item If $\ell(u)= 2n$ is even, then $u$ changes sign after the $n$th letter.
		      \smallskip
		\item If $\ell(u)= 2n + 1$ is odd, then  $u$ changes sign after the $n$th or $(n+1)$st letter.
	\end{enumerate}
\end{lemma}
\begin{proof}
	Let $\gamma$ be a minimal simple loop representing $q$.
	\Cref{cor:single_crossing_of_line_in_alternating_subword} tells us there is a unique place where $u$ changes sign and $\gamma_u$ crosses  $\mu_{ij}$.
	Let $u$ be factored as $u = u_+ u_-$ or $u =u_-u_+$ with respect to this crossing, as in \Cref{def:+-_factorisation}.
	The lemma we are trying to prove is equivalent to $\Abs{\ell(u_+) - \ell(u_-)} \leq 1$.

	Assume the change of sign in $u$ is of $Z_+$ type.
	Below, we draw $\gamma_u$ in the case where $\ell(u_+) = 4$, and $\ell(u_-)$ ranges through $\Set{3,4,5}$.
	We draw $\gamma_{u_+}$ in blue, $\gamma_{u_-}$ in red and $\mu_{ij}$ as the horizontal, dashed line.
	\begin{equation}
		\label{eqn:triple_close_to_middle_diagram}
	\begin{tikzpicture}[baseline={($(baseline_text.center)+(0,0.85cm)$)},every node/.append style={scale=1}, inner sep=0pt, outer sep=0pt]
		\input{figs/change_of_sign_close_to_middle_3_crossing.tex}
	\end{tikzpicture}%

		\qquad
	\begin{tikzpicture}[baseline={($(baseline_text.center)+(0,0.85cm)$)},every node/.append style={scale=1}, inner sep=0pt, outer sep=0pt]
		\input{figs/change_of_sign_close_to_middle_4_crossing.tex}
	\end{tikzpicture}%

		\qquad
	\begin{tikzpicture}[baseline={($(baseline_text.center)+(0,0.85cm)$)},every node/.append style={scale=1}, inner sep=0pt, outer sep=0pt]
		\input{figs/change_of_sign_close_to_middle_5_crossing.tex}
	\end{tikzpicture}%

	\end{equation}
	We see that if we were to fix $\ell(u_+) = 4$, then $\ell(u_-)$ can only take values in $\Set{3,4,5}$.
	The case $\ell(u_-) \leq 2$ is not possible, since otherwise an end of $\gamma_{u_-}$ would be trapped in a region bounded by $\lambda_j$ and $\gamma_{u_+}$.
	The case $\ell(u_-) \geq 6$ is not possible, since otherwise an end of $\gamma_{u_+}$ would be trapped in a region bounded by $\lambda_i$ and $\gamma_{u_-}$.
	We draw the cases where $\ell(u_-) = 2$ and $\ell(u_-) = 6$ below to demonstrate this.
	\begin{equation}
		\label{eqn:double_close_to_middle_diagram}
	\begin{tikzpicture}[baseline={($(baseline_text.center)+(0,0.9cm)$)},every node/.append style={scale=1}, inner sep=0pt, outer sep=0pt]
		\input{figs/change_of_sign_close_to_middle_2_crossing.tex}
	\end{tikzpicture}%

		\qquad
	\begin{tikzpicture}[baseline={($(baseline_text.center)+(0,0.9cm)$)},every node/.append style={scale=1}, inner sep=0pt, outer sep=0pt]
		\input{figs/change_of_sign_close_to_middle_6_crossing.tex}
	\end{tikzpicture}%

	\end{equation}
	The only part of \eqref{eqn:triple_close_to_middle_diagram} and \eqref{eqn:double_close_to_middle_diagram} which is relevant to this argument is the outer part, where the highest crossings of $\lambda_i$ and $\lambda_j$ happen.
	Changing the value of $\ell(u_+)$ only adds more inner crossings, so the same argument as above can be made for any value of $\ell(u_+)$.
	The same argument can be made regardless of change of sign type ($Z_+$ or $Z_-$) in $u$.
\end{proof}
\begin{lemma}
	\label{lem:possible_lengths_other_crossings}
	Let $q \in Q$ and let $\gamma$ be a minimal simple loop representing $q$.
	Let~$u$ be a maximal~$(\nFreeGenElt_i,\nFreeGenElt_j)$-alternating subword with $\ell(u) \geq 4$.
	All points where~$\gamma$ crosses~$\mu_{ij}$ occur at the change of sign of some maximal $(\nFreeGenElt_i,\nFreeGenElt_j)$-alternating subword~$v$ with $\ell(v) \geq 3$.
	Furthermore, for any such~$v$, we have $\Abs{\ell(u) - \ell(v)} \leq 2$ and if either of~$\ell(u)$ or~$\ell(v)$ is even, then $\Abs{\ell(u) - \ell(v)} \leq 1$.
\end{lemma}
\begin{proof}
	\Cref{cor:single_crossing_of_line_in_alternating_subword} tells us that $\gamma_u$ crosses $\mu_{ij}$ exactly once, at the change of sign in  $u$.
	In \Cref{fig:u_and_other_crossing_minimal}, we draw $\gamma_u$ where $\ell(u) = 4$, and $u$ has a $Z_+$ change of sign.
	There, we also draw another sub-arc of~$\gamma$ which crosses~$\mu_{ij} $ to the right of where~$\gamma_u$ crosses~$\mu_{ij}$.
	We see that the crossing of $\mu_{ij}$ not in $\gamma_u$ must be part of~$\gamma_v$, where~$v$ is either the word $\nFreeGenElt_i\nFreeGenElt_j\nFreeGenElt_i^{-1}$ or its inverse.
	This is because the two ends of the other crossing are inside a region bounded by $\gamma_u$ and a part of $\lambda_i$.
	The only way the ends can escape this region without forming bigons is by making a path which corresponds to one of those words.
	We can make a similar argument if $u$ has a $Z_-$ change of sign, if the other crossing of $\mu_{ij}$ is to the left of the $\gamma_u$ crossing, and for any $\ell(u) \geq 4$.
	We have proved the first claim.

	According to \Cref{def:+-_factorisation}, we factorise $u$ and $v$ with respect to the crossings of~$\mu_{ij}$.
	We draw $\gamma_u$ and $\gamma_v$ with colours corresponding to this factorisation in \Cref{fig:u_and_other_crossing_extend_and_colour}.
	\begin{figure}[t]
		\centering
		\begin{subfigure}[t]{.49\textwidth}
			\centering
			\begin{minipage}[b][2.5cm][b]{\linewidth}
				\centering
	\begin{tikzpicture}[,every node/.append style={scale=1}, inner sep=0pt, outer sep=0pt]
		\input{figs/4_crossing_and_other_mu_ij_01.tex}
	\end{tikzpicture}%

			\end{minipage}
			\caption{The other crossing is trapped in a region.}
			\label{fig:u_and_other_crossing_minimal}
		\end{subfigure}%
		\hfill
		\begin{subfigure}[t]{.49\textwidth}
			\centering
			\begin{minipage}[b][2.5cm][b]{\linewidth}
				\centering
	\begin{tikzpicture}[,every node/.append style={scale=1}, inner sep=0pt, outer sep=0pt]
		\input{figs/change_of_sign_u_and_v.tex}
	\end{tikzpicture}%

			\end{minipage}
			\caption{Extend the other crossing and colour the sub-arcs by the factorisation in \Cref{def:+-_factorisation}.}
			\label{fig:u_and_other_crossing_extend_and_colour}
		\end{subfigure}
		\caption{A possible drawing of $\gamma_u$ where $\ell(u) = 4$.
			Another crossing of $\mu_{ij}$ is drawn using a dashed line.
		}
		\label{fig:uv_subarcs}
	\end{figure}
	Statement (1) of \Cref{lem:crossing_diagonals_and_one_change_of_sign} tells us that the change of sign types ($Z_+$ or~$Z_-$) in $u$ and $v$ must agree.

	By a very similar argument to that in the proof of \Cref{lem:change_sign_close_to_middle}, we have the following inequalities for $\ell(u)$ and $\ell(v)$.
	\begin{align*}
		\ell(u_+) - 2 \leq \ell(v_+) \leq \ell(u_+) \\
		\ell(v_-) -2 \leq \ell(u_-) \leq \ell(v_-)
	\end{align*}
	\Cref{lem:change_sign_close_to_middle} tells us $\Abs{\ell(u_+) - \ell(u_-)} \leq 1$ and $\Abs{\ell(v_+) - \ell(v_-)} \leq 1$.
	We combine these into a system of inequalities.
	\begin{flalign}
		\begin{aligned}
			\Abs{\ell(u_+) - \ell(u_-)}  & \leq 1         \\
			\Abs{\ell(v_+) - \ell(v_-)}  & \leq 1         \\
			\ell(u_+) - 2 \leq \ell(v_+) & \leq \ell(u_+) \\
			\ell(v_-) -2 \leq \ell(u_-)  & \leq \ell(v_-)
		\end{aligned}
		\label{eqn:uv_+-_system_of_inequalities}
	\end{flalign}
	All solutions to \eqref{eqn:uv_+-_system_of_inequalities} are shown in \Cref{tab:solutions_to_inequalities_01} and equivalently in \Cref{tab:solutions_to_inequalities_02}.

	We made two assumptions before drawing \Cref{fig:u_and_other_crossing_extend_and_colour} that, had we made a different choice, would have affected the system of inequalities \eqref{eqn:uv_+-_system_of_inequalities}.
	\begin{enumerate}
		\item That the crossing of $\gamma_u$ with $\mu_{ij}$ is to the left of where $\gamma_v$ crosses $\mu_{ij}$.
		      \smallskip
		\item That the changes of sign in $u$ and  $v$ are of  $Z_+$ type.
	\end{enumerate}
	If we were to make the opposite choice for either of these, then the resulting system of inequalities would correspond to swapping the variables $\ell(u_\pm) \leftrightarrow \ell(v_\pm)$ in \eqref{eqn:uv_+-_system_of_inequalities}.
	Reading \Cref{tab:solutions_to_inequalities_01,tab:solutions_to_inequalities_02} with this ambiguity completes the proof.
	\begin{table}[ht]
		\centering
		\begin{subtable}{\textwidth}
			\centering
			\begin{tabular}{l|l}
				\multicolumn{1}{c|}{$\ell(u_-) = $} &
				\multicolumn{1}{c}{$(\ell(v_+),\ell(v_-)) \in $}                                                   \\
				\hline
				$n-1$                               & $\Set{(n-2, n-1), (n-1,n-1), (n-1,n),(n,n-1),(n,n),(n,n+1)}$ \\
				$n$                                 & $\Set{(n-1, n), (n,n), (n,n+1)}$                             \\
				$n+1$                               & $\Set{(n,n+1)}$                                              \\
			\end{tabular}
			\caption{Solutions grouped by $\ell(u_-)$, with $n \coloneq \ell(u_+)$.}
			\label{tab:solutions_to_inequalities_01}
		\end{subtable}

		\medskip

		\begin{subtable}{\textwidth}
			\centering
			\begin{tabular}{l|l}
				\multicolumn{1}{c|}{$\ell(v_-) = $} &
				\multicolumn{1}{c}{$(\ell(u_+),\ell(u_-)) \in $}                                                   \\
				\hline
				$n-1$                               & $\Set{(n,n-1)}$                                              \\
				$n$                                 & $\Set{(n+1, n), (n,n), (n,n-1)}$                             \\
				$n+1$                               & $\Set{(n+2, n+1), (n+1,n+1), (n+1,n),(n,n+1),(n,n),(n,n-1)}$ \\
			\end{tabular}
			\caption{Solutions grouped by $\ell(v_-)$, with $n \coloneq \ell(v_+)$.}
			\label{tab:solutions_to_inequalities_02}
		\end{subtable}

		\caption{Possible solutions to the system of inequalities
			\eqref{eqn:uv_+-_system_of_inequalities},
			expressed in terms of the lengths of $u_\pm$ and $v_\pm$.}
		\label{tab:solutions_to_inequalities}
	\end{table}
\end{proof}
The lemma above tells us that two maximal $(\nFreeGenElt_i,\nFreeGenElt_j)$-alternating subwords of some $q \in Q$ can only differ in length by~$2$ if both have odd length.
The following lemma further restricts such pairs of subwords.
\begin{lemma}
	\label{lem:phase_coherent}
	Let $q \in Q$.
	Let $u$ and $v$ be two maximal $(\nFreeGenElt_i,\nFreeGenElt_j)$-alternating subwords with $\ell(u) = \ell(v) + 2 \geq 5$ and odd.
	These subwords start with the same generator.
\end{lemma}
\begin{proof}
	For example, the first letters of $u$ and $v$ may be $\nFreeGenElt_i$ and $\nFreeGenElt_i^{-1}$, respectively.
	Only their underlying generators must agree.

	Let $\gamma$ be a minimal simple loop representing $q$.
	Let the first letter of $u$ be one of~$\Set{\nFreeGenElt_i^{\pm 1}}$, and consider the region $R$, bounded by $\gamma_u$ and the segment of $\lambda_i$ between the first and last crossings of $\lambda_i$, as highlighted below in grey for $u = \nFreeGenElt_i\nFreeGenElt_j\nFreeGenElt_i\nFreeGenElt_j^{-1}\nFreeGenElt_i^{-1}$.
	\begin{equation*}
	\begin{tikzpicture}[baseline=(baseline_text.center),every node/.append style={scale=1}, inner sep=0pt, outer sep=0pt]
		\input{figs/alt_length_5.tex}
	\end{tikzpicture}%

	\end{equation*}
	Let $2k + 1 = \ell(u)$.
	If~$k$ is even, then~$\gamma_u$ crosses~$\lambda_i$ an odd number of times and~$R$ must contain the point $\nPoint_i$.
	If~$k$ is odd, then~$R$ contains $\nPoint_j$.
	We similarly define the region $R^\prime$ for $v$, for which equivalent results hold.

	Since $\ell(v) = 2(k-1) + 1$, if~$v$ starts with a different generator from~$u$, then both~$R$ and~$R^\prime$ contain the same point, and thus intersect.
	However, these regions have disjoint boundaries, so they must be nested, which leads to a contradiction.
\end{proof}

%% file: figs/all_arcs.tex
\definecolor{c819d43}{RGB}{129,157,67}
\definecolor{ca2a2a2}{RGB}{162,162,162}
\definecolor{c979797}{RGB}{151,151,151}

  \path[draw=black,fill=black,line width=0.02cm,shift={(0.0049, -0.0001)}] (8.1993, 17.3074) -- (8.7657, 17.3068);

  \path[draw=black,fill=black,line width=0.02cm] (10.4536, 17.3073) -- (9.8908, 17.3068);

  \path[draw=black,line width=0.02cm] (8.2042, 17.3073).. controls (8.2224, 17.1836) and (8.2827, 17.0665) .. (8.3729, 16.9799).. controls (8.4631, 16.8933) and (8.5825, 16.8377) .. (8.7068, 16.8246).. controls (8.8511, 16.8093) and (9.0006, 16.8518) .. (9.1152, 16.9407).. controls (9.2298, 17.0296) and (9.3082, 17.1638) .. (9.3293, 17.3073);

  \path[draw=black,line width=0.02cm,shift={(0.5724, 0.0052)}] (8.1934, 17.3016).. controls (8.2138, 17.1695) and (8.2825, 17.0455) .. (8.3836, 16.958).. controls (8.4847, 16.8706) and (8.6174, 16.8205) .. (8.751, 16.8194).. controls (8.8863, 16.8182) and (9.0215, 16.8672) .. (9.1246, 16.9548).. controls (9.2277, 17.0424) and (9.2978, 17.1679) .. (9.3185, 17.3016);

  \path[draw=c819d43,line width=0.02cm] (10.4534, 17.3068) -- (9.3255, 18.8354);

  \path[draw=black,line width=0.02cm,shift={(1.1417, -0.0)}] (8.1878, 17.3104).. controls (8.2077, 17.174) and (8.2793, 17.0458) .. (8.3849, 16.9572).. controls (8.4905, 16.8686) and (8.6293, 16.8205) .. (8.7671, 16.8246).. controls (8.8987, 16.8285) and (9.0284, 16.8799) .. (9.1269, 16.9672).. controls (9.2255, 17.0546) and (9.2922, 17.1771) .. (9.3119, 17.3073);

  \path[draw=c819d43,line width=0.02cm] (8.2032, 17.3068) -- (9.3231, 18.8389);

  \path[draw=black,line width=0.02cm] (8.2042, 17.3073).. controls (8.2202, 17.1049) and (8.312, 16.9093) .. (8.4576, 16.7677).. controls (8.6031, 16.626) and (8.8011, 16.5396) .. (9.0039, 16.5291).. controls (9.2212, 16.5179) and (9.4417, 16.5945) .. (9.6053, 16.7379).. controls (9.7689, 16.8813) and (9.8736, 17.09) .. (9.8908, 17.3068);

  \path[draw=c819d43,line width=0.02cm] (8.7657, 17.3068) -- (9.3255, 18.8354);

  \path[draw=black,line width=0.02cm,shift={(0.5833, -0.0119)}] (8.1873, 17.3186).. controls (8.2033, 17.0903) and (8.317, 16.8704) .. (8.4941, 16.7253).. controls (8.6711, 16.5802) and (8.909, 16.5119) .. (9.1361, 16.541).. controls (9.3277, 16.5656) and (9.5098, 16.6583) .. (9.6424, 16.7988).. controls (9.7749, 16.9393) and (9.8569, 17.1265) .. (9.8703, 17.3192);

  \path[draw=c819d43,line width=0.02cm] (9.3283, 17.3068) -- (9.3255, 18.8354);

  \path[draw=black,line width=0.02cm] (8.2042, 17.3073).. controls (8.1939, 17.007) and (8.31, 16.7044) .. (8.5185, 16.488).. controls (8.727, 16.2716) and (9.0251, 16.1444) .. (9.3256, 16.1435).. controls (9.6271, 16.1426) and (9.9269, 16.2691) .. (10.1367, 16.4855).. controls (10.3465, 16.702) and (10.4636, 17.0055) .. (10.4534, 17.3068);

  \path[draw=black,fill=black,line width=0.02cm] (8.765, 17.3094) -- (9.3295, 17.3104) -- (8.765, 17.3094).. controls (8.9532, 17.3105) and (9.1413, 17.3109) .. (9.3295, 17.3104).. controls (9.5166, 17.31) and (9.7037, 17.3088) .. (9.8908, 17.3068);

  \path[draw=c819d43,line width=0.02cm] (9.8908, 17.3068) -- (9.3255, 18.8354);

  \path[draw=ca2a2a2,line width=0.02cm,dash pattern=on 0.06cm off 0.02cm,cm={ 0.0,1.0,1.0,0.0,(-29.7, 29.7)}] (-12.3901, 39.0255) circle (1.5255cm);

  \path[draw=black,fill=c979797,line width=0.0487cm,cm={ 0.0,1.0,1.0,0.0,(-29.7, 29.7)}] (-12.3932, 40.1534) circle (0.0528cm);

  \path[draw=black,fill=c979797,line width=0.0487cm,cm={ 0.0,1.0,1.0,0.0,(-29.7, 29.7)}] (-12.3932, 37.9032) circle (0.0528cm);

  \path[draw=black,fill=c979797,line width=0.0487cm,cm={ 0.0,1.0,1.0,0.0,(-29.7, 29.7)}] (-12.3932, 38.4657) circle (0.0528cm);

  \path[draw=black,fill=c979797,line width=0.0487cm,cm={ 0.0,1.0,1.0,0.0,(-29.7, 29.7)}] (-12.3932, 39.0283) circle (0.0528cm);

  \path[draw=black,fill=c979797,line width=0.0487cm,cm={ 0.0,1.0,1.0,0.0,(-29.7, 29.7)}] (-12.3932, 39.5908) circle (0.0528cm);

  \node[anchor=south west,line width=0.02cm,dash pattern=on 0.02cm off 0.02cm] (baseline_text) at (8.2032, 17.3068){};

%% file: figs/square_impossible.tex
\definecolor{ce6e6e6}{RGB}{230,230,230}
\definecolor{c819d43}{RGB}{129,157,67}
\definecolor{c979797}{RGB}{151,151,151}

  \path[fill=ce6e6e6,line width=0.02cm] (4.2325, 23.2062).. controls (4.0308, 23.2318) and (3.8672, 23.3274) .. (3.7662, 23.479).. controls (3.7299, 23.5334) and (3.7007, 23.6036) .. (3.6854, 23.6728).. controls (3.6755, 23.7182) and (3.6746, 23.8269) .. (3.6838, 23.8691).. controls (3.7175, 24.0224) and (3.8112, 24.1344) .. (3.9391, 24.1746).. controls (3.9755, 24.186) and (4.0087, 24.1901) .. (4.0645, 24.1901).. controls (4.1158, 24.1901) and (4.1159, 24.19) .. (4.1159, 24.1712).. controls (4.1159, 24.1613) and (4.1183, 24.1619) .. (4.1441, 24.1784) -- (4.1624, 24.1901) -- (4.2119, 24.1901) -- (4.2614, 24.1901) -- (4.2614, 23.9789) -- (4.2614, 23.7678) -- (4.2484, 23.7638).. controls (4.1999, 23.7493) and (4.1845, 23.6798) .. (4.2223, 23.6456).. controls (4.2497, 23.6206) and (4.293, 23.6206) .. (4.3205, 23.6456).. controls (4.3577, 23.6794) and (4.3437, 23.7477) .. (4.2963, 23.7634) -- (4.2846, 23.7672) -- (4.2846, 23.9786) -- (4.2846, 24.1901) -- (4.3458, 24.1901) -- (4.407, 24.1901) -- (4.407, 24.2016) -- (4.407, 24.2132) -- (4.3458, 24.2132) -- (4.2846, 24.2132) -- (4.2846, 24.4398) -- (4.2846, 24.6663) -- (4.3491, 24.6663) -- (4.4136, 24.6663) -- (4.4136, 24.6531).. controls (4.4136, 24.6458) and (4.4145, 24.6398) .. (4.4156, 24.6398).. controls (4.4168, 24.6398) and (4.4262, 24.645) .. (4.4366, 24.6513) -- (4.4555, 24.6628) -- (4.4883, 24.6579).. controls (4.7897, 24.6129) and (5.0248, 24.3568) .. (5.0652, 24.0296).. controls (5.0703, 23.9885) and (5.0695, 23.8918) .. (5.0638, 23.8459).. controls (5.0257, 23.5419) and (4.8159, 23.3006) .. (4.5248, 23.226).. controls (4.5009, 23.2199) and (4.4658, 23.2124) .. (4.4467, 23.2094).. controls (4.3984, 23.2018) and (4.2813, 23.2) .. (4.2325, 23.2062) -- cycle;

  \path[draw=c819d43,line width=0.02cm] (4.2726, 23.6972) -- (4.2726, 25.1969);

  \node[text=black,line width=0.005cm,xscale=1.0003,yscale=0.9997,anchor=south west] (text75238) at (3.9145, 25.3178){$\scriptscriptstyle\lambda_i$};

  \path[draw=black,line width=0.0185cm,cm={ 1.0367,-0.0,-0.0,1.1231,(-0.3675, -1.95)}] (4.6029, 23.2852) -- (4.258, 23.2852).. controls (4.1833, 23.2874) and (4.1081, 23.2649) .. (4.0469, 23.222).. controls (3.9625, 23.1628) and (3.9071, 23.0658) .. (3.8928, 22.9637).. controls (3.8784, 22.8616) and (3.9038, 22.7554) .. (3.9574, 22.6673).. controls (4.0109, 22.5791) and (4.0916, 22.5088) .. (4.1834, 22.4618).. controls (4.2752, 22.4148) and (4.3778, 22.3906) .. (4.4808, 22.3848).. controls (4.6089, 22.3777) and (4.7391, 22.3991) .. (4.8561, 22.4516).. controls (4.9732, 22.5041) and (5.0767, 22.5882) .. (5.1479, 22.6949).. controls (5.2373, 22.8288) and (5.2734, 22.9971) .. (5.2477, 23.156).. controls (5.222, 23.3149) and (5.135, 23.4629) .. (5.0093, 23.5635).. controls (4.8949, 23.655) and (4.7494, 23.7069) .. (4.6029, 23.7085) -- (4.258, 23.7085);

  \node[text=black,line width=0.005cm,xscale=1.0003,yscale=0.9997,anchor=south west] (text6) at (4.9822, 24.5433){$\scriptscriptstyle\gamma_u$};

  \node[text=black,line width=0.005cm,xscale=1.0003,yscale=0.9997,anchor=south west] (text7) at (4.5316, 23.4405){$\scriptscriptstyle R$};

  \path[draw=black,fill=black,line width=0.015cm,cm={ -0.3791,0.4627,-0.4629,-0.3789,(15.3076, 27.2897)}] (8.1902, 16.8198) -- (8.0991, 16.7825) -- (8.1124, 16.8801) -- cycle;

  \path[draw=black,fill=black,line width=0.015cm,cm={ -0.3791,0.4627,-0.4629,-0.3789,(15.0097, 26.8132)}] (8.1902, 16.8198) -- (8.0991, 16.7825) -- (8.1124, 16.8801) -- cycle;

  \path[draw=black,fill=c979797,line width=0.0487cm,cm={ 0.0,1.0,1.0,0.0,(-29.7, 29.7)}] (-6.0028, 33.9714) circle (0.0458cm);

%% file: figs/iji_turn_02.tex
\definecolor{c819d43}{RGB}{129,157,67}
\definecolor{ce6e6e6}{RGB}{230,230,230}
\definecolor{c979797}{RGB}{151,151,151}

  \path[draw=c819d43,line width=0.02cm] (8.4228, 13.9723) -- (8.4228, 13.7144);

  \path[draw=c819d43,line width=0.02cm,dash pattern=on 0.02cm off 0.04cm] (8.4228, 14.4372) -- (8.4228, 13.9723);

  \path[fill=ce6e6e6,line width=0.02cm] (8.0431, 13.4348).. controls (8.0153, 13.4714) and (8.0391, 13.4189) .. (7.9974, 13.5649).. controls (7.9432, 13.7544) and (8.022, 13.856) .. (8.1676, 13.9283).. controls (8.2101, 13.9493) and (8.253, 13.9576) .. (8.3359, 13.9605) -- (8.4096, 13.9631) -- (8.4107, 13.8389) -- (8.4118, 13.7147) -- (8.4232, 13.7134) -- (8.4346, 13.7121) -- (8.4346, 13.837) -- (8.4346, 13.9618) -- (8.6028, 13.9618) -- (8.771, 13.9618) -- (8.771, 13.9722) -- (8.771, 13.9826) -- (8.6028, 13.9826) -- (8.4346, 13.9826) -- (8.4346, 14.071) -- (8.4346, 14.4357) -- (9.3275, 14.43).. controls (10.6351, 14.4217) and (10.7194, 14.4282) .. (10.7482, 14.4217).. controls (10.9185, 14.383) and (11.0414, 14.2833) .. (11.0871, 14.1231).. controls (11.1107, 14.0401) and (11.1152, 14.1045) .. (11.1152, 13.879).. controls (11.1152, 13.7037) and (11.1039, 13.7839) .. (11.098, 13.7489).. controls (11.0699, 13.5815) and (11.0033, 13.4099) .. (10.8753, 13.3454).. controls (10.8399, 13.3275) and (10.7822, 13.2778) .. (10.7459, 13.3147) -- (10.7023, 13.306) -- (9.2343, 13.3053).. controls (8.4269, 13.3049) and (8.6602, 13.2792) .. (8.3735, 13.2914).. controls (8.3247, 13.2935) and (8.2735, 13.2892) .. (8.2277, 13.3062).. controls (8.1574, 13.3323) and (8.0885, 13.3751) .. (8.0431, 13.4348) -- cycle;

  \node[text=black,line width=0.005cm,anchor=south west] (text27967-3) at (10.3428, 13.4045){$\scriptstyle \nPoint_k$};

  \path[draw=c819d43,line width=0.02cm] (10.3735, 14.8134) -- (10.3735, 13.7144);

  \node[text=black,line width=0.005cm,anchor=south west] (text17-9-6) at (10.1177, 14.9263){$\scriptscriptstyle\lambda_{k}$};

  \path[draw=black,fill=c979797,line width=0.0487cm,cm={ 0.0,1.0,1.0,0.0,(-29.7, 29.7)}] (-15.9856, 40.0735) circle (0.0458cm);

  \path[draw=c819d43,line width=0.02cm] (8.4228, 14.8134) -- (8.4228, 14.4372);

  \node[text=black,line width=0.005cm,anchor=south west] (text17-8) at (8.167, 14.9263){$\scriptscriptstyle\lambda_{i}$};

  \path[draw=black,fill=c979797,line width=0.0487cm,cm={ 0.0,1.0,1.0,0.0,(-29.7, 29.7)}] (-15.9856, 38.1228) circle (0.0458cm);

  \path[fill=black,line width=0.0185cm,shift={(3.6072, -2.1531)}] (4.7681, 15.499) -- (6.9651, 15.4993);

  \path[draw=black,line width=0.02cm,shift={(3.6072, -2.1531)}] (4.5124, 16.5903) -- (6.9651, 16.5903).. controls (7.0443, 16.5903) and (7.1246, 16.5885) .. (7.2005, 16.566).. controls (7.2764, 16.5435) and (7.3465, 16.5003) .. (7.3983, 16.4404).. controls (7.4501, 16.3805) and (7.4829, 16.3056) .. (7.4988, 16.228).. controls (7.5147, 16.1504) and (7.5146, 16.0705) .. (7.5111, 15.9914).. controls (7.5078, 15.9173) and (7.5014, 15.8429) .. (7.481, 15.7716).. controls (7.4607, 15.7003) and (7.4257, 15.6325) .. (7.3746, 15.5788).. controls (7.3236, 15.5251) and (7.257, 15.4868) .. (7.1857, 15.4669).. controls (7.1143, 15.4469) and (7.0392, 15.4454) .. (6.9651, 15.4454) -- (4.7459, 15.4454).. controls (4.7106, 15.4454) and (4.6752, 15.4457) .. (4.6404, 15.4512).. controls (4.6056, 15.4568) and (4.5714, 15.4675) .. (4.5404, 15.4843).. controls (4.4785, 15.518) and (4.4319, 15.5763) .. (4.4057, 15.6418).. controls (4.3692, 15.7329) and (4.3692, 15.8379) .. (4.4057, 15.9289).. controls (4.4319, 15.9944) and (4.4785, 16.0527) .. (4.5404, 16.0864).. controls (4.5714, 16.1033) and (4.6056, 16.114) .. (4.6404, 16.1195).. controls (4.6752, 16.125) and (4.7106, 16.1253) .. (4.7459, 16.1253) -- (5.1602, 16.1253);

  \node[text=black,line width=0.005cm,anchor=south west] (text57948-9-3-0-2) at (9.3526, 13.427){$\scriptscriptstyle R$};

  \node[text=black,line width=0.005cm,anchor=south west] (text1-7) at (8.295, 13.4045){$\scriptstyle \nPoint_i$};

  \node[text=black,line width=0.005cm,anchor=south west] (text3) at (8.1587, 14.1197){$\scriptstyle \rho$};

%% file: figs/crossing_diagonals.tex
\definecolor{c819d43}{RGB}{129,157,67}
\definecolor{c979797}{RGB}{151,151,151}

  \path[draw=c819d43,line width=0.02cm] (10.8047, 11.0997) -- (10.8047, 10.2123);

  \path[draw=black,fill=c979797,line width=0.0487cm] (10.8047, 10.2123) circle (0.0458cm);

  \node[text=black,line width=0.005cm,anchor=south west] (text57948-7) at (10.496, 11.2126){$\scriptscriptstyle\lambda_{j}$};

  \node[text=black,line width=0.005cm,anchor=south west] (text21055) at (10.4343, 10.1503){$\scriptstyle \nPoint_j$};

  \path[draw=c819d43,line width=0.02cm] (8.688, 11.0997) -- (8.688, 10.2123);

  \path[draw=black,fill=c979797,line width=0.0487cm] (8.688, 10.2123) circle (0.0458cm);

  \node[text=black,line width=0.005cm,anchor=south west] (text17) at (8.3793, 11.2126){$\scriptscriptstyle\lambda_{i}$};

  \node[text=black,line width=0.005cm,anchor=south west] (text18) at (8.8538, 10.1503){$\scriptstyle \nPoint_i$};

  \path[draw=black,line width=0.02cm,shift={(3.603, -5.5829)}] (4.8299, 16.3811) -- (5.4277, 16.3811).. controls (5.5628, 16.3766) and (5.6967, 16.3397) .. (5.815, 16.2743).. controls (5.951, 16.199) and (6.0638, 16.0884) .. (6.1713, 15.976).. controls (6.2788, 15.8636) and (6.3837, 15.747) .. (6.5096, 15.6557).. controls (6.6813, 15.5311) and (6.8894, 15.4573) .. (7.1012, 15.4461) -- (7.2301, 15.4462).. controls (7.2747, 15.4462) and (7.3197, 15.447) .. (7.3631, 15.4572).. controls (7.4064, 15.4675) and (7.4478, 15.4872) .. (7.4816, 15.5162).. controls (7.5154, 15.5453) and (7.5414, 15.583) .. (7.5584, 15.6241).. controls (7.5755, 15.6653) and (7.5839, 15.7096) .. (7.5876, 15.754).. controls (7.5917, 15.8034) and (7.59, 15.8536) .. (7.5772, 15.9014).. controls (7.5644, 15.9492) and (7.5401, 15.9942) .. (7.5051, 16.0292).. controls (7.4702, 16.0642) and (7.4252, 16.0886) .. (7.3773, 16.1011).. controls (7.3294, 16.1137) and (7.2794, 16.1145) .. (7.2299, 16.1144) -- (6.9466, 16.114);

  \path[draw=black,line width=0.02cm,dash pattern=on 0.02cm off 0.04cm,cm={ -1.0,-0.0,-0.0,1.0,(15.9796, -5.5829)}] (4.8299, 16.3811) -- (5.4277, 16.3811).. controls (5.5628, 16.3766) and (5.6967, 16.3397) .. (5.815, 16.2743).. controls (5.951, 16.199) and (6.0638, 16.0884) .. (6.1713, 15.976).. controls (6.2788, 15.8636) and (6.3837, 15.747) .. (6.5096, 15.6557).. controls (6.6813, 15.5311) and (6.8894, 15.4573) .. (7.1012, 15.4461) -- (7.2301, 15.4462).. controls (7.2747, 15.4462) and (7.3197, 15.447) .. (7.3631, 15.4572).. controls (7.4064, 15.4675) and (7.4478, 15.4872) .. (7.4816, 15.5162).. controls (7.5154, 15.5453) and (7.5414, 15.583) .. (7.5584, 15.6241).. controls (7.5755, 15.6653) and (7.5839, 15.7096) .. (7.5876, 15.754).. controls (7.5917, 15.8034) and (7.59, 15.8536) .. (7.5772, 15.9014).. controls (7.5644, 15.9492) and (7.5401, 15.9942) .. (7.5051, 16.0292).. controls (7.4702, 16.0642) and (7.4252, 16.0886) .. (7.3773, 16.1011).. controls (7.3294, 16.1137) and (7.2794, 16.1145) .. (7.2299, 16.1144) -- (6.9466, 16.114);

%% file: figs/change_of_sign_demo_01.tex
\definecolor{c819d43}{RGB}{129,157,67}
\definecolor{c484537}{RGB}{72,69,55}
\definecolor{c979797}{RGB}{151,151,151}

  \path[draw=c819d43,line width=0.02cm] (8.688, 10.8814) -- (8.688, 10.2123);

  \path[draw=c819d43,line width=0.02cm] (10.4872, 10.8814) -- (10.4872, 10.2123);

  \path[draw=black,fill=c484537,line width=0.015cm,dash pattern=on 0.015cm off 0.06cm] (8.6811, 10.2112) -- (10.4872, 10.2123);

  \path[draw=black,fill=c979797,line width=0.0487cm] (8.688, 10.2123) circle (0.0458cm);

  \path[draw=black,fill=c979797,line width=0.0487cm] (10.4872, 10.2123) circle (0.0458cm);

  \path[draw=black,line width=0.02cm] (8.3793, 10.2112).. controls (8.3822, 10.2587) and (8.4, 10.3051) .. (8.4294, 10.3425).. controls (8.4595, 10.3807) and (8.5011, 10.409) .. (8.5461, 10.4276).. controls (8.5911, 10.4461) and (8.6395, 10.4554) .. (8.688, 10.4599).. controls (8.8359, 10.4735) and (8.985, 10.4433) .. (9.1251, 10.3941).. controls (9.2652, 10.3448) and (9.3977, 10.2768) .. (9.5312, 10.2118).. controls (9.681, 10.1389) and (9.8333, 10.0693) .. (9.9936, 10.0238).. controls (10.1538, 9.9783) and (10.3232, 9.9574) .. (10.4878, 9.9832).. controls (10.5379, 9.991) and (10.588, 10.0034) .. (10.6325, 10.0279).. controls (10.6769, 10.0524) and (10.7155, 10.09) .. (10.7336, 10.1374).. controls (10.7427, 10.1612) and (10.7464, 10.187) .. (10.7444, 10.2123);

  \node[text=black,anchor=south west,line width=0.02cm] (baseline_text) at (9.2078, 9.5499){$\scriptstyle{u=u_+u_-}$};

  \path[draw=black,fill=black,line width=0.0118cm,cm={ -0.6728,-0.3549,-0.3519,0.6804,(20.6941, 1.7605)}] (8.1902, 16.8198) -- (8.0991, 16.7825) -- (8.1124, 16.8801) -- cycle;

%% file: figs/change_of_sign_demo_02.tex
\definecolor{c819d43}{RGB}{129,157,67}
\definecolor{c484537}{RGB}{72,69,55}
\definecolor{c979797}{RGB}{151,151,151}

  \path[draw=c819d43,line width=0.02cm] (8.688, 10.8814) -- (8.688, 10.2123);

  \path[draw=c819d43,line width=0.02cm] (10.4872, 10.8814) -- (10.4872, 10.2123);

  \path[draw=black,fill=c484537,line width=0.015cm,dash pattern=on 0.015cm off 0.06cm] (8.6811, 10.2112) -- (10.4872, 10.2123);

  \path[draw=black,fill=c979797,line width=0.0487cm] (8.688, 10.2123) circle (0.0458cm);

  \path[draw=black,fill=c979797,line width=0.0487cm] (10.4872, 10.2123) circle (0.0458cm);

  \path[draw=black,line width=0.02cm] (8.3793, 10.2112).. controls (8.3822, 10.2587) and (8.4, 10.3051) .. (8.4294, 10.3425).. controls (8.4595, 10.3807) and (8.5011, 10.409) .. (8.5461, 10.4276).. controls (8.5911, 10.4461) and (8.6395, 10.4554) .. (8.688, 10.4599).. controls (8.8359, 10.4735) and (8.985, 10.4433) .. (9.1251, 10.3941).. controls (9.2652, 10.3448) and (9.3977, 10.2768) .. (9.5312, 10.2118).. controls (9.681, 10.1389) and (9.8333, 10.0693) .. (9.9936, 10.0238).. controls (10.1538, 9.9783) and (10.3232, 9.9574) .. (10.4878, 9.9832).. controls (10.5379, 9.991) and (10.588, 10.0034) .. (10.6325, 10.0279).. controls (10.6769, 10.0524) and (10.7155, 10.09) .. (10.7336, 10.1374).. controls (10.7427, 10.1612) and (10.7464, 10.187) .. (10.7444, 10.2123);

  \node[text=black,anchor=south west,line width=0.02cm] (baseline_text) at (9.2078, 9.5499){$\scriptstyle{u=u_-u_+}$};

  \path[draw=black,fill=black,line width=0.0118cm,cm={ 0.6449,0.4033,0.4008,-0.6528,(-2.1555, 17.7821)}] (8.1902, 16.8198) -- (8.0991, 16.7825) -- (8.1124, 16.8801) -- cycle;

%% file: figs/change_of_sign_demo_03.tex
\definecolor{c819d43}{RGB}{129,157,67}
\definecolor{c484537}{RGB}{72,69,55}
\definecolor{c979797}{RGB}{151,151,151}

  \path[draw=c819d43,line width=0.02cm] (8.688, 10.8814) -- (8.688, 10.2123);

  \path[draw=c819d43,line width=0.02cm] (10.4872, 10.8814) -- (10.4872, 10.2123);

  \path[draw=black,fill=c484537,line width=0.015cm,dash pattern=on 0.015cm off 0.06cm] (8.6811, 10.2112) -- (10.4872, 10.2123);

  \path[draw=black,fill=c979797,line width=0.0487cm] (8.688, 10.2123) circle (0.0458cm);

  \path[draw=black,fill=c979797,line width=0.0487cm] (10.4872, 10.2123) circle (0.0458cm);

  \node[text=black,anchor=south west,line width=0.02cm] (baseline_text) at (9.2078, 9.5499){$\scriptstyle{u=u_+u_-}$};

  \path[draw=black,line width=0.02cm,cm={ 1.0,-0.0,-0.0,-1.0,(0.0, 20.4351)}] (8.3793, 10.2112).. controls (8.3822, 10.2587) and (8.4, 10.3051) .. (8.4294, 10.3425).. controls (8.4595, 10.3807) and (8.5011, 10.409) .. (8.5461, 10.4276).. controls (8.5911, 10.4461) and (8.6395, 10.4554) .. (8.688, 10.4599).. controls (8.8359, 10.4735) and (8.985, 10.4433) .. (9.1251, 10.3941).. controls (9.2652, 10.3448) and (9.3977, 10.2768) .. (9.5312, 10.2118).. controls (9.681, 10.1389) and (9.8333, 10.0693) .. (9.9936, 10.0238).. controls (10.1538, 9.9783) and (10.3232, 9.9574) .. (10.4878, 9.9832).. controls (10.5379, 9.991) and (10.588, 10.0034) .. (10.6325, 10.0279).. controls (10.6769, 10.0524) and (10.7155, 10.09) .. (10.7336, 10.1374).. controls (10.7427, 10.1612) and (10.7464, 10.187) .. (10.7444, 10.2123);

  \path[draw=black,fill=black,line width=0.0118cm,cm={ 0.6728,-0.3549,0.3519,0.6804,(-1.658, 1.7569)}] (8.1902, 16.8198) -- (8.0991, 16.7825) -- (8.1124, 16.8801) -- cycle;

%% file: figs/change_of_sign_demo_04.tex
\definecolor{c819d43}{RGB}{129,157,67}
\definecolor{c484537}{RGB}{72,69,55}
\definecolor{c979797}{RGB}{151,151,151}

  \path[draw=c819d43,line width=0.02cm] (8.688, 10.8814) -- (8.688, 10.2123);

  \path[draw=c819d43,line width=0.02cm] (10.4872, 10.8814) -- (10.4872, 10.2123);

  \path[draw=black,fill=c484537,line width=0.015cm,dash pattern=on 0.015cm off 0.06cm] (8.6811, 10.2112) -- (10.4872, 10.2123);

  \path[draw=black,fill=c979797,line width=0.0487cm] (8.688, 10.2123) circle (0.0458cm);

  \path[draw=black,fill=c979797,line width=0.0487cm] (10.4872, 10.2123) circle (0.0458cm);

  \node[text=black,anchor=south west,line width=0.02cm] (baseline_text) at (9.2078, 9.5499){$\scriptstyle{u=u_-u_+}$};

  \path[draw=black,line width=0.02cm,cm={ 1.0,-0.0,-0.0,-1.0,(0.0, 20.4351)}] (8.3793, 10.2112).. controls (8.3822, 10.2587) and (8.4, 10.3051) .. (8.4294, 10.3425).. controls (8.4595, 10.3807) and (8.5011, 10.409) .. (8.5461, 10.4276).. controls (8.5911, 10.4461) and (8.6395, 10.4554) .. (8.688, 10.4599).. controls (8.8359, 10.4735) and (8.985, 10.4433) .. (9.1251, 10.3941).. controls (9.2652, 10.3448) and (9.3977, 10.2768) .. (9.5312, 10.2118).. controls (9.681, 10.1389) and (9.8333, 10.0693) .. (9.9936, 10.0238).. controls (10.1538, 9.9783) and (10.3232, 9.9574) .. (10.4878, 9.9832).. controls (10.5379, 9.991) and (10.588, 10.0034) .. (10.6325, 10.0279).. controls (10.6769, 10.0524) and (10.7155, 10.09) .. (10.7336, 10.1374).. controls (10.7427, 10.1612) and (10.7464, 10.187) .. (10.7444, 10.2123);

  \path[draw=black,fill=black,line width=0.0118cm,cm={ -0.6728,0.3549,-0.3519,-0.6804,(20.6708, 18.6643)}] (8.1902, 16.8198) -- (8.0991, 16.7825) -- (8.1124, 16.8801) -- cycle;

%% file: figs/change_of_sign_close_to_middle_3_crossing.tex
\definecolor{c819d43}{RGB}{129,157,67}
\definecolor{c484537}{RGB}{72,69,55}
\definecolor{c979797}{RGB}{151,151,151}
\definecolor{cd40000}{RGB}{212,0,0}
\definecolor{navy}{RGB}{0,0,128}

  \path[draw=c819d43,line width=0.02cm] (8.688, 11.3047) -- (8.688, 10.2123);

  \node[text=black,line width=0.005cm,anchor=south west] (text17) at (8.3793, 11.4242){$\scriptscriptstyle\lambda_{i}$};

  \path[draw=c819d43,line width=0.02cm] (10.4872, 11.3047) -- (10.4872, 10.2123);

  \path[draw=black,fill=c484537,line width=0.015cm,dash pattern=on 0.015cm off 0.06cm] (8.6811, 10.2112) -- (10.4872, 10.2123);

  \path[draw=black,fill=c979797,line width=0.0487cm] (8.688, 10.2123) circle (0.0458cm);

  \path[draw=black,fill=c979797,line width=0.0487cm] (10.4872, 10.2123) circle (0.0458cm);

  \node[text=black,line width=0.005cm,anchor=south west] (text57948-7) at (10.1785, 11.4242){$\scriptscriptstyle\lambda_{j}$};

  \path[draw=cd40000,line width=0.02cm] (9.581, 10.2099).. controls (9.6433, 10.1903) and (9.7056, 10.1704) .. (9.7677, 10.1502).. controls (9.9084, 10.1045) and (10.0492, 10.057) .. (10.195, 10.0319).. controls (10.2873, 10.016) and (10.3812, 10.0116) .. (10.4749, 10.0107) -- (10.5871, 10.0097).. controls (10.6082, 10.0095) and (10.6293, 10.0096) .. (10.6501, 10.0129).. controls (10.6708, 10.0162) and (10.6911, 10.0228) .. (10.7094, 10.0332).. controls (10.7461, 10.0538) and (10.7729, 10.089) .. (10.7898, 10.1275).. controls (10.8223, 10.2015) and (10.8226, 10.2892) .. (10.7898, 10.363).. controls (10.7728, 10.4014) and (10.7457, 10.4362) .. (10.7091, 10.4568).. controls (10.6908, 10.4671) and (10.6706, 10.4737) .. (10.6499, 10.4771).. controls (10.6292, 10.4806) and (10.6081, 10.4808) .. (10.5871, 10.4808) -- (8.5758, 10.4814).. controls (8.5303, 10.4814) and (8.4841, 10.4798) .. (8.4409, 10.4657).. controls (8.3976, 10.4515) and (8.3583, 10.4256) .. (8.3285, 10.3912).. controls (8.2739, 10.3281) and (8.253, 10.2419) .. (8.2508, 10.1584).. controls (8.2485, 10.0723) and (8.2659, 9.9822) .. (8.3195, 9.9147).. controls (8.3492, 9.8774) and (8.3896, 9.8488) .. (8.4346, 9.8331).. controls (8.4796, 9.8174) and (8.5281, 9.8155) .. (8.5758, 9.8155) -- (10.5819, 9.8158).. controls (10.6206, 9.8159) and (10.6594, 9.8162) .. (10.6976, 9.8225).. controls (10.7358, 9.8288) and (10.7731, 9.8408) .. (10.8069, 9.8596).. controls (10.8746, 9.8971) and (10.9249, 9.961) .. (10.9564, 10.0317).. controls (11.0167, 10.1669) and (11.0148, 10.3273) .. (10.9564, 10.4634).. controls (10.9292, 10.5268) and (10.8892, 10.5863) .. (10.8334, 10.627).. controls (10.7878, 10.6603) and (10.7332, 10.6799) .. (10.6777, 10.6903).. controls (10.6479, 10.6959) and (10.6175, 10.6977) .. (10.5871, 10.6977) -- (10.1139, 10.6977);

  \path[draw=navy,line width=0.02cm] (9.5817, 10.2095).. controls (9.4682, 10.2449) and (9.3536, 10.2769) .. (9.2381, 10.3053).. controls (9.1487, 10.3273) and (9.0588, 10.3472) .. (8.9679, 10.3616).. controls (8.8754, 10.3763) and (8.7817, 10.382) .. (8.688, 10.3828) -- (8.5758, 10.3838).. controls (8.5547, 10.384) and (8.5336, 10.3839) .. (8.5129, 10.3806).. controls (8.4921, 10.3773) and (8.4718, 10.3707) .. (8.4535, 10.3604).. controls (8.4168, 10.3397) and (8.39, 10.3045) .. (8.3731, 10.2661).. controls (8.3406, 10.1921) and (8.3403, 10.1044) .. (8.3731, 10.0305).. controls (8.3901, 9.9922) and (8.4172, 9.9573) .. (8.4538, 9.9368).. controls (8.4721, 9.9265) and (8.4923, 9.9199) .. (8.513, 9.9164).. controls (8.5337, 9.913) and (8.5548, 9.9128) .. (8.5758, 9.9127) -- (10.608, 9.9121).. controls (10.6502, 9.9121) and (10.6933, 9.9141) .. (10.7329, 9.9288).. controls (10.7726, 9.9435) and (10.8075, 9.9697) .. (10.8344, 10.0023).. controls (10.8874, 10.0666) and (10.9093, 10.1519) .. (10.9121, 10.2351).. controls (10.915, 10.3213) and (10.8972, 10.4114) .. (10.8434, 10.4788).. controls (10.8137, 10.5161) and (10.7733, 10.5447) .. (10.7283, 10.5604).. controls (10.6833, 10.5762) and (10.6348, 10.578) .. (10.5871, 10.578) -- (8.5758, 10.5777).. controls (8.5374, 10.5777) and (8.4989, 10.5773) .. (8.4611, 10.571).. controls (8.4233, 10.5647) and (8.3864, 10.5525) .. (8.353, 10.5337).. controls (8.2863, 10.496) and (8.2371, 10.4321) .. (8.2065, 10.3619).. controls (8.1474, 10.2263) and (8.1498, 10.0667) .. (8.2065, 9.9302).. controls (8.2363, 9.8583) and (8.2837, 9.7916) .. (8.3504, 9.7517).. controls (8.3838, 9.7318) and (8.4212, 9.7188) .. (8.4595, 9.712).. controls (8.4978, 9.7053) and (8.5369, 9.7048) .. (8.5758, 9.7048) -- (10.5819, 9.7048).. controls (10.7088, 9.6992) and (10.8373, 9.744) .. (10.9332, 9.8273).. controls (10.9814, 9.8691) and (11.0214, 9.9202) .. (11.0511, 9.9767).. controls (11.1012, 10.0718) and (11.1213, 10.1813) .. (11.1154, 10.2886).. controls (11.1104, 10.3818) and (11.0858, 10.4744) .. (11.0394, 10.5555).. controls (10.9877, 10.6459) and (10.9089, 10.7208) .. (10.8152, 10.7663).. controls (10.7445, 10.8007) and (10.6657, 10.8184) .. (10.5871, 10.8175) -- (10.1209, 10.8175);

  \node[text=black,anchor=south west,line width=0.02cm] (baseline_text) at (9.2078, 9.1265){$\scriptstyle\ell(u_-)=3$};

%% file: figs/change_of_sign_close_to_middle_4_crossing.tex
\definecolor{c819d43}{RGB}{129,157,67}
\definecolor{c484537}{RGB}{72,69,55}
\definecolor{c979797}{RGB}{151,151,151}
\definecolor{cd40000}{RGB}{212,0,0}
\definecolor{navy}{RGB}{0,0,128}

  \path[draw=c819d43,line width=0.02cm] (6.9249, 17.5159) -- (6.9249, 16.4235);

  \node[text=black,line width=0.005cm,anchor=south west] (text17) at (6.6162, 17.6354){$\scriptscriptstyle\lambda_{i}$};

  \path[draw=c819d43,line width=0.02cm] (8.7241, 17.5159) -- (8.7241, 16.4235);

  \path[draw=black,fill=c484537,line width=0.015cm,dash pattern=on 0.015cm off 0.06cm,shift={(-1.7631, 6.2112)}] (8.6811, 10.2112) -- (10.4872, 10.2123);

  \path[draw=black,fill=c979797,line width=0.0487cm] (6.9249, 16.4235) circle (0.0458cm);

  \path[draw=black,fill=c979797,line width=0.0487cm] (8.724, 16.4235) circle (0.0458cm);

  \node[text=black,line width=0.005cm,anchor=south west] (text57948-7) at (8.4153, 17.6354){$\scriptscriptstyle\lambda_{j}$};

  \path[draw=cd40000,line width=0.02cm,shift={(-1.7631, 6.2112)}] (9.581, 10.2099).. controls (9.6433, 10.1903) and (9.7056, 10.1704) .. (9.7677, 10.1502).. controls (9.9084, 10.1045) and (10.0492, 10.057) .. (10.195, 10.0319).. controls (10.2873, 10.016) and (10.3812, 10.0116) .. (10.4749, 10.0107) -- (10.5871, 10.0097).. controls (10.6082, 10.0095) and (10.6293, 10.0096) .. (10.6501, 10.0129).. controls (10.6708, 10.0162) and (10.6911, 10.0228) .. (10.7094, 10.0332).. controls (10.7461, 10.0538) and (10.7729, 10.089) .. (10.7898, 10.1275).. controls (10.8223, 10.2015) and (10.8226, 10.2892) .. (10.7898, 10.363).. controls (10.7728, 10.4014) and (10.7457, 10.4362) .. (10.7091, 10.4568).. controls (10.6908, 10.4671) and (10.6706, 10.4737) .. (10.6499, 10.4771).. controls (10.6292, 10.4806) and (10.6081, 10.4808) .. (10.5871, 10.4808) -- (8.5758, 10.4814).. controls (8.5303, 10.4814) and (8.4841, 10.4798) .. (8.4409, 10.4657).. controls (8.3976, 10.4515) and (8.3583, 10.4256) .. (8.3285, 10.3912).. controls (8.2739, 10.3281) and (8.253, 10.2419) .. (8.2508, 10.1584).. controls (8.2485, 10.0723) and (8.2659, 9.9822) .. (8.3195, 9.9147).. controls (8.3492, 9.8774) and (8.3896, 9.8488) .. (8.4346, 9.8331).. controls (8.4796, 9.8174) and (8.5281, 9.8155) .. (8.5758, 9.8155) -- (10.5819, 9.8158).. controls (10.6206, 9.8159) and (10.6594, 9.8162) .. (10.6976, 9.8225).. controls (10.7358, 9.8288) and (10.7731, 9.8408) .. (10.8069, 9.8596).. controls (10.8746, 9.8971) and (10.9249, 9.961) .. (10.9564, 10.0317).. controls (11.0167, 10.1669) and (11.0148, 10.3273) .. (10.9564, 10.4634).. controls (10.9292, 10.5268) and (10.8892, 10.5863) .. (10.8334, 10.627).. controls (10.7878, 10.6603) and (10.7332, 10.6799) .. (10.6777, 10.6903).. controls (10.6479, 10.6959) and (10.6175, 10.6977) .. (10.5871, 10.6977) -- (8.5758, 10.6977);

  \path[draw=navy,line width=0.02cm,shift={(-1.7631, 6.2112)}] (9.5817, 10.2095).. controls (9.4682, 10.2449) and (9.3536, 10.2769) .. (9.2381, 10.3053).. controls (9.1487, 10.3273) and (9.0588, 10.3472) .. (8.9679, 10.3616).. controls (8.8754, 10.3763) and (8.7817, 10.382) .. (8.688, 10.3828) -- (8.5758, 10.3838).. controls (8.5547, 10.384) and (8.5336, 10.3839) .. (8.5129, 10.3806).. controls (8.4921, 10.3773) and (8.4718, 10.3707) .. (8.4535, 10.3604).. controls (8.4168, 10.3397) and (8.39, 10.3045) .. (8.3731, 10.2661).. controls (8.3406, 10.1921) and (8.3403, 10.1044) .. (8.3731, 10.0305).. controls (8.3901, 9.9922) and (8.4172, 9.9573) .. (8.4538, 9.9368).. controls (8.4721, 9.9265) and (8.4923, 9.9199) .. (8.513, 9.9164).. controls (8.5337, 9.913) and (8.5548, 9.9128) .. (8.5758, 9.9127) -- (10.608, 9.9121).. controls (10.6502, 9.9121) and (10.6933, 9.9141) .. (10.7329, 9.9288).. controls (10.7726, 9.9435) and (10.8075, 9.9697) .. (10.8344, 10.0023).. controls (10.8874, 10.0666) and (10.9093, 10.1519) .. (10.9121, 10.2351).. controls (10.915, 10.3213) and (10.8972, 10.4114) .. (10.8434, 10.4788).. controls (10.8137, 10.5161) and (10.7733, 10.5447) .. (10.7283, 10.5604).. controls (10.6833, 10.5762) and (10.6348, 10.578) .. (10.5871, 10.578) -- (8.5758, 10.5777).. controls (8.5374, 10.5777) and (8.4989, 10.5773) .. (8.4611, 10.571).. controls (8.4233, 10.5647) and (8.3864, 10.5525) .. (8.353, 10.5337).. controls (8.2863, 10.496) and (8.2371, 10.4321) .. (8.2065, 10.3619).. controls (8.1474, 10.2263) and (8.1498, 10.0667) .. (8.2065, 9.9302).. controls (8.2363, 9.8583) and (8.2837, 9.7916) .. (8.3504, 9.7517).. controls (8.3838, 9.7318) and (8.4212, 9.7188) .. (8.4595, 9.712).. controls (8.4978, 9.7053) and (8.5369, 9.7048) .. (8.5758, 9.7048) -- (10.5819, 9.7048).. controls (10.7088, 9.6992) and (10.8373, 9.744) .. (10.9332, 9.8273).. controls (10.9814, 9.8691) and (11.0214, 9.9202) .. (11.0511, 9.9767).. controls (11.1012, 10.0718) and (11.1213, 10.1813) .. (11.1154, 10.2886).. controls (11.1104, 10.3818) and (11.0858, 10.4744) .. (11.0394, 10.5555).. controls (10.9877, 10.6459) and (10.9089, 10.7208) .. (10.8152, 10.7663).. controls (10.7445, 10.8007) and (10.6657, 10.8184) .. (10.5871, 10.8175) -- (10.1209, 10.8175);

  \node[text=black,anchor=south west,line width=0.02cm] (baseline_text) at (7.3849, 15.3291){$\scriptstyle\ell(u_-)=4$};

%% file: figs/change_of_sign_close_to_middle_5_crossing.tex
\definecolor{c819d43}{RGB}{129,157,67}
\definecolor{c484537}{RGB}{72,69,55}
\definecolor{c979797}{RGB}{151,151,151}
\definecolor{navy}{RGB}{0,0,128}
\definecolor{cd40000}{RGB}{212,0,0}

  \path[draw=c819d43,line width=0.02cm] (5.484, 12.9256) -- (5.484, 11.8333);

  \node[text=black,line width=0.005cm,anchor=south west] (text17) at (5.1753, 13.0452){$\scriptscriptstyle\lambda_{i}$};

  \path[draw=c819d43,line width=0.02cm] (7.2831, 12.9256) -- (7.2831, 11.8333);

  \path[draw=black,fill=c484537,line width=0.015cm,dash pattern=on 0.015cm off 0.06cm,shift={(-3.204, 1.6209)}] (8.6811, 10.2112) -- (10.4872, 10.2123);

  \path[draw=black,fill=c979797,line width=0.0487cm] (5.484, 11.8333) circle (0.0458cm);

  \path[draw=black,fill=c979797,line width=0.0487cm] (7.2831, 11.8333) circle (0.0458cm);

  \node[text=black,line width=0.005cm,anchor=south west] (text57948-7) at (6.9744, 13.0452){$\scriptscriptstyle\lambda_{j}$};

  \path[draw=navy,line width=0.02cm,shift={(-3.204, 1.6209)}] (9.5817, 10.2095).. controls (9.4682, 10.2449) and (9.3536, 10.2769) .. (9.2381, 10.3053).. controls (9.1487, 10.3273) and (9.0588, 10.3472) .. (8.9679, 10.3616).. controls (8.8754, 10.3763) and (8.7817, 10.382) .. (8.688, 10.3828) -- (8.5758, 10.3838).. controls (8.5547, 10.384) and (8.5336, 10.3839) .. (8.5129, 10.3806).. controls (8.4921, 10.3773) and (8.4718, 10.3707) .. (8.4535, 10.3604).. controls (8.4168, 10.3397) and (8.39, 10.3045) .. (8.3731, 10.2661).. controls (8.3406, 10.1921) and (8.3403, 10.1044) .. (8.3731, 10.0305).. controls (8.3901, 9.9922) and (8.4172, 9.9573) .. (8.4538, 9.9368).. controls (8.4721, 9.9265) and (8.4923, 9.9199) .. (8.513, 9.9164).. controls (8.5337, 9.913) and (8.5548, 9.9128) .. (8.5758, 9.9127) -- (10.608, 9.9121).. controls (10.6502, 9.9121) and (10.6933, 9.9141) .. (10.7329, 9.9288).. controls (10.7726, 9.9435) and (10.8075, 9.9697) .. (10.8344, 10.0023).. controls (10.8874, 10.0666) and (10.9093, 10.1519) .. (10.9121, 10.2351).. controls (10.915, 10.3213) and (10.8972, 10.4114) .. (10.8434, 10.4788).. controls (10.8137, 10.5161) and (10.7733, 10.5447) .. (10.7283, 10.5604).. controls (10.6833, 10.5762) and (10.6348, 10.578) .. (10.5871, 10.578) -- (8.5758, 10.5777).. controls (8.5374, 10.5777) and (8.4989, 10.5773) .. (8.4611, 10.571).. controls (8.4233, 10.5647) and (8.3864, 10.5525) .. (8.353, 10.5337).. controls (8.2863, 10.496) and (8.2371, 10.4321) .. (8.2065, 10.3619).. controls (8.1474, 10.2263) and (8.1498, 10.0667) .. (8.2065, 9.9302).. controls (8.2363, 9.8583) and (8.2837, 9.7916) .. (8.3504, 9.7517).. controls (8.3838, 9.7318) and (8.4212, 9.7188) .. (8.4595, 9.712).. controls (8.4978, 9.7053) and (8.5369, 9.7048) .. (8.5758, 9.7048) -- (10.5819, 9.7048).. controls (10.7088, 9.6992) and (10.8373, 9.744) .. (10.9332, 9.8273).. controls (10.9814, 9.8691) and (11.0214, 9.9202) .. (11.0511, 9.9767).. controls (11.1012, 10.0718) and (11.1213, 10.1813) .. (11.1154, 10.2886).. controls (11.1104, 10.3818) and (11.0858, 10.4744) .. (11.0394, 10.5555).. controls (10.9877, 10.6459) and (10.9089, 10.7208) .. (10.8152, 10.7663).. controls (10.7445, 10.8007) and (10.6657, 10.8184) .. (10.5871, 10.8175) -- (10.1209, 10.8175);

  \path[draw=cd40000,line width=0.02cm,shift={(0.0794, 1.1906)}] (5.2924, 11.128).. controls (5.1994, 11.1287) and (5.1064, 11.1038) .. (5.0262, 11.0568).. controls (4.9667, 11.0219) and (4.9142, 10.9749) .. (4.874, 10.9189).. controls (4.7903, 10.8025) and (4.7633, 10.6526) .. (4.7745, 10.5097).. controls (4.7825, 10.4066) and (4.8104, 10.3028) .. (4.8707, 10.2187).. controls (4.9108, 10.1629) and (4.9643, 10.1172) .. (5.0241, 10.0834).. controls (5.1045, 10.0379) and (5.1966, 10.0135) .. (5.2889, 10.0133) -- (7.3198, 10.0113).. controls (7.3852, 10.0112) and (7.4513, 10.0144) .. (7.5135, 10.0342).. controls (7.564, 10.0502) and (7.6112, 10.0756) .. (7.6545, 10.1061).. controls (7.7909, 10.2022) and (7.8883, 10.3505) .. (7.9279, 10.5125).. controls (7.9605, 10.6458) and (7.955, 10.7879) .. (7.9151, 10.9192).. controls (7.8813, 11.0304) and (7.8223, 11.1351) .. (7.7372, 11.2142).. controls (7.6542, 11.2913) and (7.5481, 11.3421) .. (7.4368, 11.363).. controls (7.3619, 11.377) and (7.2852, 11.3776) .. (7.209, 11.3776) -- (6.8375, 11.3776)(6.2976, 10.6402).. controls (6.3599, 10.6206) and (6.4222, 10.6007) .. (6.4843, 10.5805).. controls (6.625, 10.5348) and (6.7658, 10.4874) .. (6.9116, 10.4622).. controls (7.0039, 10.4463) and (7.0978, 10.4419) .. (7.1915, 10.441) -- (7.3037, 10.44).. controls (7.3248, 10.4398) and (7.3459, 10.4399) .. (7.3667, 10.4432).. controls (7.3874, 10.4465) and (7.4077, 10.4532) .. (7.426, 10.4635).. controls (7.4627, 10.4841) and (7.4895, 10.5193) .. (7.5064, 10.5578).. controls (7.5389, 10.6318) and (7.5392, 10.7195) .. (7.5064, 10.7933).. controls (7.4894, 10.8317) and (7.4624, 10.8666) .. (7.4258, 10.8871).. controls (7.4075, 10.8974) and (7.3872, 10.904) .. (7.3665, 10.9075).. controls (7.3458, 10.9109) and (7.3247, 10.9111) .. (7.3037, 10.9111) -- (5.2924, 10.9102).. controls (5.247, 10.9102) and (5.201, 10.9086) .. (5.1578, 10.8948).. controls (5.1146, 10.881) and (5.0751, 10.8556) .. (5.0451, 10.8215).. controls (4.99, 10.7588) and (4.9695, 10.6722) .. (4.9674, 10.5887).. controls (4.9652, 10.5026) and (4.9825, 10.4125) .. (5.0361, 10.3451).. controls (5.0658, 10.3077) and (5.1061, 10.2791) .. (5.1512, 10.2634).. controls (5.1962, 10.2477) and (5.2447, 10.2459) .. (5.2924, 10.2459) -- (7.2985, 10.2462).. controls (7.3372, 10.2462) and (7.376, 10.2466) .. (7.4142, 10.2528).. controls (7.4524, 10.2591) and (7.4897, 10.2711) .. (7.5235, 10.2899).. controls (7.5912, 10.3274) and (7.6415, 10.3914) .. (7.673, 10.462).. controls (7.7333, 10.5972) and (7.7314, 10.7577) .. (7.673, 10.8937).. controls (7.6458, 10.9571) and (7.6058, 11.0166) .. (7.55, 11.0573).. controls (7.5044, 11.0906) and (7.4498, 11.1102) .. (7.3943, 11.1206).. controls (7.3645, 11.1262) and (7.3341, 11.1281) .. (7.3037, 11.1281) -- (5.2924, 11.1281);

  \node[text=black,anchor=south west,line width=0.02cm] (baseline_text) at (5.9702, 10.722){$\scriptstyle\ell(u_-)=5$};

%% file: figs/change_of_sign_close_to_middle_2_crossing.tex
\definecolor{c819d43}{RGB}{129,157,67}
\definecolor{c484537}{RGB}{72,69,55}
\definecolor{c979797}{RGB}{151,151,151}
\definecolor{cd40000}{RGB}{212,0,0}
\definecolor{navy}{RGB}{0,0,128}

  \path[draw=c819d43,line width=0.02cm] (8.688, 11.3047) -- (8.688, 10.2123);

  \node[text=black,line width=0.005cm,anchor=south west] (text17) at (8.3793, 11.4242){$\scriptscriptstyle\lambda_{i}$};

  \path[draw=c819d43,line width=0.02cm] (10.4872, 11.3047) -- (10.4872, 10.2123);

  \path[draw=black,fill=c484537,line width=0.015cm,dash pattern=on 0.015cm off 0.06cm] (8.6811, 10.2112) -- (10.4872, 10.2123);

  \path[draw=black,fill=c979797,line width=0.0487cm] (8.688, 10.2123) circle (0.0458cm);

  \path[draw=black,fill=c979797,line width=0.0487cm] (10.4872, 10.2123) circle (0.0458cm);

  \node[text=black,line width=0.005cm,anchor=south west] (text57948-7) at (10.1785, 11.4242){$\scriptscriptstyle\lambda_{j}$};

  \path[draw=cd40000,line width=0.02cm] (9.581, 10.2099).. controls (9.6433, 10.1903) and (9.7056, 10.1704) .. (9.7677, 10.1502).. controls (9.9084, 10.1045) and (10.0492, 10.057) .. (10.195, 10.0319).. controls (10.2873, 10.016) and (10.3812, 10.0116) .. (10.4749, 10.0107) -- (10.5871, 10.0097).. controls (10.6082, 10.0095) and (10.6293, 10.0096) .. (10.6501, 10.0129).. controls (10.6708, 10.0162) and (10.6911, 10.0228) .. (10.7094, 10.0332).. controls (10.7461, 10.0538) and (10.7729, 10.089) .. (10.7898, 10.1275).. controls (10.8223, 10.2015) and (10.8226, 10.2892) .. (10.7898, 10.363).. controls (10.7728, 10.4014) and (10.7457, 10.4362) .. (10.7091, 10.4568).. controls (10.6908, 10.4671) and (10.6706, 10.4737) .. (10.6499, 10.4771).. controls (10.6292, 10.4806) and (10.6081, 10.4808) .. (10.5871, 10.4808) -- (8.5758, 10.4814).. controls (8.5303, 10.4814) and (8.4841, 10.4798) .. (8.4409, 10.4657).. controls (8.3976, 10.4515) and (8.3583, 10.4256) .. (8.3285, 10.3912).. controls (8.2739, 10.3281) and (8.253, 10.2419) .. (8.2508, 10.1584).. controls (8.2485, 10.0723) and (8.2659, 9.9822) .. (8.3195, 9.9147).. controls (8.3492, 9.8774) and (8.3896, 9.8488) .. (8.4346, 9.8331).. controls (8.4796, 9.8174) and (8.5281, 9.8155) .. (8.5758, 9.8155) -- (10.5819, 9.8158).. controls (10.6206, 9.8159) and (10.6594, 9.8162) .. (10.6976, 9.8225).. controls (10.7358, 9.8288) and (10.7731, 9.8409) .. (10.8069, 9.8596).. controls (10.8745, 9.8972) and (10.9248, 9.9611) .. (10.9564, 10.0317).. controls (11.0169, 10.1669) and (11.0158, 10.3277) .. (10.9564, 10.4634).. controls (10.9288, 10.5266) and (10.8883, 10.5853) .. (10.8334, 10.627).. controls (10.7883, 10.6613) and (10.734, 10.6834) .. (10.6777, 10.6903);

  \path[draw=navy,line width=0.02cm] (9.5817, 10.2095).. controls (9.4682, 10.2449) and (9.3536, 10.2769) .. (9.2381, 10.3053).. controls (9.1487, 10.3273) and (9.0588, 10.3472) .. (8.9679, 10.3616).. controls (8.8754, 10.3763) and (8.7817, 10.382) .. (8.688, 10.3828) -- (8.5758, 10.3838).. controls (8.5547, 10.384) and (8.5336, 10.3839) .. (8.5129, 10.3806).. controls (8.4921, 10.3773) and (8.4718, 10.3707) .. (8.4535, 10.3604).. controls (8.4168, 10.3397) and (8.39, 10.3045) .. (8.3731, 10.2661).. controls (8.3406, 10.1921) and (8.3403, 10.1044) .. (8.3731, 10.0305).. controls (8.3901, 9.9922) and (8.4172, 9.9573) .. (8.4538, 9.9368).. controls (8.4721, 9.9265) and (8.4923, 9.9199) .. (8.513, 9.9164).. controls (8.5337, 9.913) and (8.5548, 9.9128) .. (8.5758, 9.9127) -- (10.608, 9.9121).. controls (10.6502, 9.9121) and (10.6933, 9.9141) .. (10.7329, 9.9288).. controls (10.7726, 9.9435) and (10.8075, 9.9697) .. (10.8344, 10.0023).. controls (10.8874, 10.0666) and (10.9093, 10.1519) .. (10.9121, 10.2351).. controls (10.915, 10.3213) and (10.8972, 10.4114) .. (10.8434, 10.4788).. controls (10.8137, 10.5161) and (10.7733, 10.5447) .. (10.7283, 10.5604).. controls (10.6833, 10.5762) and (10.6348, 10.578) .. (10.5871, 10.578) -- (8.5758, 10.5777).. controls (8.5374, 10.5777) and (8.4989, 10.5773) .. (8.4611, 10.571).. controls (8.4233, 10.5647) and (8.3864, 10.5525) .. (8.353, 10.5337).. controls (8.2863, 10.496) and (8.2371, 10.4321) .. (8.2065, 10.3619).. controls (8.1474, 10.2263) and (8.1498, 10.0667) .. (8.2065, 9.9302).. controls (8.2363, 9.8583) and (8.2837, 9.7916) .. (8.3504, 9.7517).. controls (8.3838, 9.7318) and (8.4212, 9.7188) .. (8.4595, 9.712).. controls (8.4978, 9.7053) and (8.5369, 9.7048) .. (8.5758, 9.7048) -- (10.5819, 9.7048).. controls (10.7088, 9.6992) and (10.8373, 9.744) .. (10.9332, 9.8273).. controls (10.9814, 9.8691) and (11.0214, 9.9202) .. (11.0511, 9.9767).. controls (11.1012, 10.0718) and (11.1213, 10.1813) .. (11.1154, 10.2886).. controls (11.1104, 10.3818) and (11.0858, 10.4744) .. (11.0394, 10.5555).. controls (10.9877, 10.6459) and (10.9089, 10.7208) .. (10.8152, 10.7663).. controls (10.7445, 10.8007) and (10.6657, 10.8184) .. (10.5871, 10.8175) -- (10.1209, 10.8175);

  \node[text=black,anchor=south west,line width=0.02cm] (baseline_text) at (9.1113, 9.1105){$\scriptstyle\ell(u_-)=2$};

%% file: figs/change_of_sign_close_to_middle_6_crossing.tex
\definecolor{c819d43}{RGB}{129,157,67}
\definecolor{c484537}{RGB}{72,69,55}
\definecolor{c979797}{RGB}{151,151,151}
\definecolor{navy}{RGB}{0,0,128}
\definecolor{cd40000}{RGB}{212,0,0}

  \path[draw=c819d43,line width=0.02cm] (5.4046, 11.735) -- (5.4046, 10.6426);

  \node[text=black,line width=0.005cm,anchor=south west] (text17) at (5.0959, 11.8545){$\scriptscriptstyle\lambda_{i}$};

  \path[draw=c819d43,line width=0.02cm] (7.2038, 11.735) -- (7.2038, 10.6426);

  \path[draw=black,fill=c484537,line width=0.015cm,dash pattern=on 0.015cm off 0.06cm,shift={(-3.2834, 0.4303)}] (8.6811, 10.2112) -- (10.4872, 10.2123);

  \path[draw=black,fill=c979797,line width=0.0487cm] (5.4046, 10.6426) circle (0.0458cm);

  \path[draw=black,fill=c979797,line width=0.0487cm] (7.2038, 10.6426) circle (0.0458cm);

  \node[text=black,line width=0.005cm,anchor=south west] (text57948-7) at (6.8951, 11.8545){$\scriptscriptstyle\lambda_{j}$};

  \path[draw=navy,line width=0.02cm,shift={(-3.2834, 0.4303)}] (9.5817, 10.2095).. controls (9.4682, 10.2449) and (9.3536, 10.2769) .. (9.2381, 10.3053).. controls (9.1487, 10.3273) and (9.0588, 10.3472) .. (8.9679, 10.3616).. controls (8.8754, 10.3763) and (8.7817, 10.382) .. (8.688, 10.3828) -- (8.5758, 10.3838).. controls (8.5547, 10.384) and (8.5336, 10.3839) .. (8.5129, 10.3806).. controls (8.4921, 10.3773) and (8.4718, 10.3707) .. (8.4535, 10.3604).. controls (8.4168, 10.3397) and (8.39, 10.3045) .. (8.3731, 10.2661).. controls (8.3406, 10.1921) and (8.3403, 10.1044) .. (8.3731, 10.0305).. controls (8.3901, 9.9922) and (8.4172, 9.9573) .. (8.4538, 9.9368).. controls (8.4721, 9.9265) and (8.4923, 9.9199) .. (8.513, 9.9164).. controls (8.5337, 9.913) and (8.5548, 9.9128) .. (8.5758, 9.9127) -- (10.608, 9.9121).. controls (10.6502, 9.9121) and (10.6933, 9.9141) .. (10.7329, 9.9288).. controls (10.7726, 9.9435) and (10.8075, 9.9697) .. (10.8344, 10.0023).. controls (10.8874, 10.0666) and (10.9093, 10.1519) .. (10.9121, 10.2351).. controls (10.915, 10.3213) and (10.8972, 10.4114) .. (10.8434, 10.4788).. controls (10.8137, 10.5161) and (10.7733, 10.5447) .. (10.7283, 10.5604).. controls (10.6833, 10.5762) and (10.6348, 10.578) .. (10.5871, 10.578) -- (8.5758, 10.5777).. controls (8.5374, 10.5777) and (8.4989, 10.5773) .. (8.4611, 10.571).. controls (8.4233, 10.5647) and (8.3864, 10.5525) .. (8.353, 10.5337).. controls (8.2863, 10.496) and (8.2371, 10.4321) .. (8.2065, 10.3619).. controls (8.1474, 10.2263) and (8.1498, 10.0667) .. (8.2065, 9.9302).. controls (8.2363, 9.8583) and (8.2837, 9.7916) .. (8.3504, 9.7517).. controls (8.3838, 9.7318) and (8.4212, 9.7188) .. (8.4595, 9.712).. controls (8.4978, 9.7053) and (8.5369, 9.7048) .. (8.5758, 9.7048) -- (10.5819, 9.7048).. controls (10.7088, 9.6992) and (10.8373, 9.744) .. (10.9332, 9.8273).. controls (10.9814, 9.8691) and (11.0214, 9.9202) .. (11.0511, 9.9767).. controls (11.1012, 10.0718) and (11.1213, 10.1813) .. (11.1154, 10.2886).. controls (11.1104, 10.3818) and (11.0858, 10.4744) .. (11.0394, 10.5555).. controls (10.9877, 10.6459) and (10.9089, 10.7208) .. (10.8152, 10.7663).. controls (10.7445, 10.8007) and (10.6657, 10.8184) .. (10.5871, 10.8175) -- (10.1209, 10.8175);

  \path[draw=cd40000,line width=0.02cm] (5.2924, 11.128).. controls (5.1994, 11.1287) and (5.1064, 11.1038) .. (5.0262, 11.0568).. controls (4.9667, 11.0219) and (4.9142, 10.9749) .. (4.874, 10.9189).. controls (4.7903, 10.8025) and (4.7633, 10.6526) .. (4.7745, 10.5097).. controls (4.7825, 10.4066) and (4.8104, 10.3028) .. (4.8707, 10.2187).. controls (4.9108, 10.1629) and (4.9643, 10.1172) .. (5.0241, 10.0834).. controls (5.1045, 10.0379) and (5.1966, 10.0135) .. (5.2889, 10.0133) -- (7.3198, 10.0113).. controls (7.3852, 10.0112) and (7.4513, 10.0144) .. (7.5135, 10.0342).. controls (7.564, 10.0502) and (7.6112, 10.0756) .. (7.6545, 10.1061).. controls (7.7909, 10.2022) and (7.8883, 10.3505) .. (7.9279, 10.5125).. controls (7.9605, 10.6458) and (7.955, 10.7879) .. (7.9151, 10.9192).. controls (7.8813, 11.0304) and (7.8223, 11.1351) .. (7.7372, 11.2142).. controls (7.6542, 11.2913) and (7.5481, 11.3421) .. (7.4368, 11.363).. controls (7.3619, 11.377) and (7.2852, 11.3776) .. (7.209, 11.3776) -- (5.2924, 11.3776)(6.2976, 10.6402).. controls (6.3599, 10.6206) and (6.4222, 10.6007) .. (6.4843, 10.5805).. controls (6.625, 10.5348) and (6.7658, 10.4874) .. (6.9116, 10.4622).. controls (7.0039, 10.4463) and (7.0978, 10.4419) .. (7.1915, 10.441) -- (7.3037, 10.44).. controls (7.3248, 10.4398) and (7.3459, 10.4399) .. (7.3667, 10.4432).. controls (7.3874, 10.4465) and (7.4077, 10.4532) .. (7.426, 10.4635).. controls (7.4627, 10.4841) and (7.4895, 10.5193) .. (7.5064, 10.5578).. controls (7.5389, 10.6318) and (7.5392, 10.7195) .. (7.5064, 10.7933).. controls (7.4894, 10.8317) and (7.4624, 10.8666) .. (7.4258, 10.8871).. controls (7.4075, 10.8974) and (7.3872, 10.904) .. (7.3665, 10.9075).. controls (7.3458, 10.9109) and (7.3247, 10.9111) .. (7.3037, 10.9111) -- (5.2924, 10.9102).. controls (5.247, 10.9102) and (5.201, 10.9086) .. (5.1578, 10.8948).. controls (5.1146, 10.881) and (5.0751, 10.8556) .. (5.0451, 10.8215).. controls (4.99, 10.7588) and (4.9695, 10.6722) .. (4.9674, 10.5887).. controls (4.9652, 10.5026) and (4.9825, 10.4125) .. (5.0361, 10.3451).. controls (5.0658, 10.3077) and (5.1061, 10.2791) .. (5.1512, 10.2634).. controls (5.1962, 10.2477) and (5.2447, 10.2459) .. (5.2924, 10.2459) -- (7.2985, 10.2462).. controls (7.3372, 10.2462) and (7.376, 10.2466) .. (7.4142, 10.2528).. controls (7.4524, 10.2591) and (7.4897, 10.2711) .. (7.5235, 10.2899).. controls (7.5912, 10.3274) and (7.6415, 10.3914) .. (7.673, 10.462).. controls (7.7333, 10.5972) and (7.7314, 10.7577) .. (7.673, 10.8937).. controls (7.6458, 10.9571) and (7.6058, 11.0166) .. (7.55, 11.0573).. controls (7.5044, 11.0906) and (7.4498, 11.1102) .. (7.3943, 11.1206).. controls (7.3645, 11.1262) and (7.3341, 11.1281) .. (7.3037, 11.1281) -- (5.2924, 11.1281);

  \node[text=black,anchor=south west,line width=0.02cm] (baseline_text) at (5.8185, 9.5568){$\scriptstyle\ell(u_-)=6$};

%% file: figs/4_crossing_and_other_mu_ij_01.tex
\definecolor{c819d43}{RGB}{129,157,67}
\definecolor{c484537}{RGB}{72,69,55}
\definecolor{c979797}{RGB}{151,151,151}

  \path[draw=c819d43,line width=0.02cm] (8.4882, 11.4475) -- (8.4882, 10.1396);

  \node[text=black,line width=0.005cm,xscale=0.9997,yscale=1.0003,anchor=south west] (text17) at (8.1216, 11.5968){$\scriptscriptstyle\lambda_{i}$};

  \path[draw=c819d43,line width=0.02cm] (10.482, 11.4475) -- (10.482, 10.1396);

  \path[draw=black,fill=c484537,line width=0.0125cm,dash pattern=on 0.0125cm off 0.0501cm,cm={ 1.1964,-0.0,-0.0,1.1973,(-1.9065, -2.0883)}] (8.6811, 10.2112) -- (10.3545, 10.2123);

  \node[text=black,line width=0.005cm,xscale=0.9997,yscale=1.0003,anchor=south west] (text57948-7) at (10.275, 11.5968){$\scriptscriptstyle\lambda_{j}$};

  \node[anchor=south west,line width=0.02cm] (baseline_text) at (8.688, 10.2123){};

  \path[draw=black,fill=c979797,line width=0.0487cm] (8.4882, 10.1396) circle (0.0458cm);

  \path[draw=black,fill=c979797,line width=0.0487cm] (10.482, 10.1396) circle (0.0458cm);

  \node[anchor=south west,line width=0.02cm,dash pattern=on 0.02cm off 0.02cm] (text35) at (8.4882, 10.1396){};

  \path[draw=black,line width=0.02cm,shift={(0.0, 0.0362)}] (10.2547, 10.9011) -- (10.4958, 10.9011).. controls (10.5454, 10.9011) and (10.5953, 10.9005) .. (10.6441, 10.8917).. controls (10.693, 10.883) and (10.7405, 10.8661) .. (10.783, 10.8404).. controls (10.8679, 10.7891) and (10.9275, 10.7038) .. (10.967, 10.6128).. controls (11.031, 10.4651) and (11.0519, 10.2999) .. (11.0339, 10.14).. controls (11.0168, 9.9875) and (10.9633, 9.8363) .. (10.8638, 9.7195).. controls (10.822, 9.6702) and (10.7709, 9.6277) .. (10.7111, 9.6031).. controls (10.6513, 9.5786) and (10.5854, 9.5747) .. (10.5208, 9.5747) -- (8.5438, 9.5754).. controls (8.4958, 9.5754) and (8.4473, 9.5763) .. (8.4003, 9.5862).. controls (8.3532, 9.5961) and (8.3078, 9.6147) .. (8.2686, 9.6427).. controls (8.2256, 9.6734) and (8.1904, 9.7151) .. (8.1674, 9.7628).. controls (8.1347, 9.8253) and (8.1185, 9.8964) .. (8.1209, 9.967).. controls (8.1232, 10.0375) and (8.1441, 10.1074) .. (8.1809, 10.1676).. controls (8.2011, 10.2008) and (8.2264, 10.2313) .. (8.2576, 10.2544).. controls (8.2991, 10.2852) and (8.3507, 10.302) .. (8.4023, 10.3015) -- (8.6708, 10.3016).. controls (8.7481, 10.3016) and (8.8256, 10.3057) .. (8.9027, 10.2997).. controls (8.9412, 10.2967) and (8.9798, 10.2904) .. (9.0158, 10.2762).. controls (9.0517, 10.2619) and (9.0851, 10.2392) .. (9.1071, 10.2075).. controls (9.1301, 10.1744) and (9.1394, 10.1336) .. (9.1425, 10.0934).. controls (9.1457, 10.0533) and (9.1432, 10.0129) .. (9.145, 9.9727).. controls (9.1472, 9.924) and (9.1565, 9.8737) .. (9.1857, 9.8347).. controls (9.2142, 9.7966) and (9.2609, 9.7728) .. (9.3084, 9.7722) -- (10.5208, 9.7722).. controls (10.6065, 9.7729) and (10.6914, 9.8095) .. (10.7508, 9.8714).. controls (10.7879, 9.9102) and (10.815, 9.9581) .. (10.8325, 10.0089).. controls (10.8512, 10.0631) and (10.8591, 10.1204) .. (10.8624, 10.1776).. controls (10.8672, 10.2601) and (10.8622, 10.3438) .. (10.8389, 10.4231).. controls (10.8157, 10.5023) and (10.7733, 10.5771) .. (10.7111, 10.6315).. controls (10.6735, 10.6644) and (10.6289, 10.6891) .. (10.5809, 10.7027).. controls (10.5329, 10.7163) and (10.4824, 10.718) .. (10.4325, 10.718) -- (8.2906, 10.7187);

  \path[draw=black,line width=0.02cm,dash pattern=on 0.02cm off 0.02cm] (9.6865, 10.3635) -- (9.6865, 10.1472).. controls (9.6865, 10.0941) and (9.687, 10.041) .. (9.6895, 9.988);

%% file: figs/change_of_sign_u_and_v.tex
\definecolor{c819d43}{RGB}{129,157,67}
\definecolor{c484537}{RGB}{72,69,55}
\definecolor{c979797}{RGB}{151,151,151}
\definecolor{cd40000}{RGB}{212,0,0}

  \path[draw=c819d43,line width=0.02cm] (8.4882, 11.4475) -- (8.4882, 10.1396);

  \node[text=black,line width=0.005cm,xscale=0.9997,yscale=1.0003,anchor=south west] (text17) at (8.1216, 11.5968){$\scriptscriptstyle\lambda_{i}$};

  \path[draw=c819d43,line width=0.02cm] (10.482, 11.4475) -- (10.482, 10.1396);

  \path[draw=black,fill=c484537,line width=0.0125cm,dash pattern=on 0.0125cm off 0.0501cm,cm={ 1.1964,-0.0,-0.0,1.1973,(-1.9065, -2.0883)}] (8.6811, 10.2112) -- (10.3545, 10.2123);

  \node[text=black,line width=0.005cm,xscale=0.9997,yscale=1.0003,anchor=south west] (text57948-7) at (10.275, 11.5968){$\scriptscriptstyle\lambda_{j}$};

  \node[anchor=south west,line width=0.02cm] (baseline_text) at (8.688, 10.2123){};

  \path[draw=black,fill=c979797,line width=0.0487cm] (8.4882, 10.1396) circle (0.0458cm);

  \path[draw=black,fill=c979797,line width=0.0487cm] (10.482, 10.1396) circle (0.0458cm);

  \node[anchor=south west,line width=0.02cm,dash pattern=on 0.02cm off 0.02cm] (text35) at (8.4882, 10.1396){};

  \path[draw=cd40000,line width=0.02cm] (9.2404, 10.1375) -- (9.2391, 9.94).. controls (9.2383, 9.9172) and (9.2459, 9.8941) .. (9.2601, 9.8763).. controls (9.2717, 9.8619) and (9.2874, 9.851) .. (9.3045, 9.8441).. controls (9.3246, 9.8358) and (9.3466, 9.8329) .. (9.3683, 9.8316).. controls (9.39, 9.8304) and (9.4118, 9.8306) .. (9.4336, 9.8306) -- (10.522, 9.8306).. controls (10.561, 9.8306) and (10.6002, 9.831) .. (10.6387, 9.8373).. controls (10.6772, 9.8435) and (10.7149, 9.8556) .. (10.7491, 9.8744).. controls (10.8174, 9.912) and (10.8685, 9.9763) .. (10.8998, 10.0478).. controls (10.9426, 10.1456) and (10.9526, 10.2557) .. (10.9401, 10.3617).. controls (10.9342, 10.4114) and (10.9234, 10.4609) .. (10.9038, 10.507).. controls (10.8841, 10.553) and (10.8553, 10.5957) .. (10.8169, 10.628).. controls (10.7656, 10.6711) and (10.6996, 10.6936) .. (10.6333, 10.7042).. controls (10.567, 10.7147) and (10.4996, 10.7144) .. (10.4325, 10.7144) -- (8.2906, 10.7151);

  \path[draw=blue,line width=0.02cm] (9.2404, 10.1375) -- (9.2404, 10.2159).. controls (9.2412, 10.2457) and (9.2288, 10.2757) .. (9.2071, 10.2963).. controls (9.1933, 10.3093) and (9.1761, 10.3185) .. (9.1581, 10.3246).. controls (9.1401, 10.3307) and (9.1212, 10.3338) .. (9.1023, 10.3356).. controls (9.0644, 10.339) and (9.0263, 10.3379) .. (8.9883, 10.3378) -- (8.4023, 10.3378).. controls (8.3134, 10.3296) and (8.2294, 10.2788) .. (8.1809, 10.2039).. controls (8.1524, 10.16) and (8.1359, 10.1084) .. (8.1336, 10.0562).. controls (8.1312, 10.0039) and (8.143, 9.9511) .. (8.1674, 9.9048).. controls (8.1904, 9.8572) and (8.2256, 9.8155) .. (8.2686, 9.7848).. controls (8.3078, 9.7568) and (8.3532, 9.7382) .. (8.4003, 9.7283).. controls (8.4473, 9.7184) and (8.4958, 9.7176) .. (8.5438, 9.7175) -- (10.5208, 9.7168).. controls (10.603, 9.7168) and (10.6867, 9.7197) .. (10.7645, 9.7465).. controls (10.8422, 9.7732) and (10.9123, 9.8227) .. (10.9621, 9.8882).. controls (11.0228, 9.968) and (11.0525, 10.0675) .. (11.0667, 10.1669).. controls (11.0887, 10.322) and (11.0732, 10.488) .. (10.9904, 10.621).. controls (10.94, 10.7019) and (10.8647, 10.7673) .. (10.7767, 10.8037).. controls (10.6886, 10.84) and (10.5911, 10.8447) .. (10.4958, 10.8447) -- (10.2547, 10.8447);

  \path[draw=blue,line width=0.02cm,dash pattern=on 0.02cm off 0.02cm] (10.2547, 10.9779) -- (10.5552, 10.9779).. controls (10.6722, 10.9796) and (10.7895, 10.9516) .. (10.8931, 10.8973).. controls (10.9691, 10.8575) and (11.0383, 10.803) .. (11.0893, 10.734).. controls (11.1329, 10.675) and (11.1623, 10.6067) .. (11.1835, 10.5365).. controls (11.2036, 10.4695) and (11.2165, 10.4003) .. (11.2208, 10.3305).. controls (11.2264, 10.2389) and (11.2171, 10.1465) .. (11.1944, 10.0576).. controls (11.1702, 9.9627) and (11.1302, 9.8709) .. (11.0703, 9.7934).. controls (11.0137, 9.7202) and (10.9382, 9.6609) .. (10.8518, 9.6277).. controls (10.7654, 9.5946) and (10.671, 9.5901) .. (10.5785, 9.5902) -- (8.6072, 9.5909).. controls (8.5229, 9.5909) and (8.4378, 9.5925) .. (8.356, 9.6128).. controls (8.2742, 9.6331) and (8.1962, 9.6719) .. (8.1347, 9.7295).. controls (8.0551, 9.8041) and (8.0048, 9.9096) .. (7.9982, 10.0186).. controls (7.9916, 10.1276) and (8.0289, 10.2387) .. (8.1006, 10.321).. controls (8.1759, 10.4073) and (8.2883, 10.4601) .. (8.4027, 10.463) -- (9.5192, 10.4632).. controls (9.5337, 10.4632) and (9.5482, 10.4628) .. (9.5624, 10.4601).. controls (9.5767, 10.4574) and (9.5904, 10.4527) .. (9.6038, 10.447).. controls (9.6355, 10.4336) and (9.6657, 10.4149) .. (9.687, 10.3878).. controls (9.7109, 10.3576) and (9.7221, 10.3176) .. (9.7174, 10.2793) -- (9.7174, 10.1373);

  \path[draw=cd40000,line width=0.02cm,dash pattern=on 0.02cm off 0.02cm] (9.7174, 10.1373) -- (9.7174, 10.0878).. controls (9.7201, 10.0655) and (9.7289, 10.0441) .. (9.7425, 10.0263).. controls (9.7617, 10.0012) and (9.7898, 9.9839) .. (9.8197, 9.9736).. controls (9.8495, 9.9632) and (9.8812, 9.9592) .. (9.9128, 9.9577).. controls (9.9759, 9.9547) and (10.0392, 9.9585) .. (10.1024, 9.9588) -- (10.527, 9.9604).. controls (10.552, 9.9605) and (10.5771, 9.9609) .. (10.6018, 9.9651).. controls (10.6264, 9.9694) and (10.6505, 9.9776) .. (10.6722, 9.99).. controls (10.7156, 10.0149) and (10.7475, 10.0564) .. (10.7694, 10.1014).. controls (10.816, 10.1967) and (10.8255, 10.31) .. (10.7932, 10.4111).. controls (10.7843, 10.4388) and (10.7724, 10.4656) .. (10.756, 10.4896).. controls (10.7397, 10.5136) and (10.7188, 10.5346) .. (10.6944, 10.5504).. controls (10.67, 10.5662) and (10.6423, 10.5766) .. (10.6138, 10.582).. controls (10.5852, 10.5874) and (10.556, 10.5877) .. (10.5269, 10.5877) -- (8.2906, 10.5884);

%% file: figs/alt_length_5.tex
\definecolor{ce6e6e6}{RGB}{230,230,230}
\definecolor{c819d43}{RGB}{129,157,67}
\definecolor{c484537}{RGB}{72,69,55}
\definecolor{c979797}{RGB}{151,151,151}

  \path[fill=ce6e6e6,line width=0.02cm,dash pattern=on 0.02cm off 0.06cm] (7.4341, 9.5107).. controls (7.4139, 9.5123) and (7.3872, 9.5165) .. (7.3748, 9.52).. controls (7.2858, 9.5456) and (7.2005, 9.6456) .. (7.161, 9.7706).. controls (7.1083, 9.9374) and (7.1542, 10.134) .. (7.2734, 10.2531).. controls (7.3302, 10.3098) and (7.3734, 10.3256) .. (7.4713, 10.3256) -- (7.53, 10.3256) -- (7.53, 10.2285) -- (7.53, 10.1313) -- (7.5205, 10.1289).. controls (7.5071, 10.1256) and (7.4839, 10.1042) .. (7.4777, 10.0895).. controls (7.4598, 10.0466) and (7.4868, 9.9979) .. (7.5326, 9.9906).. controls (7.5732, 9.9841) and (7.6143, 10.0194) .. (7.6141, 10.0605).. controls (7.6139, 10.0904) and (7.5902, 10.1229) .. (7.5645, 10.1286) -- (7.5529, 10.1311) -- (7.5529, 10.2284) -- (7.5529, 10.3256) -- (7.5845, 10.3256).. controls (7.6344, 10.3256) and (7.7389, 10.3183) .. (7.7787, 10.312).. controls (7.8277, 10.3043) and (7.8907, 10.2871) .. (7.9394, 10.2682).. controls (7.9615, 10.2596) and (7.9842, 10.2507) .. (7.99, 10.2486) -- (8.0004, 10.2446) -- (7.9928, 10.2298).. controls (7.9887, 10.2216) and (7.9853, 10.2139) .. (7.9853, 10.2125).. controls (7.9853, 10.2111) and (8.0022, 10.2112) .. (8.0228, 10.2127) -- (8.0604, 10.2154) -- (8.2132, 10.139).. controls (8.2972, 10.0969) and (8.4377, 10.0254) .. (8.5254, 9.98).. controls (8.8041, 9.8357) and (8.9138, 9.7825) .. (9.0317, 9.7346).. controls (9.1691, 9.6788) and (9.2272, 9.6694) .. (9.4166, 9.6723).. controls (9.5289, 9.6741) and (9.5344, 9.6749) .. (9.5784, 9.6965).. controls (9.6498, 9.7314) and (9.7162, 9.835) .. (9.7439, 9.9545).. controls (9.7569, 10.0107) and (9.7599, 10.0873) .. (9.7514, 10.1449).. controls (9.7303, 10.2878) and (9.6516, 10.4313) .. (9.5709, 10.4738).. controls (9.5209, 10.5002) and (9.5393, 10.499) .. (9.1551, 10.5007) -- (8.8079, 10.5022) -- (8.8079, 10.5191).. controls (8.8079, 10.5285) and (8.8061, 10.5361) .. (8.8039, 10.5361).. controls (8.8017, 10.5361) and (8.7867, 10.5283) .. (8.7705, 10.5188) -- (8.7411, 10.5016) -- (8.147, 10.5016) -- (7.5529, 10.5016) -- (7.5529, 10.5648) -- (7.5529, 10.6279) -- (7.7268, 10.6279) -- (7.9007, 10.6279) -- (7.9019, 10.6088) -- (7.903, 10.5897) -- (7.9366, 10.6089) -- (7.9702, 10.6281) -- (8.7458, 10.6268).. controls (9.5984, 10.6254) and (9.543, 10.6269) .. (9.6035, 10.6028).. controls (9.7138, 10.5586) and (9.8051, 10.4335) .. (9.8428, 10.2748).. controls (9.8583, 10.2094) and (9.8643, 10.1592) .. (9.8666, 10.0731).. controls (9.8698, 9.9512) and (9.8604, 9.8737) .. (9.8316, 9.7878).. controls (9.7912, 9.6671) and (9.7088, 9.5678) .. (9.6184, 9.5311).. controls (9.562, 9.5082) and (9.5947, 9.5093) .. (8.9504, 9.5079) -- (8.3602, 9.5066) -- (8.3602, 9.5258).. controls (8.3602, 9.5364) and (8.3592, 9.5451) .. (8.3581, 9.5451).. controls (8.3569, 9.5451) and (8.3413, 9.5365) .. (8.3235, 9.526) -- (8.2911, 9.5068) -- (7.8809, 9.5074).. controls (7.6553, 9.5077) and (7.4542, 9.5092) .. (7.4341, 9.5108) -- cycle;

  \path[draw=black,fill=black,line width=0.023cm,cm={ 0.1361,-0.5591,-0.5591,-0.1361,(20.9086, 16.6241)}] (5.4435, 22.9858) -- (5.3196, 22.9566) -- (5.3563, 23.0785) -- cycle;

  \path[draw=c819d43,line width=0.02cm] (7.5428, 11.1534) -- (7.5428, 10.0599);

  \node[text=black,line width=0.005cm,anchor=south west] (text17) at (7.2341, 11.22){$\scriptscriptstyle\lambda_{i}$};

  \path[draw=c819d43,line width=0.02cm] (9.342, 11.1534) -- (9.342, 10.061);

  \path[draw=black,fill=c484537,line width=0.015cm,dash pattern=on 0.015cm off 0.06cm,shift={(-1.1452, -0.1513)}] (8.6811, 10.2112) -- (10.4872, 10.2123);

  \path[draw=black,fill=c979797,line width=0.0487cm] (7.5428, 10.061) circle (0.0458cm);

  \path[draw=black,fill=c979797,line width=0.0487cm] (9.342, 10.061) circle (0.0458cm);

  \node[text=black,line width=0.005cm,anchor=south west] (text57948-7) at (9.0333, 11.22){$\scriptscriptstyle\lambda_{j}$};

  \path[draw=black,line width=0.02cm,shift={(-1.1452, -0.1513)}] (8.4088, 10.7894) -- (10.5871, 10.7897).. controls (10.6421, 10.7897) and (10.6982, 10.7867) .. (10.7495, 10.7669).. controls (10.8008, 10.747) and (10.8454, 10.7122) .. (10.8814, 10.6706).. controls (10.9321, 10.6121) and (10.9673, 10.5411) .. (10.9895, 10.4669).. controls (11.0117, 10.3928) and (11.0213, 10.3152) .. (11.0244, 10.2379).. controls (11.0279, 10.1506) and (11.0231, 10.0624) .. (11.0012, 9.9779).. controls (10.9792, 9.8933) and (10.9394, 9.8124) .. (10.879, 9.7493).. controls (10.8445, 9.7133) and (10.8027, 9.6836) .. (10.7557, 9.6668).. controls (10.7087, 9.65) and (10.6579, 9.6476) .. (10.608, 9.6476) -- (8.6287, 9.6482).. controls (8.5796, 9.6482) and (8.5292, 9.6515) .. (8.4845, 9.6718).. controls (8.4397, 9.6921) and (8.4028, 9.7268) .. (8.3731, 9.7659).. controls (8.3076, 9.8522) and (8.2718, 9.9605) .. (8.2721, 10.0688).. controls (8.2723, 10.1771) and (8.3085, 10.2849) .. (8.3731, 10.3719).. controls (8.4025, 10.4115) and (8.4391, 10.4469) .. (8.4839, 10.4673).. controls (8.5064, 10.4775) and (8.5304, 10.484) .. (8.5549, 10.4871).. controls (8.5793, 10.4902) and (8.604, 10.4901) .. (8.6287, 10.4897) -- (8.688, 10.4887).. controls (8.7817, 10.4871) and (8.8762, 10.4869) .. (8.9679, 10.4675).. controls (9.0703, 10.4458) and (9.1658, 10.3997) .. (9.2596, 10.3533).. controls (9.3559, 10.3057) and (9.4517, 10.2573) .. (9.5472, 10.2081).. controls (9.6061, 10.1777) and (9.665, 10.1473) .. (9.7238, 10.1168).. controls (9.8785, 10.0366) and (10.0333, 9.9557) .. (10.195, 9.8907).. controls (10.2395, 9.8728) and (10.2848, 9.8564) .. (10.3318, 9.8468).. controls (10.3788, 9.8372) and (10.427, 9.8352) .. (10.4749, 9.8348) -- (10.5871, 9.8338).. controls (10.6093, 9.8336) and (10.6315, 9.8336) .. (10.6534, 9.8369).. controls (10.6753, 9.8401) and (10.6967, 9.8465) .. (10.7165, 9.8565).. controls (10.756, 9.8764) and (10.787, 9.9102) .. (10.8111, 9.9473).. controls (10.8667, 10.0326) and (10.8937, 10.1356) .. (10.8905, 10.2373).. controls (10.8873, 10.3391) and (10.8544, 10.4392) .. (10.7997, 10.5251).. controls (10.7768, 10.561) and (10.7486, 10.5949) .. (10.7111, 10.6153).. controls (10.6924, 10.6255) and (10.6719, 10.6323) .. (10.6509, 10.6358).. controls (10.6299, 10.6393) and (10.6085, 10.6395) .. (10.5871, 10.6395) -- (8.4088, 10.6401);

  \node[text=black,line width=0.005cm,anchor=south west] (text2) at (7.9749, 9.5895){$\scriptscriptstyle R$};

  \path[draw=black,fill=black,line width=0.023cm,cm={ 0.5511,-0.1657,-0.1657,-0.5511,(8.7806, 24.2076)}] (5.4435, 22.9858) -- (5.3196, 22.9566) -- (5.3563, 23.0785) -- cycle;

  \path[draw=black,fill=black,line width=0.023cm,cm={ -0.5511,-0.1657,0.1657,-0.5511,(7.4792, 23.0654)}] (5.4435, 22.9858) -- (5.3196, 22.9566) -- (5.3563, 23.0785) -- cycle;

  \path[draw=black,fill=black,line width=0.023cm,cm={ 0.5645,-0.1121,0.1121,0.5645,(2.4208, -2.1432)}] (5.4435, 22.9858) -- (5.3196, 22.9566) -- (5.3563, 23.0785) -- cycle;

  \node[text=black,line width=0.005cm,anchor=south west] (text3) at (7.4449, 9.7745){$\scriptscriptstyle \nPoint_i$};

  \node[text=black,line width=0.005cm,anchor=south west] (text4) at (9.3104, 9.7745){$\scriptscriptstyle \nPoint_j$};

%% file: sections/lemma/lemma_main.tex
\section{A lemma controlling certain Coxeter words}
\label{sec:coxeter_word_reduction_lemma}
\input{sections/lemma/preamble}
\input{sections/lemma/word_problem_in_cox_groups}
\input{sections/lemma/property_A}
\input{sections/lemma/property_B}

%% file: sections/lemma/preamble.tex
This section concerns controlling word operations which relate certain Coxeter words representing the same group element.
This will be applied to words from $Q$.

We use the terms \emph{Coxeter word} and \emph{free word} to specify words belonging to either a Coxeter group or free group.
Since Coxeter generators are involutions, we will only consider positive Coxeter words.
Given a Coxeter graph $\Gamma$, if the word operations relevant to a Coxeter word $w$ come from the group $\nCoxG$, then we will introduce $w$ as a \emph{Coxeter word in $\nCoxG$}.
Similarly to \Cref{sec:Q_restrictions}, an \emph{alternating} Coxeter word alternates between two generators, the term \emph{$(\nCoxGenElt_i,\nCoxGenElt_j)$-alternating word} specifies the two generators, and a \emph{maximal alternating subword} is not contained in any longer alternating subword.
We may refer to the length of an alternating subword relative to $m_{ij}$.
In this case, $i$ and $j$ should be understood to take values such that the subword is an $(\nCoxGenElt_i,\nCoxGenElt_j)$-alternating subword.

%% file: sections/lemma/word_problem_in_cox_groups.tex
\subsection{The word problem in Coxeter groups}
We begin by stating the solution to the word problem for Coxeter groups.
We then define toggling word operations and make some basic observations about them.
\Cref{thm:tits_algorithm} is a classical result of Tits.
It and \Cref{def:M_operations} are covered in \cite[Section 3.4]{davis_geometry_2025}.
\begin{definition}
	\label{def:M_operations}
	Let $\varepsilon$ denote the empty word.
	Consider the following word operations on Coxeter words.
	\begin{enumerate}
		\item $\nCoxGenElt^2_i \leadsto \varepsilon $.
		      \smallskip
		\item $\nPi{\nCoxGenElt_i}{\nCoxGenElt_j}{m_{ij}} \leadsto \nPi{\nCoxGenElt_j}{\nCoxGenElt_i}{m_{ij}} \qquad (m_{ij} \neq \infty)$.
	\end{enumerate}
	We call operations of type $(2)$ \emph{Artin operations}.
\end{definition}
We say a Coxeter word is \emph{reduced} if it is of minimum length among all words representing the same group element.
\begin{theorem}
	\label{thm:tits_algorithm}
	Let $\Gamma$ be a Coxeter graph.
	\begin{enumerate}
		\item A Coxeter word in $\nCoxG$ is reduced if and only if it is of minimum length with respect to the operations defined in \Cref{def:M_operations}.
		      \smallskip
		\item Any two reduced Coxeter words which correspond to the same group element in $\nCoxG$ are related by a sequence of Artin operations.
	\end{enumerate}
\end{theorem}
\begin{definition}
	\label{def:toggling_operation}
	Let $w$ be a Coxeter word.
	We say that a word $w'$ is obtained from~$w$ by a \emph{toggling operation} if $w'$ is obtained from $w$ by replacing some $\nPi{\nCoxGenElt_i}{\nCoxGenElt_j}{m_{ij} + k}$ subword by $\nPi{\nCoxGenElt_j}{\nCoxGenElt_i}{m_{ij} -k}$, for some integer $k$ with $\Abs{k} \leq m_{ij}$, and $m_{ij} < \infty$.
\end{definition}
We may say that a subword is \emph{toggled} if a toggling operation is applied.
We note that Artin operations are toggling operations, so we may refer to their application as \emph{toggling} if this is obvious from context.
\begin{lemma}
	\label{lem:square_free_maximal_is_conserved}
	Let $w$ be a square-free Coxeter word.
	Let $u$ be a maximal alternating subword of $w$ which is toggled to $u^\prime$, resulting in the word $w^\prime$.
	Then $w^\prime$ is square-free and $u^\prime$ is a maximal alternating subword of $w^\prime$.
\end{lemma}
\begin{proof}
	We can assume $u$ is a $(\nCoxGenElt_i,\nCoxGenElt_j)$-alternating word.
	It is sufficient to prove that:
	\begin{enumerate}
		\item The first letter of $u^\prime$ and the preceding letter in  $w^\prime$ do not make a square.
		      Similarly, the last letter of $u^\prime$ and the following letter in $w^\prime$ do not make a square.
		      \smallskip
		\item If the first letter of $u^\prime$ is $\nCoxGenElt_i$, then the preceding letter in  $w^\prime$ is not  $\nCoxGenElt_j$.
		      If the first letter of $u^\prime$ is $\nCoxGenElt_j$, then the preceding letter in $w^\prime$ is not $\nCoxGenElt_i$.
		      Similarly, for the last letter of $u^\prime$ and the following letter in $w^\prime$.
	\end{enumerate}
	Any contradiction to (1) would contradict that $u$ was a maximal alternating subword.
	Any contradiction to (2) would contradict that $w$ was square-free.
\end{proof}
\begin{remark}
	The above lemma does not say that \emph{all} maximal alternating subwords of $w$ are maximal alternating subwords of $w^\prime$.
	It is possible that another alternating subword interacts with the first or last letter of $u^\prime$.
\end{remark}
\begin{remark}
	\label{rmk:maximal_implies_overlap_at_most_1_and_share_exactly_1_letter}
	Given a Coxeter word $w$ and a maximal alternating subword $u$ of $w$, if there is another alternating subword $u^\prime$ which overlaps but is not a subword of $u$, then $u$ and $u^\prime$ overlap over one letter and have exactly one generator in common.
\end{remark}

%% file: sections/lemma/property_A.tex
\subsection{Sufficient conditions for a Coxeter word to be reduced}
Given a square-free Coxeter word $w$, if we can control what Artin operations do to $w$ without introducing squares, then \Cref{thm:tits_algorithm} tells us that $w$ is reduced.
We will consider the case where all $m_{ij}\geq 4$, so we will need to worry about the application of Artin operations creating new alternating subwords of length $4$.
There are two situations where this can occur:
\begin{enumerate}
	\item An $(\nCoxGenElt_i,\nCoxGenElt_j)$-alternating subword $h$ with $\ell(h) = m_{ij}$ is adjacent to an $(\nCoxGenElt_i,\nCoxGenElt_k)$-alternating subword $u$ with $\ell(u) = 3$.
	      In this case, toggling $h$ might create a new $(\nCoxGenElt_i,\nCoxGenElt_k)$-alternating subword $u^\prime$ with $\ell(u^\prime)=4$.
	      \smallskip
	\item An $(\nCoxGenElt_i,\nCoxGenElt_j)$-alternating subword $h$ with $\ell(h) = m_{ij}$ is immediately followed by a subword $u = \nCoxGenElt_k\nCoxGenElt_i$, which is then immediately followed by an $(\nCoxGenElt_k,\nCoxGenElt_l)$-alternating subword $h^\prime$ with $\ell(h^\prime) = m_{kl}$.
	      In this case, some combination of toggling $h$ and $h^\prime$ might create an alternating subword $u^\prime = \nCoxGenElt_i\nCoxGenElt_k\nCoxGenElt_i\nCoxGenElt_k$ which overlaps the last letter of $h$ and the first letter of $h^\prime$.
\end{enumerate}
In situation (1) we say that~$u$ is a \emph{3-neighbour} of~$h$, and in situation (2) we say that~$u$ is a \emph{2-bridge} connecting~$h$ and~$h^\prime$.
\crefname{enumi}{property}{properties}
\Crefname{enumi}{Property}{Properties}
\begin{definition}
	\label{def:factorisation_prop_F}
	Suppose $w$ is a square-free Coxeter word, and let $h_1,\ldots,h_m$ be any collection of maximal alternating subwords of $w$ with all $\ell(h_k) \geq 2$.
	We say that $h_1,\ldots,h_m$ have \emph{property~$\nFactPropF$} if:
	\begin{enumerate}[label={$\nFactPropF$}.\arabic*, ref={$\nFactPropF$}.\arabic*]
		\item Every pair $h_i,h_j$ has either $0$ or $2$ generators in common.\label{prop:F_0_or_2_gens_in_common}
		      \smallskip
		\item No  $h_k$ has a $3$-neighbour.\label{prop:F_3_neighbour}
		      \smallskip
		\item No pair $h_i,h_j$ is connected by a 2-bridge.\label{prop:F_2_bridge}
	\end{enumerate}
\end{definition}
\begin{remark}
	\label{rmk:F_a_subwords_disjoint}
	Since the $h_1,\ldots,h_m$ subwords are assumed to be maximal, by \Cref{rmk:maximal_implies_overlap_at_most_1_and_share_exactly_1_letter} and \cref{prop:F_0_or_2_gens_in_common}, these subwords are also all disjoint.
\end{remark}
\begin{remark}
	\label{rmk:A_l_emptyword_implies_adjacent_as_have_0_gens_in_common}
	Given a square-free Coxeter word $w$ and a collection of maximal alternating subwords $h_1,\ldots,h_m$ with property $\nFactPropF$, if any pair $h_i,h_j$ is immediately adjacent in $w$, then $h_i$ and $h_j$ must have 0 generators in common.
	This follows from \cref{prop:F_0_or_2_gens_in_common}.
\end{remark}
\begin{remark}
	\label{rmk:subset_preserves_F}
	Given a Coxeter word $w$ and a collection of maximal alternating subwords $h_1,\ldots,h_m$ with property $\nFactPropF$, any subset of alternating subwords $A \subseteq \Set{h_1,\ldots,h_m}$ also has property $\nFactPropF$.
\end{remark}
\begin{lemma}
	\label{lem:fact_prop_F_adjacent_3_overlapping_4}
	Let $\Gamma$ be a Coxeter graph and let $w$ be a square-free Coxeter word in~$\nCoxG$.
	Suppose the collection of maximal alternating subwords $h_1,\ldots,h_m$ has property~$\nFactPropF$.
	Then the following hold for subwords of $w$, which need not respect the $h_k$ subwords.
	\begin{enumerate}
		\item If $u$ is an alternating, length $\geq 3$ subword of $w$  which is adjacent to any $h_k$ subword, then $u$ and $h_k$ have 0 generators in common.
		      \smallskip
		\item If any alternating subword of $w$ of length $\geq 4$ overlaps some $h_k$, then it is a subword of $h_k$.
	\end{enumerate}
\end{lemma}
\begin{proof}
	For (1), assume such a $u$ shares a generator with $h_k$.
	If~$u$ does not overlap any other~$h_{k^\prime}$ factor, then~$u$ is a~$3$-neighbour of~$h_k$, contradicting \cref{prop:F_3_neighbour}.
	If~$u$ does overlap some other factor~$h_{k^\prime}$, then by \Cref{rmk:maximal_implies_overlap_at_most_1_and_share_exactly_1_letter}, it overlaps over exactly one letter.
	The remaining~$2$ letters of~$u$ form a $2$-bridge between $h_k$ and $h_{k^\prime}$, contradicting \cref{prop:F_2_bridge}.

	For (2), suppose the contrary.
	Let~$u$ be such a subword of $w$ with $\ell(u) \geq 4$, such that $u$ overlaps, but is not a subword of, some $h_k$ subword.
	By \Cref{rmk:maximal_implies_overlap_at_most_1_and_share_exactly_1_letter} again,~$u$ overlaps~$h_k$ over exactly one letter.
	The remaining $\geq 3$ letters of~$u$ form an alternating subword of $w$ which is adjacent to~$h_k$ and has 1 generator in common with~$h_k$, contradicting (1).
\end{proof}
For legibility, it will be useful to define a shorthand to describe certain subwords of Coxeter words.
We use $a,b,c,d$ to denote arbitrary pairwise distinct letters.
A \emph{pattern word} is a word in these symbols, with repeated occurrences of a symbol representing repeated occurrences of the same letter.
If there are two pattern words in the same context which use the same symbol, then those symbols should correspond to the same generator.
For example, $\nPatrn{aba}$ and $\nPatrn{bca}$ occur in $w$ if and only if there are pairwise distinct $i,j,k$ such that $\nCoxGenElt_i\nCoxGenElt_j\nCoxGenElt_i$ and $\nCoxGenElt_j\nCoxGenElt_k\nCoxGenElt_i$ are both subwords of $w$.

It will also be useful to name some of the ways that alternating subwords can change due to toggles on maximal alternating subwords.
Let~$h$ be some maximal $(\nCoxGenElt_i,\nCoxGenElt_j)$-alternating subword of $w$.
Let~$w^\prime$ be the result of toggling~$h$ to~$h^\prime$.
Suppose~$u$ is a maximal alternating subword of~$w$ which is disjoint from~$h$, and that the letters corresponding to~$u$ in~$w^\prime$ comprise a maximal alternating subword of~$w^\prime$.
In this case, we say that~$u$ is \emph{transported} to~$w^\prime$.
The important features of~$u$ are unchanged.

Suppose instead that the boundary of $u$ interacts with the toggling of $h$ in a way that results in a new, longer alternating subword $u^\prime$, as in the pattern words of the following example.
\newlength{\overlap}
\settowidth{\overlap}{$b)$}
\begin{equation*}
	\nPatrn{\underbrace{(ababa)}_{h}\underbrace{(cbc)}_{u}} \quad \rightsquigarrow \quad \nPatrn{\overbrace{baba\phantom{(b}}^{h^\prime}\hspace{-\overlap}\underbrace{(bcbc)}_{u^\prime}}\, .
\end{equation*}
In this case, we say that $u^\prime$ was \emph{born} as a \emph{passenger} of toggling $h$.
If instead we look at the above example as a toggle starting from the right-hand side, we would say that $u^\prime$ was \emph{killed} as a result of toggling $h^\prime$.

We do not use the terms born or killed in the case where an alternating subword is changed because it itself is toggled.
In this case, we would say that the subword is transported.
\begin{lemma}
	\label{lem:G_preserved_and_4s_preserved}
	Let $w$ be a square-free Coxeter word and let $h_1,\ldots,h_m$ be a collection of maximal alternating subwords of $w$ with property $\nFactPropF$.
	Apply toggling operations to any number of the $h_k$ subwords, resulting in the collection $h^\prime_1,\ldots,h^\prime_m$, all subwords of the resulting $w^\prime$.
	Then:
	\begin{enumerate}
		\item The word~$w^\prime$ is square-free, each of $h^\prime_1,\ldots,h^\prime_m$ is maximal in~$w^\prime$, and this collection has property~$\nFactPropF$.
		\item No alternating subwords~$u$ with $\ell(u) \geq 4$ are born or killed as a result of the toggling.
	\end{enumerate}
\end{lemma}
\begin{proof}
	For~(1), we claim that after toggling some $h_k$ to $h^\prime_k$, the resulting word $w^\prime$ is square-free and the factors $h_1,\ldots,h^\prime_k,\ldots,h_m$ have property $\nFactPropF$.
	Then,~(1) follows by induction.

	We prove the claim.
	We first note that by \Cref{lem:square_free_maximal_is_conserved},~$h^\prime_k$ is maximal in~$w^\prime$, which is square-free.
	Of the remaining $\Set{h_1,\ldots,h_m}\setminus \Set{h_k}$, if any factor~$h_i$ was immediately adjacent to~$h_k$, then by \Cref{rmk:A_l_emptyword_implies_adjacent_as_have_0_gens_in_common},~$h_i$ would share $0$ generators with~$h_k$, so~$h_i$ would still be maximal in~$w^\prime$.
	We have shown that each of $h_1,\ldots,h_k^\prime,\ldots,h_m$ is maximal in~$w^\prime$.
	\Cref{prop:F_0_or_2_gens_in_common,prop:F_3_neighbour,prop:F_2_bridge} only concern which generators appear in the relevant subwords and what surrounds them.
	This has not changed, so these properties trivially hold for this collection.

	For~$(2)$, by symmetry, it is sufficient to show that after toggling one $h_k$ to $h_k^\prime$, no such $u$ is born.
	Note that by~$(1)$, the resulting collection $h_1,\ldots,h^\prime_k,\ldots,h_m$ has property $\nFactPropF$.
	If such a $u$ was born, it would need to overlap $h^\prime_k$ over at least one letter, but not be a subword of $h^\prime_k$, contradicting \Cref{lem:fact_prop_F_adjacent_3_overlapping_4}.
\end{proof}
\begin{definition}
	\label{def:word_prop_A}
	Let $\Gamma$ be a Coxeter graph with all $m_{ij} \geq 4$.
	Let $w$ be a square-free Coxeter word in $W_\Gamma$, and let $h_1,\ldots,h_m$ be all alternating subwords of $w$ of length at least~$m_{ij}$, for all $i$ and $j$.
	We say $w$ \emph{has property $\nWordPropA$} if $h_1,\ldots,h_m$ are all maximal and have property $\nFactPropF$.
\end{definition}
\begin{lemma}
	\label{lem:prop_A_is_reduced_and_control_on_reduced_forms}
	Let $\Gamma$ be a Coxeter graph with all $m_{ij} \geq 4$.
	Let $w$ be a Coxeter word in $\nCoxG$ having property $\nWordPropA$, with $h_1,\ldots,h_m$ all alternating subwords of length at least $m_{ij}$.
	The following hold:
	\begin{enumerate}
		\item The Coxeter word $w$ is reduced.
		      \smallskip
		\item The set of reduced Coxeter words for this group element has cardinality $2^m$.
		      Each of these reduced words has property $\nWordPropA$.
		      We get any of these words by applying (or not) Artin operations to the $h_k$ subwords (in any order, or simultaneously).
		      \smallskip
		\item All alternating subwords of length $\geq 4$ in $w$ exist in all other reduced Coxeter words for $w$ at the same position as they occur in $w$.
	\end{enumerate}
\end{lemma}
\begin{proof}
	These all follow from \Cref{thm:tits_algorithm}, \Cref{lem:G_preserved_and_4s_preserved} and the fact that reducing Coxeter words requires the elimination of squares.
\end{proof}

%% file: sections/lemma/property_B.tex
\subsection{Controlling certain non-reduced Coxeter words}
\Cref{lem:prop_A_is_reduced_and_control_on_reduced_forms} gives good control over Coxeter words which have property $\nWordPropA$ when all $m_{ij} \geq 4$.
However, property $\nWordPropA$ will turn out to be too restrictive.
We will modify its definition and allow length $m_{ij} + 1$ alternating subwords.
Such subwords are always reducible, as in the following example where $m_{ij} = 4$.
\[
	(\nCoxGenElt_i\nCoxGenElt_j\nCoxGenElt_i\nCoxGenElt_j)\nCoxGenElt_i \rightsquigarrow
	(\nCoxGenElt_j\nCoxGenElt_i\nCoxGenElt_j\nCoxGenElt_i)\nCoxGenElt_i \rightsquigarrow
	\nCoxGenElt_j\nCoxGenElt_i\nCoxGenElt_j.
\]
Therefore, we cannot expect such words to be reduced.
However, we will still be able to control them quite well.

\begin{definition}
	\label{def:word_prop_B}
	A square-free Coxeter word $w$ has property $\nWordPropB$ if the following hold.
	\begin{enumerate}[label={$\nWordPropB$}.\arabic*, ref={$\nWordPropB$}.\arabic*]
		\item For odd $m_{ij}$, every $(\nCoxGenElt_i,\nCoxGenElt_j)$-alternating subword $h$ has $\ell(h) \leq m_{ij}$.\label{prop:B_odd_length}
		      \smallskip
		\item For even $m_{ij}$, every $(\nCoxGenElt_i,\nCoxGenElt_j)$-alternating subword $h$ has $\ell(h) \leq m_{ij} + 1$.
		      Moreover, if such an $h$ with $\ell(h) = m_{ij} + 1$ exists, then there must be a maximal $(\nCoxGenElt_i,\nCoxGenElt_j)$-alternating subword~$u$ elsewhere in~$w$, with $\ell(u) = m_{ij} -1$.
		      We say that~$u$ is a \emph{licence} for~$h$.
		      \label{prop:B_even_length_and_licence}
		      \smallskip
		\item Given any alternating subword $h$, with $\ell(h) = m_{ij} + 1$ as above, its licence~$u$ begins with the same generator as~$h$.
		      \label{prop:B_phase_coherence}
		      \smallskip
		\item \emph{(3-coherence)}
		      If the pattern word $\nPatrn{aba}$ occurs in $w$, then $\nPatrn{cbc}$ does not occur in $w$.
		      \label{prop:B_3_coherence}
		      \smallskip
		\item If the pattern word $\nPatrn{bacb}$ occurs in $w$, then neither $\nPatrn{dcd}$ nor  $\nPatrn{dad}$ occurs in $w$.\label{prop:B_pattern_exclusion}
	\end{enumerate}
\end{definition}
\begin{definition}
	\label{def:pm_reduction_expansion}
	Consider these toggling operations (see \Cref{def:toggling_operation}).
	\begin{align*}
		\nPi{\nCoxGenElt_i}{\nCoxGenElt_j}{m_{ij} + 1} & \leadsto \nPi{\nCoxGenElt_j}{\nCoxGenElt_i}{m_{ij} - 1}, \\
		\nPi{\nCoxGenElt_i}{\nCoxGenElt_j}{m_{ij} - 1} & \leadsto \nPi{\nCoxGenElt_j}{\nCoxGenElt_i}{m_{ij} + 1}.
	\end{align*}
	We call operations of the first type \emph{$(m_{ij} +1)$-reductions}.
	Similarly, operations of the second type are called \emph{$(m_{ij} -1)$-expansions}.
\end{definition}
Suppose $\Gamma$ is a Coxeter graph with all $m_{ij} \geq 4$ and $w$ is a Coxeter word in $\nCoxG$ which has property $\nWordPropB$.
We could apply an $(m_{ij} +1)$-reduction to each maximal alternating subword of length $m_{ij} + 1$.
We call the resulting word, denoted $w^\prime$ in the following lemma, the \emph{obvious reduction of $w$}.
\begin{lemma}
	\label{lem:B_obvious_reduction_is_A}
	Let $\Gamma$ be a Coxeter graph with all $m_{ij} \geq 4$.
	Let $w$ be a square-free Coxeter word with property~$\nWordPropB$ in $\nCoxG$.
	Then its obvious reduction $w^\prime$ has property~$\nWordPropA$, and no maximal alternating subwords $u$ with $\ell(u) \geq 4$ are born or killed in this toggling.
\end{lemma}
\begin{proof}
	Let $h_1,\ldots,h_m$ be the maximal alternating subwords of $w$ of length at least $m_{ij}$.
	We will begin by showing that this collection has property~$\nFactPropF$.
	By $3$-coherence, we get that every pair $h_i,h_j$ has~$0$ or~$2$ generators in common, and no $h_k$ has a~$3$-neighbour, giving \cref{prop:F_0_or_2_gens_in_common,prop:F_3_neighbour}.
	Now assume that two $h_i,h_j$ are connected by a $2$-bridge.
	This would necessitate one of the following pattern words in~$w$, with bracketing for readability.
	\begin{enumerate}
		\item $\nPatrn{baba(cb)dcdc}$
		\item $\nPatrn{abab(cb)dcdc}$
		\item $\nPatrn{baba(cb)cdcd}$
		\item $\nPatrn{abab(cb)cdcd}$
		\item $\nPatrn{baba(cb)dbdb}$
		\item $\nPatrn{baba(cb)dada}$
	\end{enumerate}
	Only option $(1)$ is~$3$-coherent, but that option simultaneously contains the patterns $\nPatrn{bacb}$ and $\nPatrn{dcd}$, contradicting \cref{prop:B_pattern_exclusion}.

	We have shown that $h_1,\ldots,h_m$ have property $\nFactPropF$.
	Then by \Cref{lem:G_preserved_and_4s_preserved}, the obvious reduction $w^\prime$ is square-free, the collection of resulting subwords $h^\prime_1,\ldots,h^\prime_m$ has property $\nFactPropF$, and no alternating $u$ with $\ell(u) \geq 4$ is born or killed as a result of these toggles.

	Some of $h^\prime_1,\ldots,h^\prime_m$ may be of length $m_{ij} - 1$, being the result of an $(m_{ij} + 1)$-reduction.
	By \Cref{rmk:subset_preserves_F}, the subset of $h^\prime_1,\ldots,h^\prime_m$ which are of length $m_{ij}$ still has property~$\nFactPropF$.
	Every alternating subword of~$w^\prime$ of length~$m_{ij}$ must be in this subset of $h^\prime_1,\ldots,h^\prime_m$.
	So, we have shown $w^\prime$ has property~$\nWordPropA$.
\end{proof}
The lemma above and \Cref{lem:prop_A_is_reduced_and_control_on_reduced_forms} tell us that the obvious reduction $w^\prime$ of a square-free Coxeter word $w$ with property $\nWordPropB$ is in fact reduced.
Moreover, the length $m_{ij}$ alternating subwords of $w^\prime$ are those transported from $w$ and are disjoint, and the set of reduced words for $w$ is well controlled.
We reach all reduced Coxeter words corresponding to $w$ by simultaneously toggling (or not) a subset of the disjoint, length $m_{ij}$ alternating subwords of $w^\prime$.

So, given a pair $w_1,w_2$ of square-free Coxeter words with property $\nWordPropB$ which correspond to the same Coxeter group element, there exists the following three-step sequence of word operations which relate $w_1$ and $w_2$.
\begin{equation}
	\label{eqn:prop_C_word_operation_two_way_sequence}
	\begin{tikzcd}
		w_1 \ar[rrr, rightsquigarrow, "(m_{ij} +1)\text{-reductions}"', "\text{Step (i)}"]            &  &  &
		w^\prime_1 \ar[rrr, rightsquigarrow, "\text{Artin operations}"', "\text{Step (ii)}"]          &  &  &
		w^\prime_2 \ar[rrr, rightsquigarrow, "\:(m_{ij} -1)\text{-expansions}"', "\text{Step (iii)}"] &  &  &
		w_2.
	\end{tikzcd}
\end{equation}
Each of these steps is the simultaneous toggling of a subset of disjoint alternating subwords (of $w_1$, $w_1^\prime$ or $w_2^\prime$ as appropriate).

We now show that every alternating subword toggled at any step in~\eqref{eqn:prop_C_word_operation_two_way_sequence} is transported (rather than born as a passenger) from the previous step.
Thus, $w_2$ is reached by simultaneously toggling a set of disjoint alternating subwords of $w_1$.
\begin{lemma}
	\label{lem:alts_geq_4_are_transported}
	At no step in \eqref{eqn:prop_C_word_operation_two_way_sequence} is any alternating subword $x$ with $\ell(x) \geq 4$ born or killed.
	All such $x$ are transported.
\end{lemma}
\begin{proof}
	\Cref{lem:B_obvious_reduction_is_A} and \Cref{lem:prop_A_is_reduced_and_control_on_reduced_forms} tell us that no such~$x$ is born or killed at steps~$(i)$ and~$(ii)$, respectively.
	The reversal symmetry between steps~$(i)$ and~$(iii)$ tells us that no such $x$ is born or killed in step $(iii)$.
\end{proof}
\begin{lemma}
	\label{lem:all_toggled_3s_are_transported}
	In \eqref{eqn:prop_C_word_operation_two_way_sequence}, all maximal alternating subwords $x$ with $\ell(x) = 3$ which are toggled in step $(iii)$ are transported from $w_1$.
\end{lemma}
\begin{proof}
	It is sufficient to show that such an $x$ is not born at steps $(i)$ or $(ii)$.

	Suppose such an $x$ was born at step~$(ii)$.
	In that case there must be a maximal alternating subword~$u$ of $w_2^\prime$ with $\ell(u) \geq 4$, which is killed in step~$(iii)$ by the toggling of $x$.
	But by \Cref{lem:alts_geq_4_are_transported},~$u$ must be transported to~$w_2$, a contradiction.
	So,~$x$ is not born at step~$(ii)$.

	Suppose instead that such an $x$ is born at step $(i)$, and corresponds to the pattern word $\nPatrn{bcb}$ in $w_2^\prime$.
	This would require a pattern word $\nPatrn{ababa}$ in $w_1$, whose toggle overlaps $x$ at one of its outer letters, and a pattern word $\nPatrn{cbcbc}$ in $w_2$, which is the toggle of $x$, as in the following.
	\begin{equation*}
		\begin{tikzcd}
			\nPatrn{(ababa)cb} \ar[r, rightsquigarrow, "(i)"] & \nPatrn{ba\underbrace{(bcb)}_{x}} \ar[r, rightsquigarrow, "(iii)"] & \nPatrn{ba(cbcbc)}
		\end{tikzcd}
	\end{equation*}
	By \cref{prop:B_even_length_and_licence}, there must be a licence~$u$ for $\nPatrn{cbcbc}$ in~$w_2$, and by \cref{prop:B_phase_coherence},~$u$ must correspond to the pattern word $\nPatrn{cbc}$.
	We know $u$ must be born in step~$(i)$,~$(ii)$ or~$(iii)$ since, by~$3$-coherence, $\nPatrn{cbc}$ cannot exist in $w_1$ (which we have established to contain the pattern word $\nPatrn{aba}$).

	If~$u$ is born in step~$(ii)$, then there must be a maximal alternating subword~$v$ of~$w^\prime_2$ with $\ell(v) \geq 4$, which overlaps~$u$ at one of its outer letters.
	We see $v$ must have one of the following pattern words:
	\begin{itemize}
		\item $\nPatrn{cdcd}$,
		\item $\nPatrn{dcdc}$,
		\item $\nPatrn{caca}$, or
		\item $\nPatrn{acac}$.
	\end{itemize}
	By \Cref{lem:alts_geq_4_are_transported},~$v$ is transported from~$w_1$ through to~$w_2$.
	However, by~$3$-coherence, neither $\nPatrn{caca}$ nor $\nPatrn{acac}$ can exist in~$w_1$ (which we have established to contain the pattern word $\nPatrn{bab}$), and similarly neither $\nPatrn{cdcd}$ nor $\nPatrn{dcdc}$ can exist in~$w_2$ (which we have established to contain the pattern word $\nPatrn{bcb}$).
	We have shown that the licence~$u$ is not born in step~$(ii)$.

	Now assume $u$ is born in step $(iii)$.
	This would require a maximal alternating subword~$v$ of~$w_2$ with $\ell(v) \geq 5$, which overlaps~$u$ at one of its outer letters.
	By \Cref{lem:alts_geq_4_are_transported}, $v$ is transported from $w_1$ to $w_2$ (recall that if a toggle is applied to all of $v$, we still consider that a transport).
	We see that certain pattern words coming from $v$ must exist in $w_1$ and $w_2$.
	\begin{itemize}
		\item $\nPatrn{aca}$ in $w_1$ and $\nPatrn{cacac}$ in $w_2$.
		\item $\nPatrn{dcd}$ in $w_1$ and $\nPatrn{cdcdc}$ in $w_2$.
	\end{itemize}
	Neither case is $3$-coherent in $w_2$ (which we have established to contain the pattern word $\nPatrn{cbcbc}$).
	We have shown that the licence $u$ is not born in step $(iii)$, and thus must be born in step $(i)$.

	Since the licence $u$, with pattern $\nPatrn{cbc}$, is born in step $(i)$, there must be a maximal alternating subword $v$ of $w_1$ with $\ell(v) \geq 5$ whose pattern word involves $\nPatrn{c}$.
	When $v$ is toggled, $u$ is born.
	We are left with the following possible pattern words in $w_1$.
	We bracket the part of the pattern word corresponding to $v$ in $w_1$.
	\begin{itemize}
		\item $\nPatrn{cb(dcdcd)}$,
		\item $\nPatrn{(dcdcd)bc}$,
		\item $\nPatrn{cb(acaca)}$,
		\item $\nPatrn{(acaca)bc}$,
		\item $\nPatrn{(dcdcd)b(dcdcd)}$ or
		\item $\nPatrn{(acaca)b(acaca)}$.
	\end{itemize}
	Only the first two options are $3$-coherent, bearing in mind $w_1$ contains the pattern words $\nPatrn{aba}$ and $\nPatrn{bab}$.
	However, the first option contains the pattern $\nPatrn{cbdc}$, so by \cref{prop:B_pattern_exclusion},~$w_1$ must not contain the pattern word $\nPatrn{aba}$, a contradiction.
	The same argument works to show the second option is not possible.

	We have shown that the licence $u$ for the pattern word $\nPatrn{cbcbc}$ in $w_2$ cannot be born in any step, but also cannot already exist in $w_1$ by $3$-coherence, a contradiction.
	Thus, $x$ was not born at step $(i)$ either.
\end{proof}
\begin{lemma}
	\label{lem:word_ops_relating_B_words}
	Let $\Gamma$ be a Coxeter graph with all $m_{ij} \geq 4$.
	Let $w_1$ and $w_2$ be two square-free Coxeter words, corresponding to the same Coxeter group element in $\nCoxG$, both with property $\nWordPropB$.
	Then one can get $w_2$ by simultaneously applying toggles (Artin or $(m_{ij} \pm 1)$-reduction/expansion) to a mutually disjoint set of maximal alternating subwords of $w_1$, each with length $\geq m_{ij} -1$.
\end{lemma}
\begin{proof}
	All alternating subwords $h$ which are toggled at any of the steps in \eqref{eqn:prop_C_word_operation_two_way_sequence} must have $\ell(h) \geq 3$ at the point just before they are toggled.
	By \Cref{lem:alts_geq_4_are_transported}, all such alternating subwords $h$ with $\ell(h) \geq 4$ are transported from~$w_1$.
	The only way such an $h$ with $\ell(h) = 3$ can be toggled is if it is toggled at step~$(iii)$, and by \Cref{lem:all_toggled_3s_are_transported}, all such $h$ are transported from $w_1$ also.

	So, all subwords which were toggled in any step in \eqref{eqn:prop_C_word_operation_two_way_sequence} were transported, and thus were either themselves toggled, or did not interact with any of the toggles.
	The result follows by observing that applying $(m_{ij} + 1)$-reduction, then $(m_{ij} - 1)$-expansion to the same subword is equivalent to not toggling the subword.
\end{proof}
We finish by applying \Cref{sec:Q_restrictions} to prove that for all $q \in Q$, the Coxeter word~$\nCoxWord{q}{W}$ satisfies \cref{prop:B_3_coherence,prop:B_pattern_exclusion}, and if it satisfies \cref{prop:B_odd_length,prop:B_even_length_and_licence}, then it also satisfies \cref{prop:B_phase_coherence}.
We will deal with \cref{prop:B_odd_length,prop:B_even_length_and_licence} themselves in \Cref{sec:word_reductions_in_Bhat}.
\begin{lemma}
	\label{lem:Q_W_properties_and_B1_and_B1_implies_B}
	Let $\Gamma$ be a Coxeter graph with all $m_{ij} \geq 4$.
	Every $w \in \nCoxWord{Q}{W}$ is~$3$-coherent (\cref{prop:B_3_coherence}), and has \cref{prop:B_pattern_exclusion}.
	Furthermore, if $w$ has \cref{prop:B_odd_length,prop:B_even_length_and_licence}, then it has \cref{prop:B_phase_coherence} and thus property $\nWordPropB$.
\end{lemma}
\begin{proof}
	By \Cref{lem:Q_square_free} we have that $w$ is square-free.
	We have that $w$ is $3$-coherent (\cref{prop:B_3_coherence}) by applying \Cref{lem:paired_exclusion} to $u = \nFreeGenElt_i^{\pm 1}\nFreeGenElt_k^{\pm 1}\nFreeGenElt_i^{\pm 1}$ and $v = \nFreeGenElt_j^{\pm 1}\nFreeGenElt_k^{\pm 1}\nFreeGenElt_j^{\pm 1}$, with all $\pm$ independent, and with $i \neq j$.
	Then, \cref{prop:B_pattern_exclusion} follows by applying \Cref{lem:paired_exclusion} to $u = \nFreeGenElt_i^{\pm 1}\nFreeGenElt_j^{\pm 1}\nFreeGenElt_k^{\pm 1}\nFreeGenElt_i^{\pm 1}$ and $v = \nFreeGenElt_l^{\pm 1}\nFreeGenElt_j^{\pm 1}\nFreeGenElt_l^{\pm 1}$ or $v = \nFreeGenElt_l^{\pm 1}\nFreeGenElt_k^{\pm 1}\nFreeGenElt_l^{\pm 1}$, with all~$\pm$ independent, and with $\Set{i,j,k,l}$ all distinct.
	We get \cref{prop:B_phase_coherence} by \Cref{lem:phase_coherent}.
\end{proof}

%% file: sections/theorem/theorem_main.tex
\section{The main theorem}
\input{sections/theorem/intro}
\input{sections/theorem/geometric_view}
\input{sections/theorem/controlling_w_reductions}
\input{sections/theorem/controlling_artin_operations}

%% file: sections/theorem/intro.tex
In this section we complete the proof of our main theorem, which involves the following steps.
Recall the free group $\nFree$ is generated by $\Set{\nTupleFree}$, and recall the definition of the Hurwitz words $Q \coloneq \hurelt(\nTupleFree)$.
\begin{enumerate}
	\item We first show that, given certain hypotheses on the Coxeter group, for all $q \in Q$ there exists some $r \in Q$ such that $\nCoxWord{r}{W}$ has property $\nWordPropB$, and $\nPiArt(q) = \nPiArt(r)$.
	      \smallskip
	\item We then show that given two $r, r^\prime \in Q$, if $\nPiCox(r) = \nPiCox(r^\prime)$, and if~$\nCoxWord{r}{W}$ and~$\nCoxWord{r^\prime}{W}$ both have property $\nWordPropB$, then $\nPiArt(r) = \nPiArt(r^\prime)$.
\end{enumerate}
Given proofs of the above and a Coxeter graph $\Gamma$ satisfying the required hypotheses, suppose $q,q^\prime \in Q$ were arbitrary elements such that $\nPiCox(q) = \nPiCox(q^\prime)$.
Then, by~$(1)$ we would have some $r, r^\prime$ with $\nPiArt(r) = \nPiArt(q)$ and $\nPiArt(r^\prime) = \nPiArt(q^\prime)$ where $\nCoxWord{r}{W}$ and $\nCoxWord{r^\prime}{W}$ both have property $\nWordPropB$.
Then, by $(2)$ we would have $\nPiArt(r) = \nPiArt(r^\prime)$, so $\nPiArt(q) = \nPiArt(q^\prime)$, so by \Cref{cor:equality_in_W_implies_equality_in_A_implies_theorem} we would have proved that $\nArtG$ is canonically isomorphic to $\nDualArtG$.

%% file: sections/theorem/geometric_view.tex
\subsection{A subgroup of the braid group acting on Artin group fibres}
\label{sec:B_hat}
In this subsection, let $\Gamma$ denote an arbitrary Coxeter graph.
In \Cref{cor:hurref_is_image_of_artin_action}, we showed that $Q = B_\nRank \star \nFreeGenElt_1$, so acting by $B_\nRank$ preserves $Q$ membership.
Here, we define a subgroup $\nBraidSubgrpG \leq B_\nRank$ that preserves fibres of $\nPiArt$, i.e.~given any $\hat{\beta} \in \nBraidSubgrpG$ and any $x \in \nFree$, we have $\nPiArt(\hat{\beta} \star x) = \nPiArt(x)$.
We will only consider the case where we act on some $q \in Q \subseteq \nFree$, but the arguments are general to all of $\nFree$.
We take the following definition from \cite{birman_etal_new_1998}.
\begin{definition}
	\label{def:generalised_braid_generators}
	Let $i,j \in \N$ with $1 \leq i < j \leq \nRank$.
	The \emph{Birman--Ko--Lee generator} $\sigma_{ij} \in B_\nRank$ is defined to be
	\[
		\sigma_{ij} \coloneq \sigma_i \cdots \sigma_{j-2} \sigma_{j-1} \sigma_{j-2}^{-1} \cdots \sigma_{i}^{-1}.
	\]
	By convention, $\sigma_{i,i+1} \coloneq \sigma_i$, and  $\sigma_{ji} \coloneq \sigma_{ij}$.
\end{definition}
For each arc $\mu_{ij}$ (see \Cref{fig:upside_down_ice_cream_cone}) between the points $\nPoint_i$ and $\nPoint_j$, we define the loop $\hat{\mu}_{ij}$ which follows close to $\mu_{ij}$ and surrounds $\nPoint_i$ and $\nPoint_{j}$.
Each Birman--Ko--Lee generator $\sigma_{ij}$ is a half twist around the loop~$\hat{\mu}_{ij}$.
Suppose $q = q_1\cdots q_l \in Q$ and let $\gamma$ be a minimal simple loop representing $q$.
Since we can draw each $\hat{\mu}_{ij}$ as close to $\mu_{ij}$ as we like, the geometric action of $\sigma_{ij}$ on $\gamma$ only affects the sub-arcs of $\gamma$ which cross $\mu_{ij}$, so the Artin action of $\sigma_{ij}$ on $q$ only changes subwords $q_kq_{k+1}$ where $\gamma_{q_kq_{k+1}}$ crosses $\mu_{ij}$.

It will be useful to make the following definition to deal with alternating free words which have a single change of sign.
Given some $i,j \geq 0$ and letters $x^{\pm 1},y^{\pm 1}$ of equal exponent, define $\nPi{x}{y}{i,j}$ to be the alternating word which starts with~$x$ (if $i \neq 0$), is of total length $i + j$, and changes sign after letter $i$.
For example, $\nPi{x}{y}{1,2} = xy^{-1}x^{-1}$ and $\nPi{y^{-1}}{x^{-1}}{3,3} = y^{-1}x^{-1}y^{-1}xyx$.

\begin{example}
	\label{ex:action_of_sigma_ij_m_ij}
	Let $q \in Q$ and let $\gamma$ be a minimal simple loop representing $q$.
	Suppose that $\gamma$ passes through the arc $\mu_{ij}$, and that $m_{ij} = 3$.
	We want to describe $q^\prime \coloneq \sigma_{ij}^3 \star q$, and do so by acting on $\gamma$ using the geometric action of $B_\nRank$ on $\nPuncturedD$.
	We isolate a particular crossing of $\gamma$ and $\mu_{ij}$ and assume the direction of $\gamma$ at this point is downwards.
	At this crossing, the action of $\sigma_{ij}^3$ on $\gamma$ (by half twists) looks like the following picture.
	\[
	\begin{tikzpicture}[baseline=(text40.center),every node/.append style={scale=1}, inner sep=0pt, outer sep=0pt]
		\input{figs/crossing_action_pre.tex}

	\end{tikzpicture}%

		\qquad
		\begin{tikzcd}
			{} \ar[rr, rightsquigarrow,  "\sigma_{ij}^3"] &  & {}
		\end{tikzcd}
		\qquad
	\begin{tikzpicture}[baseline=(text40.center),every node/.append style={scale=1}, inner sep=0pt, outer sep=0pt]
		\input{figs/crossing_action_post.tex}
	\end{tikzpicture}%

	\]
	The sub-arc of $\gamma$ which crossed $\mu_{ij}$ has been replaced by an arc that contributes a subword $\nFreeGenElt_j\nFreeGenElt_i\nFreeGenElt_j\nFreeGenElt_i^{-1}\nFreeGenElt_j^{-1}\nFreeGenElt_i^{-1}$.
	In $q$, $\sigma_{ij}^3$ acts by inserting $\nFreeGenElt_j\nFreeGenElt_i\nFreeGenElt_j\nFreeGenElt_i^{-1}\nFreeGenElt_j^{-1}\nFreeGenElt_i^{-1}$ at all positions corresponding to crossings of $\mu_{ij}$.
	So, since $m_{ij} =3$, we have $\nPiArt(q) = \nPiArt(\sigma_{ij}^3 \star q)$.
	Had we acted by $\sigma_{ij}^{-3}$ instead, the insertion would have been $\nFreeGenElt_i^{-1}\nFreeGenElt_j^{-1}\nFreeGenElt_i^{-1}\nFreeGenElt_j\nFreeGenElt_i\nFreeGenElt_j$, so similarly $\nPiArt(q) = \nPiArt(\sigma_{ij}^{-3} \star q)$.
	The insertion also depends on the direction in which~$\gamma$ crosses~$\mu_{ij}$, but in every case the action preserves fibres of~$\nPiArt$.
\end{example}
Given a Coxeter graph $\Gamma$, let $\nBraidSubgrpG \leq B_\nRank$ be the subgroup generated by $\sigma_{ij}^{m_{ij}}$ for all $m_{ij} < \infty$.
This subgroup was initially introduced by Kluitmann \cite{kluitmann_isotropy_1991}.
\begin{lemma}
	\label{lem:artin_preserving_braid_action}
	Let $\Gamma$ be a Coxeter graph.
	If for every $q_1, q_2 \in Q$ such that $\nPiCox(q_1) = \nPiCox(q_2)$ we have some $\hat{\beta} \in \nBraidSubgrpG$ such that $\hat{\beta} \star q_1 = q_2$, then $\nPiArt(q_1) = \nPiArt(q_2)$ and $\nArtG$ is canonically isomorphic to $\nDualArtG$.
\end{lemma}
\begin{proof}
	The word insertions seen in \Cref{ex:action_of_sigma_ij_m_ij} generalise.
	We list the word insertions due to the action of $\sigma_{ij}^{\pm n}$ in \Cref{tab:action_of_sigma_ij}, which depend on the direction of the crossing of $\mu_{ij}$.
	\begin{table}[ht]
		\centering
		{\renewcommand{\arraystretch}{1.4}
			\begin{tabular}{l|l|l}
				\multicolumn{1}{c|}{Crossing direction}       &
				\multicolumn{1}{c|}{Act by $\sigma_{ij}^{n}$} &
				\multicolumn{1}{c}{Act by $\sigma_{ij}^{-n}$}                                                                                                       \\
				\hline
				$\uparrow$                                    & $\nPi{\nFreeGenElt_i}{\nFreeGenElt_j}{n,n}$ & $\nPi{\nFreeGenElt_j^{-1}}{\nFreeGenElt_i^{-1}}{n,n}$ \\
				$\downarrow$                                  & $\nPi{\nFreeGenElt_j}{\nFreeGenElt_i}{n,n}$ & $\nPi{\nFreeGenElt_i^{-1}}{\nFreeGenElt_j^{-1}}{n,n}$ \\
			\end{tabular}
		}
		\caption{Word insertions at all crossings of $\mu_{ij}$ in some $q \in Q$ due to the action of~$\sigma_{ij}^{\pm n}$.
			The direction of the $\mu_{ij}$ crossing is specified in the left column and the element we act by is specified in the top row.
		}
		\label{tab:action_of_sigma_ij}
	\end{table}
	By \Cref{tab:action_of_sigma_ij} it is clear that $\nPiArt(q_1) = \nPiArt(q_2)$.
	The result follows by \Cref{cor:equality_in_W_implies_equality_in_A_implies_theorem}.
\end{proof}
We also record an alternative criterion.
The hypotheses of the following corollary are proven in this paper when all $m_{ij} \geq 4$.
Experimentally, they appear to hold in several other Coxeter groups which we considered in a search for examples where these hypotheses do not hold.
\begin{corollary}
	\label{cor:artin_preserving_braid_action}
	Let $\Gamma$ be a Coxeter graph.
	If for every $\beta_1, \beta_2 \in B_\nRank$ such that $\nPiCox(\beta_1 \star \nFreeGenElt_1) = \nPiCox(\beta_2 \star \nFreeGenElt_1)$ we have
	\[
		\nBraidSubgrpG \cap \beta_2 \GroupPres{\sigma_2,\ldots,\sigma_{\nRank-1}} \beta_1^{-1} \neq \emptyset,
	\]
	then $\nArtG$ is canonically isomorphic to $\nDualArtG$.
\end{corollary}
\begin{proof}
	Recall by \Cref{cor:hurref_is_image_of_artin_action} that $Q = B_\nRank \star \nFreeGenElt_1$.
	Given two $\beta_1,\beta_2 \in B_\nRank$, the set of $\beta^\prime \in B_\nRank$ such that $\beta^\prime \star (\beta_1 \star \nFreeGenElt_1) = \beta_2 \star \nFreeGenElt_1$ is exactly $\beta_2\stab_\star(\nFreeGenElt_1)\beta_1^{-1}$, and $\stab_\star(\nFreeGenElt_1)$ is exactly the subgroup $B_{\nRank - 1} \leq B_\nRank$ generated by $\sigma_2,\ldots,\sigma_{\nRank-1}$.
	The result follows by \Cref{lem:artin_preserving_braid_action}.
\end{proof}

%% file: figs/crossing_action_pre.tex
\definecolor{c819d43}{RGB}{129,157,67}
\definecolor{c979797}{RGB}{151,151,151}

  \begin{scope}[shift={(-5.6657, 5.6509)}]
    \path[draw=c819d43,line width=0.02cm] (10.4872, 10.9409) -- (10.4872, 10.2123);

    \node[text=black,line width=0.005cm,anchor=south west] (text17) at (10.1785, 11.0538){$\scriptscriptstyle\lambda_{i}$};

  \end{scope}
  \begin{scope}[shift={(-3.6287, 5.6509)}]
    \path[draw=c819d43,line width=0.02cm] (10.4872, 10.9409) -- (10.4872, 10.2123);

    \node[text=black,line width=0.005cm,anchor=south west] (text17-9) at (10.1785, 11.0538){$\scriptscriptstyle\lambda_{j}$};

  \end{scope}
  \path[draw=black,line width=0.015cm,dash pattern=on 0.015cm off 0.06cm] (4.8156, 15.8632) -- (6.8585, 15.8632);

  \path[draw=black,fill=c979797,line width=0.0487cm,cm={ 0.0,1.0,1.0,0.0,(-29.7, 29.7)}] (-13.8368, 36.5585) circle (0.0458cm);

  \path[draw=black,fill=c979797,line width=0.0487cm,cm={ 0.0,1.0,1.0,0.0,(-29.7, 29.7)}] (-13.8368, 34.5214) circle (0.0458cm);

  \path[draw=black,fill=black,line width=0.023cm,cm={ 0.559,0.1367,0.1367,-0.559,(-0.3086, 28.1092)}] (5.4435, 22.9858) -- (5.3196, 22.9566) -- (5.3563, 23.0785) -- cycle;

  \node[text=black,anchor=south west,line width=0.0185cm] (text40) at (5.8207, 15.7024){};

  \path[draw=black,line width=0.02cm,shift={(-0.0621, -0.0202)}] (5.9026, 16.4839) -- (5.9026, 15.4374);

%% file: figs/crossing_action_post.tex
\definecolor{c819d43}{RGB}{129,157,67}
\definecolor{c979797}{RGB}{151,151,151}

  \begin{scope}[shift={(-5.6632, 5.5942)}]
    \path[draw=c819d43,line width=0.02cm] (10.4872, 10.9409) -- (10.4872, 10.2123);

    \node[text=black,line width=0.005cm,anchor=south west] (text17) at (10.1785, 11.0538){$\scriptscriptstyle\lambda_{i}$};

  \end{scope}
  \begin{scope}[shift={(-3.6287, 5.5895)}]
    \path[draw=c819d43,line width=0.02cm] (10.4872, 10.9409) -- (10.4872, 10.2123);

    \node[text=black,line width=0.005cm,anchor=south west] (text17-2) at (10.1785, 11.0538){$\scriptscriptstyle\lambda_{j}$};

  \end{scope}
  \begin{scope}[shift={(0.0, -0.0614)}]
    \path[draw=black,line width=0.015cm,dash pattern=on 0.015cm off 0.06cm] (4.8156, 15.8632) -- (6.8585, 15.8632);

    \path[draw=black,fill=c979797,line width=0.0487cm,cm={ 0.0,1.0,1.0,0.0,(-29.7, 29.7)}] (-13.8368, 36.5585) circle (0.0458cm);

    \path[draw=black,fill=c979797,line width=0.0487cm,cm={ 0.0,1.0,1.0,0.0,(-29.7, 29.7)}] (-13.8368, 34.5214) circle (0.0458cm);

  \end{scope}
  \path[draw=black,fill=black,line width=0.023cm,cm={ 0.559,0.1367,0.1367,-0.559,(-0.3092, 28.6037)}] (5.4435, 22.9858) -- (5.3196, 22.9566) -- (5.3563, 23.0785) -- cycle;

  \path[draw=black,fill=black,line width=0.023cm,cm={ -0.559,-0.1367,-0.1367,0.559,(11.9858, 3.6754)}] (5.4435, 22.9858) -- (5.3196, 22.9566) -- (5.3563, 23.0785) -- cycle;

  \path[draw=black,line width=0.02cm] (5.84, 16.6101) -- (5.84, 16.4349).. controls (5.8402, 16.4029) and (5.8502, 16.3711) .. (5.8685, 16.3448).. controls (5.8976, 16.3029) and (5.9477, 16.2765) .. (5.9987, 16.2762) -- (6.8585, 16.2762).. controls (6.9142, 16.2762) and (6.9707, 16.2741) .. (7.0236, 16.2566).. controls (7.0765, 16.2391) and (7.1244, 16.2072) .. (7.1615, 16.1657).. controls (7.233, 16.0857) and (7.2629, 15.9755) .. (7.2662, 15.8682).. controls (7.2699, 15.7503) and (7.242, 15.628) .. (7.1685, 15.5358).. controls (7.1323, 15.4903) and (7.0847, 15.4534) .. (7.0303, 15.4326).. controls (6.976, 15.4119) and (6.9167, 15.4091) .. (6.8585, 15.4091) -- (4.8082, 15.4091).. controls (4.7789, 15.4091) and (4.7495, 15.4094) .. (4.7207, 15.4143).. controls (4.6919, 15.4191) and (4.6637, 15.4285) .. (4.6383, 15.443).. controls (4.5875, 15.472) and (4.5503, 15.5209) .. (4.5263, 15.5742).. controls (4.4799, 15.6769) and (4.4759, 15.7988) .. (4.5196, 15.9027).. controls (4.543, 15.9581) and (4.5813, 16.0084) .. (4.6336, 16.038).. controls (4.6598, 16.0527) and (4.6887, 16.0623) .. (4.7183, 16.0672).. controls (4.748, 16.0721) and (4.7781, 16.0724) .. (4.8082, 16.0724) -- (6.8318, 16.0724).. controls (6.8524, 16.0724) and (6.8732, 16.0721) .. (6.8934, 16.0682).. controls (6.9136, 16.0641) and (6.9331, 16.0564) .. (6.9501, 16.0449).. controls (6.9672, 16.0333) and (6.9816, 16.0181) .. (6.993, 16.0009).. controls (7.0044, 15.9837) and (7.0128, 15.9648) .. (7.0196, 15.9453).. controls (7.0442, 15.8756) and (7.0485, 15.7966) .. (7.0196, 15.7285).. controls (7.0044, 15.6927) and (6.9791, 15.6605) .. (6.9451, 15.6417).. controls (6.928, 15.6323) and (6.9092, 15.6263) .. (6.89, 15.6232).. controls (6.8708, 15.6201) and (6.8513, 15.62) .. (6.8318, 15.6201) -- (5.9429, 15.6238).. controls (5.9251, 15.6237) and (5.9073, 15.6284) .. (5.8919, 15.6372).. controls (5.8765, 15.646) and (5.8634, 15.6589) .. (5.8544, 15.6742).. controls (5.8445, 15.691) and (5.8397, 15.7103) .. (5.837, 15.7296).. controls (5.8308, 15.7751) and (5.8365, 15.8213) .. (5.837, 15.8672).. controls (5.8372, 15.8827) and (5.8368, 15.8983) .. (5.8331, 15.9134).. controls (5.8294, 15.9284) and (5.8222, 15.943) .. (5.8106, 15.9532).. controls (5.8001, 15.9624) and (5.7866, 15.9677) .. (5.7729, 15.9703).. controls (5.7592, 15.9729) and (5.7452, 15.973) .. (5.7312, 15.973) -- (4.8187, 15.977).. controls (4.7977, 15.977) and (4.7766, 15.9769) .. (4.7559, 15.9735).. controls (4.7352, 15.9701) and (4.7149, 15.9635) .. (4.6966, 15.9532).. controls (4.66, 15.9326) and (4.6331, 15.8976) .. (4.6161, 15.8592).. controls (4.5833, 15.7853) and (4.5832, 15.6975) .. (4.6161, 15.6237).. controls (4.6331, 15.5853) and (4.6601, 15.5505) .. (4.6967, 15.5299).. controls (4.715, 15.5196) and (4.7353, 15.513) .. (4.756, 15.5096).. controls (4.7767, 15.5061) and (4.7977, 15.5059) .. (4.8187, 15.5059) -- (6.8509, 15.5053).. controls (6.8932, 15.5053) and (6.9362, 15.5072) .. (6.9759, 15.5219).. controls (7.0155, 15.5366) and (7.0504, 15.5629) .. (7.0773, 15.5955).. controls (7.1304, 15.6597) and (7.1528, 15.7449) .. (7.1551, 15.8283).. controls (7.1574, 15.9143) and (7.1384, 16.0034) .. (7.0864, 16.0719).. controls (7.0595, 16.1074) and (7.0236, 16.1363) .. (6.9822, 16.1526).. controls (6.9408, 16.1689) and (6.8954, 16.1711) .. (6.8509, 16.1711) -- (4.824, 16.1708).. controls (4.7853, 16.1708) and (4.7465, 16.1704) .. (4.7083, 16.1642).. controls (4.6701, 16.1579) and (4.6328, 16.1458) .. (4.599, 16.127).. controls (4.5314, 16.0894) and (4.4812, 16.0255) .. (4.4494, 15.955).. controls (4.3886, 15.8198) and (4.3887, 15.6586) .. (4.4494, 15.5233).. controls (4.4811, 15.4527) and (4.5313, 15.3886) .. (4.5989, 15.351).. controls (4.6327, 15.3322) and (4.67, 15.3201) .. (4.7082, 15.3139).. controls (4.7464, 15.3077) and (4.7853, 15.3074) .. (4.824, 15.3075) -- (5.729, 15.3098).. controls (5.7503, 15.311) and (5.772, 15.3052) .. (5.7898, 15.2935).. controls (5.8044, 15.2839) and (5.8164, 15.2704) .. (5.8243, 15.2548).. controls (5.8321, 15.2391) and (5.8358, 15.2215) .. (5.8348, 15.204) -- (5.8348, 15.0159);

  \path[draw=black,fill=black,line width=0.023cm,cm={ 0.559,0.1367,0.1367,-0.559,(-0.314, 27.2905)}] (5.4435, 22.9858) -- (5.3196, 22.9566) -- (5.3563, 23.0785) -- cycle;

  \node[text=black,anchor=south west,line width=0.0185cm] (text40) at (5.7728, 15.6128){};

%% file: sections/theorem/controlling_w_reductions.tex
\subsection{Controlling Coxeter word reductions}
\label{sec:word_reductions_in_Bhat}
\Cref{lem:Q_W_properties_and_B1_and_B1_implies_B} shows that for any $w \in Q_W$, if alternating subwords in $w$ have \cref{prop:B_odd_length,prop:B_even_length_and_licence}, then $w$ has property $\nWordPropB$.
In this subsection, we show that when all $m_{ij} \geq 4$, for all $q \in Q$ there exists $\hat{\beta} \in \nBraidSubgrpG$ such that $(\hat{\beta} \star q)_W$ has property~$\nWordPropB$.

Recall the notation $\nPi{x}{y}{k,l}$ from \Cref{sec:B_hat}.
Let $q \in Q$ and suppose that $u = \nPi[\big]{\nFreeGenElt_i^{\pm 1}}{\nFreeGenElt_j^{\pm 1}}{k,l}$ is a maximal alternating subword of $q$.
Then \Cref{lem:change_sign_close_to_middle} tells us that if $\ell(u) =k+l\geq 3$, then $\Abs{k-l} \leq 1$.
In particular, if $\ell(u) \geq 3$, then at least one of the following statements is true.
\begin{enumerate}
	\item  $k \geq m_{ij}$ and  $l \geq m_{ij}$.
	      \smallskip
	\item $k \leq m_{ij}$ and $l \leq m_{ij}$.
\end{enumerate}
Both statements are true only if $l=k=m_{ij}$.
Bearing this in mind, let us use \Cref{tab:action_of_sigma_ij} to compute how $\sigma_{ij}^{\pm m_{ij}}$ acts on certain $(\nFreeGenElt_i,\nFreeGenElt_j)$-alternating words that can appear in $q$.
\begin{flalign}
	\begin{aligned}
		\sigma_{ij}^{-m_{ij}} \star \nPi{\nFreeGenElt_i}{\nFreeGenElt_j}{k,l}          & =
		\begin{cases}
			\nPi{\nFreeGenElt_i}{\nFreeGenElt_j}{k-m_{ij}, l-m_{ij}},               & k,l \geq m_{ij} \\
			\nPi{\nFreeGenElt_j^{-1}}{\nFreeGenElt_i^{-1}}{m_{ij} - k, m_{ij} - l}, & k,l \leq m_{ij}
		\end{cases} \\
		\sigma_{ij}^{m_{ij}} \star \nPi{\nFreeGenElt_i^{-1}}{\nFreeGenElt_j^{-1}}{k,l} & =
		\begin{cases}
			\nPi{\nFreeGenElt_i^{-1}}{\nFreeGenElt_j^{-1}}{k-m_{ij}, l-m_{ij}}, & k,l \geq m_{ij} \\
			\nPi{\nFreeGenElt_j}{\nFreeGenElt_i}{m_{ij} - k, m_{ij} - l},       & k,l \leq m_{ij}
		\end{cases}
	\end{aligned}
	\label{eqn:action_of_sigma_ij_on_alternating_words}
\end{flalign}
These equations are well-defined because if $k=l=m_{ij}$ then the right-hand side is the empty word in either case.
These equations are valid regardless of whether $i < j$.

\begin{lemma}
	\label{lem:reduction_of_troublesome_alternating_subwords}
	Suppose $q \in Q$ has a maximal $(\nFreeGenElt_i,\nFreeGenElt_j)$-alternating subword $u$ with $\ell(u) \geq 4$, satisfying any of the following properties:
	\begin{enumerate}
		\item $\ell(u) \geq m_{ij} + 2$.
		\item $\ell(u) = m_{ij} + 1$ is even.
		\item $\ell(u) = m_{ij} + 1$ is odd, and there does not exist any maximal $(\nFreeGenElt_i,\nFreeGenElt_j)$-alternating subword $v$ with $\ell(v) = m_{ij} - 1$.
	\end{enumerate}
	Then there exists some $\hat{\beta} \in \nBraidSubgrpG$ such that $\ell(\hat{\beta} \star q) < \ell(q)$.
\end{lemma}
\begin{proof}
	Let $\gamma$ be a minimal simple loop representing $q$.
	Let $v_1,\ldots,v_l$ be all maximal $(\nFreeGenElt_i,\nFreeGenElt_j)$-alternating subwords with $\ell(v_k) \geq 3$ in $q$, including $u$.
	\Cref{cor:single_crossing_of_line_in_alternating_subword} tells us there is a crossing of $\mu_{ij}$ in $\gamma_u$ and by \Cref{lem:possible_lengths_other_crossings}, all crossings of $\mu_{ij}$ occur in $\gamma_{v_k}$ for some $k$.
	This tells us that the action of~$\hat{\beta}^{\pm 1} \coloneq \sigma_{ij}^{\pm m_{ij}}$ on $q$ is determined by its action on each of the~$v_k$.

	In all cases for this lemma, \Cref{lem:possible_lengths_other_crossings} tells us that all $v_k$ satisfy $\ell(v_k) \geq m_{ij}$.
	\Cref{lem:crossing_diagonals_and_one_change_of_sign} tells us that all of $v_1,\ldots,v_l$ have changes of sign of one type ($Z_+$ or~$Z_-$).

	Let us assume all the changes of sign are of $Z_+$ type.
	Then \eqref{eqn:action_of_sigma_ij_on_alternating_words} tells us that acting by $\hat{\beta}^{-1}$ will reduce the length of $u$ and will not increase the length of any of the $v_k$.
	So, $\ell(\hat{\beta}^{-1} \star q) < \ell(q)$.
	If all the changes of sign were instead of $Z_-$ type, then by a similar argument $\ell(\hat{\beta}^{+1} \star q) < \ell(q)$.
\end{proof}
\begin{corollary}
	\label{cor:bhat_B}
	Let $\Gamma$ be a Coxeter graph with all $m_{ij} \geq 4$.
	For all $q \in Q$, there exists some $\hat{\beta} \in \nBraidSubgrpG$ such that $\nCoxWord{(\hat{\beta} \star q)}{W}$ has property $\nWordPropB$.
\end{corollary}
\begin{proof}
	Any alternating subword $u$ of $\nCoxWord{q}{W}$ which violates \cref{prop:B_odd_length} or \cref{prop:B_even_length_and_licence} must have $\ell(u) \geq 4$, and must satisfy one of the cases of \Cref{lem:reduction_of_troublesome_alternating_subwords}.
	Whenever such a subword appears, we apply \Cref{lem:reduction_of_troublesome_alternating_subwords} until we get a $\hat{\beta} \in \nBraidSubgrpG$ such that $w \coloneq \nCoxWord{(\hat{\beta} \star q)}{W}$ has \cref{prop:B_odd_length,prop:B_even_length_and_licence}.
	Then, by \Cref{lem:Q_W_properties_and_B1_and_B1_implies_B}, $w$ has property~$\nWordPropB$.
\end{proof}
\begin{remark}
	Given a Coxeter graph $\Gamma$ such that all $m_{ij} \geq 5$ and all $m_{ij}$ are odd, we can show there exists a $\hat{\beta}$ such that $\nCoxWord{(\hat{\beta} \star q)}{W}$ has property $\nWordPropA$ (and property $\nWordPropB$), and is thus reduced.
\end{remark}

%% file: sections/theorem/controlling_artin_operations.tex
\subsection{Performing toggling operations using a subgroup of the braid group}
Given a Coxeter graph $\Gamma$ with all $m_{ij} \geq 4$, \Cref{lem:word_ops_relating_B_words} tells us that two Coxeter words $w_1,w_2$ with property $\nWordPropB$ are related by simultaneously applying toggles (Artin or $(m_{ij} \pm 1)$-reduction/expansion) to a mutually disjoint set of maximal alternating subwords.
We now want to show that such a sequence of toggles can be performed by correctly choosing elements of~$\nBraidSubgrpG$, a result similar to \Cref{lem:reduction_of_troublesome_alternating_subwords}.
We begin by showing this approach works on individual subwords.
\begin{lemma}
	\label{lem:B_word_operations_achieved_in_B_hat}
	Let $\Gamma$ be a Coxeter graph and let $q \in Q$.
	Let $u$ be a maximal $(\nFreeGenElt_i,\nFreeGenElt_j)$-alternating subword of $q$ with $\ell(u) \geq 3$.
	The following all hold.
	\begin{enumerate}
		\item If $\ell(u) = m_{ij}$, then there exists $\hat{\beta} \in \Set{\sigma_{ij}^{\pm m_{ij}}}$ such that $\nCoxWord{(\hat{\beta}\star u)}{W}$ is the result of applying the relevant Artin operation to $\nCoxWord{u}{W}$.
		      \smallskip
		\item If $\ell(u) = m_{ij}  + 1$, then there exists $\hat{\beta} \in \Set{\sigma_{ij}^{\pm m_{ij}}}$ such that $\nCoxWord{(\hat{\beta}\star u)}{W}$ is the result of applying an $(m_{ij} + 1)$-reduction to $\nCoxWord{u}{W}$.
		      \smallskip
		\item If $\ell(u) = m_{ij}  - 1$, then there exists $\hat{\beta} \in \Set{\sigma_{ij}^{\pm m_{ij}}}$ such that $\nCoxWord{(\hat{\beta}\star u)}{W}$ is the result of applying an $(m_{ij} - 1)$-expansion to $\nCoxWord{u}{W}$.
	\end{enumerate}
	The choice of element of $\Set{\sigma_{ij}^{\pm m_{ij}}}$ depends on the change of sign type in $u$.
\end{lemma}
\begin{proof}
	In each case, by \Cref{lem:crossing_diagonals_and_one_change_of_sign}, $u$ has one change of sign, so the actions of $\sigma_{ij}^{\pm m_{ij}}$ on $u$ are described by the cases $k,l \leq m_{ij}$ in \eqref{eqn:action_of_sigma_ij_on_alternating_words}.
	The result follows by making the correct choice in $\Set{\pm m_{ij}}$, which depends on the change of sign type in $u$.
	If $u$ has a $Z_+$ change of sign, act by $\sigma_{ij}^{-m_{ij}}$, and vice versa.
\end{proof}
Suppose we have some $w \in Q_W$, and that a toggling operation is applied to some maximal $(\nCoxGenElt_i,\nCoxGenElt_j)$-alternating subword of $w$.
We no longer know if the resulting word is in $Q_W$.
The argument we wish to make is that a toggling operation must also be applied to every other maximal $(\nCoxGenElt_i,\nCoxGenElt_j)$-alternating subword of $w$ in order to get a word in $Q_W$ (as would occur if we were acting on $q$ by $\sigma_{ij}^{\pm m_{ij}} \in \nBraidSubgrpG$).
The key to this argument is statement~$(1)$ in \Cref{lem:crossing_diagonals_and_one_change_of_sign}, and the fact that the change of sign type in a maximal alternating subword can still be detected in the Coxeter word, which is proved in \Cref{lem:change_of_sign_given_by_Z}.
\begin{lemma}
	\label{lem:non_crossing_pairs}
	Let $q \in Q$.
	Let~$u$ be a maximal $(\nFreeGenElt_i,\nFreeGenElt_j)$-alternating subword and let~$v$ be a maximal $(\nFreeGenElt_k, \nFreeGenElt_l)$-alternating subword, such that $\ell(u),\ell(v) \geq 3$.
	We may assume $i < j$ and $k < l$.
	If $\Set{i,j,k,l}$ are all distinct, then one of the following holds:
	\begin{enumerate}
		\item $k,l < i$,
		      \smallskip
		\item $k,l > j$,
		      \smallskip
		\item $i < k < l < j\;$ or
		      \smallskip
		\item $k < i < j < l$.
	\end{enumerate}
\end{lemma}
\begin{proof}
	The two other possible situations are
	\begin{enumerate}[label=(\roman*)]
		\item $i < k  < j < l\;$ and
		      \smallskip
		\item $k < i < l < j$.
	\end{enumerate}
	Let us assume (i) occurs.
	Below, we draw one of the possible pictures for $\gamma_u$ near the crossing of $\mu_{ij}$.
	\[
	\begin{tikzpicture}[,every node/.append style={scale=1}, inner sep=0pt, outer sep=0pt]
		\input{figs/non_crossing_mu_ij.tex}
	\end{tikzpicture}%

	\]
	We see that we cannot draw $\gamma_{v}$ without crossing $\gamma_u$.
	This occurs for all such possible pictures of $\gamma_u$ near the crossing of $\mu_{ij}$.
	We can make a similar argument to show~$(ii)$ cannot occur.
\end{proof}
Let $\Gamma$ be a Coxeter graph and let $w = w_1\cdots w_m$ be a square-free Coxeter word in $\nCoxG$.
Let $i<j$ and suppose $u$ is a maximal $(\nCoxGenElt_i,\nCoxGenElt_j)$-alternating subword of~$w$.
Suppose that $u$ has length $l$ and starts at letter $k$, so $w$ has the following factorisation.
\[
	w = (w_1\cdots w_{k-1})u(w_{k+l} \cdots w_m).
\]
We define the following function which encodes how the point associated to the letter before $u$ is placed relative to  $\nPoint_i$ and  $\nPoint_j$.
\[
	\nPrevRegion(u) = \begin{cases}
		+1 & \text{if $k = 1$},                                \\
		-1 & \text{if $w_{k-1}\in\{s_{i+1},\ldots,s_{j-1}\}$}, \\
		+1 & \text{otherwise}
	\end{cases}
\]
We define $R_{i,j}$ to be the interior of the region of $\nPuncturedD$ bounded by $\lambda_i$, $\lambda_j$ and $\mu_{ij}$.
Below we show the region $R_{2,4}$ shaded in grey for the rank 5 case.
\[
	\begin{tikzpicture}[,every node/.append style={scale=1}, inner sep=0pt, outer sep=0pt]
		\input{figs/R_24.tex}
	\end{tikzpicture}%

\]
The name $\nPrevRegion$ stands for \emph{previous region}, which refers to whether the point associated to the letter preceding $u$ is in the region $R_{i,j}$ associated to $u$.
Since $u$ is maximal and $w$ is square-free, this point does not lie on the boundary of $R_{i,j}$.
\begin{lemma}
	\label{lem:prev_region_constant}
	Let $w \in Q_W$ and $h_1,\ldots,h_m$ be a collection of maximal alternating subwords of $w$ such that all $\ell(h_k) \geq 3$.
	Let $u$ be a maximal alternating subword of $w$ with $\ell(u) \geq 3$, which is either some $h_k$ subword, or is disjoint from all of $h_1,\ldots,h_m$.
	Let $w^\prime$ be a Coxeter word obtained by applying toggling operations to any number of the $h_k$ subwords, and let $u^\prime$ be the subword of $w^\prime$ coming from $u$.
	Then $\nPrevRegion(u^\prime) = \nPrevRegion(u)$.
\end{lemma}
\begin{proof}
	Since $u$ is either one of the $h_k$ subwords, or is disjoint from each of the $h_k$ subwords, we can unambiguously refer to the subword of $w^\prime$ which comes from $u$, i.e.~$u^\prime$ is well-defined.

	Either $u$ is the same word as $u^\prime$, or $u^\prime$ is the result of a toggling operation applied to $u$.
	In either case, $u$ and $u^\prime$ have the same appearing letters, so $\nPrevRegion(u^\prime) \neq \nPrevRegion(u)$ only if the letter preceding~$u^\prime$ is different to the letter preceding~$u$.
	Suppose the letter preceding $u$ is different to the letter preceding $u^\prime$.
	This only occurs if for some $p$, a toggling operation is applied to $h_p$, and $u$ occurs immediately after $h_p$.
	Let $u$ be such an $(\nCoxGenElt_i,\nCoxGenElt_j)$-alternating subword with $i < j$, and $h_p$ such an $(\nCoxGenElt_k,\nCoxGenElt_l)$-alternating subword with $k<l$.
	Since $u$ and $h_p$ are adjacent maximal alternating subwords, they do not share both appearing generators, so by $3$-coherence (\Cref{lem:Q_W_properties_and_B1_and_B1_implies_B}), $u$ and $h_p$ share no appearing letters.
	So, \Cref{lem:non_crossing_pairs} tells us one of the following cases occurs:
	\begin{enumerate}
		\item $k,l < i$,
		      \smallskip
		\item  $k,l > j$,
		      \smallskip
		\item  $i < k < l < j$ or
		      \smallskip
		\item $k < i < j < l$
	\end{enumerate}
	If the letter preceding $u$ is $\nCoxGenElt_k$, then the letter preceding $u^\prime$ is $\nCoxGenElt_l$, but in all the cases above, $\nPrevRegion(u) = \nPrevRegion(u^\prime)$.
	This still holds if the letter preceding $u$ is $\nCoxGenElt_l$.
\end{proof}
Let $\sgn \colon \Z \setminus \Set{0} \to \Set{\pm 1}$ be the sign function.
Let $\Gamma$ be a Coxeter graph and let $w$ be a Coxeter word in $\nCoxG$.
Let $u$ be an $(\nCoxGenElt_i,\nCoxGenElt_j)$-alternating subword of $w$ such that $\nCoxGenElt_i$ is the initial letter of $u$ (we can always choose $i$ and $j$ so that this holds).
The points $\nPoint_i$ and $\nPoint_j$ occur in some left-to-right ordering in $\nPuncturedD$, which is geometrically relevant.
With $\nCoxGenElt_i$ the initial letter of $u$, this relative positioning of~$\nPoint_i$ and~$\nPoint_j$ is encoded using a function, defined as follows.
\[
	\nOrdering(u) \coloneq \sgn(j-i).
\]
We denote the product of $\nPrevRegion$ and  $\nOrdering$ by the function $Z$, i.e.
\[
	Z(u) \coloneq \nPrevRegion(u)\nOrdering(u).
\]
\begin{lemma}
	\label{lem:change_of_sign_given_by_Z}
	Let $q \in Q$ and let $u$ be a maximal alternating subword of $q$ with $\ell(u) \geq 3$.
	The change of sign type in $u$ is given by $Z(\nCoxWord{u}{W})$, i.e.~if $Z(\nCoxWord{u}{W}) = \pm1$, then $u$ has a $Z_\pm$ change of sign.
\end{lemma}
\begin{proof}
	Let $\gamma$ be a minimal simple loop representing $q$.
	Suppose $u$ is a (maximal) $(\nFreeGenElt_i,\nFreeGenElt_j)$-alternating subword of $q$.
	Now suppose that $Z(\nCoxWord{u}{W}) = -1$ and that $u$ has a $Z_+$ change of sign, contradicting the lemma.
	So, the first letter in $u$ has positive exponent.
	There are two possibilities, either
	\begin{enumerate}
		\item $\nOrdering(\nCoxWord{u}{W}) = -1$ and  $\nPrevRegion(\nCoxWord{u}{W}) = 1$ or
		      \smallskip
		\item  $\nOrdering(\nCoxWord{u}{W}) = 1$ and  $\nPrevRegion(\nCoxWord{u}{W}) = -1$.
	\end{enumerate}
	Assume (1), so $u$ starts with $\nFreeGenElt_j^{+1}$, where $j>i$, and the letter preceding $\nCoxWord{u}{W}$ is outside the region $R_{i,j}$.
	Suppose the letter preceding $u$ is $\nFreeGenElt_k^{-1}$, and that $k>j$.
	Then there would be a sub-arc of $\gamma$ that looks like the following picture, which includes the beginning of $\gamma_u$ and a bit of $\gamma$ before $\gamma_u$.
	\[
	\begin{tikzpicture}[,every node/.append style={scale=1}, inner sep=0pt, outer sep=0pt]
		\input{figs/crossing_io_01.tex}
	\end{tikzpicture}%

	\]
	We see there is no way for $\gamma$ to next cross  $\lambda_i$ without crossing itself.
	The same holds if the preceding letter is instead $\nFreeGenElt_k^{+1}$, or if $k<i$.

	For case (2), we get a similar contradiction.
	Suppose the letter preceding $u$ is~$\nFreeGenElt_k^{-1}$, where $i < k < j$.
	So $u$ starts with $\nFreeGenElt_i^{+1}$ and there is a sub-arc of $\gamma$ that is as in the following picture.
	\[
	\begin{tikzpicture}[,every node/.append style={scale=1}, inner sep=0pt, outer sep=0pt]
		\input{figs/crossing_io_02.tex}
	\end{tikzpicture}%

	\]
	We see there is no way for $\gamma$ to next cross  $\lambda_j$ without crossing itself.
	The same holds if the preceding letter is instead $\nFreeGenElt_k^{+1}$.
	We can construct similar contradictions if instead $Z(\nCoxWord{u}{W}) = +1$ and $u$ has a $Z_-$ change of sign.

	The function $\nPrevRegion$ is appropriately defined for this to still work when $u$ is at the beginning of $q$.
\end{proof}
For the following lemma, given some $q \in Q$, if there is some subword $u$ of $q$, then we may say that \emph{$u$ is toggled} as a shorthand for the toggling of $\nCoxWord{u}{W}$ as a subword of $\nCoxWord{q}{W}$.
\begin{lemma}
	\label{lem:toggle_applied_universally}
	Let~$\Gamma$ be a Coxeter graph with all $m_{ij} \geq 4$.
	Let $w_1,w_2 \in \nCoxWord{Q}{W}$ both have property $\nWordPropB$ and correspond to the same element in $\nCoxG$.
	Thus,~$w_1$ and~$w_2$ are related by simultaneous toggles as in \Cref{lem:word_ops_relating_B_words}.
	Let $u$ be a maximal $(\nCoxGenElt_i, \nCoxGenElt_j)$-alternating subword of~$w_1$ which is toggled. Then the following hold.
	\begin{enumerate}
		\item If $\ell(u) = m_{ij} - 1$, then there exists a maximal $(\nCoxGenElt_i,\nCoxGenElt_j)$-alternating subword~$v$ of $w_1$ with $\ell(v) = m_{ij} + 1$, which is toggled,
		\item There exists a maximal $(\nCoxGenElt_i, \nCoxGenElt_j)$-alternating subword $v$ of $w_1$ with $\ell(v) \geq 4$, and
		\item All other maximal $(\nCoxGenElt_i, \nCoxGenElt_j)$-alternating subwords~$v$ of~$w_1$ with $\ell(v) \geq 3$ are also toggled.
	\end{enumerate}
\end{lemma}
\begin{proof}
	We begin by observing that statement~$(2)$ immediately follows from statement~$(1)$, since $\ell(u) < 4$ only occurs if $\ell(u)=3$ and $m_{ij} = 4$.

	Now, we prove statement $(3)$ for the case $\ell(u) \geq 4$.
	Let $h_1,\ldots,h_m$ be the maximal alternating subwords of $w_1$ which are toggled to get $w_2$.
	Let $u^\prime$ be the subword of $w_2$ which is the result of toggling $u$.
	Any $h_k$ either has $\ell(h_k) \geq 4$, or it is toggled to $h^\prime_k$, a subword of $w_2$, with $\ell(h^\prime_k) \geq 5$.
	So, by $3$-coherence, no $h_k$ shares exactly one appearing letter with $u$, i.e.~all $h_k$ are $(\nCoxGenElt_x,\nCoxGenElt_y)$-alternating subwords with $\Abs{\Set{i,j} \cap \Set{x,y}} \in \Set{0,2}$.

	Let $v$ be any other maximal $(\nCoxGenElt_i,\nCoxGenElt_j)$-alternating subword of $w_1$.
	If $v$ overlaps any other $h_k$, but is not $h_k$ itself, then by \Cref{rmk:maximal_implies_overlap_at_most_1_and_share_exactly_1_letter}, $v$ (and thus $u$) must share exactly one appearing letter with~$h_k$, contradicting the statement above.
	So $v$ is either some~$h_k$, or is disjoint from all $h_1,\ldots,h_m$.

	Since $\ell(u) \geq 4$, by \Cref{cor:single_crossing_of_line_in_alternating_subword} and \Cref{lem:possible_lengths_other_crossings} we have $\ell(v) \geq 3$.
	Let~$v^\prime$ be the subword of $w_2$ which is the transport of~$v$ (a priori, we do not know if it is toggled or not, but it is transported because of our observations on its disjointness from the toggled subwords).
	By \Cref{lem:prev_region_constant} we have that $\nPrevRegion(v) = \nPrevRegion(v^\prime)$ and $\nPrevRegion(u) = \nPrevRegion(u^\prime)$.

	Since~$u$ is toggled, we have $\nOrdering(u) = -\nOrdering(u^\prime)$, so $Z(u) = -Z(u^\prime)$.
	By \Cref{lem:crossing_diagonals_and_one_change_of_sign} and \Cref{lem:change_of_sign_given_by_Z}, we require that $Z(u^\prime) = Z(v^\prime)$, so we must have $\nOrdering(v^\prime) = -\nOrdering(v)$, which can only occur if $v$ is toggled.
	We have partially proven $(3)$, dealing with the case $\ell(u) \geq 4$.

	We now prove statement $(1)$.
	If $\ell(u) = m_{ij} - 1$, then $u$ is toggled to $u^\prime$ in $w_2$, with $\ell(u^\prime) = m_{ij} + 1 \geq 5$.
	By \cref{prop:B_even_length_and_licence}, there must be a licence $x^\prime$ in $w_2$ with $\ell(x^\prime) = m_{ij} - 1$.
	Now, we swap the roles of $w_1$ and $w_2$, and apply what we have proven for statement $(3)$ to $u^\prime$.
	This tells us that $x^\prime$ must be toggled to get from $w_2$ to $w_1$, resulting in a maximal $(\nCoxGenElt_i, \nCoxGenElt_j)$-alternating subword $x$ of $w_1$, with $\ell(x) = m_{ij} + 1$.
	This proves statement $(1)$.

	Now we finish by proving statement $(3)$ for the case $\ell(u) = 3$.
	We must have that $m_{ij} = 4$ and $\ell(u) = m_{ij} - 1$, so by $(1)$ there exists a maximal $(\nCoxGenElt_i,\nCoxGenElt_j)$-alternating subword $x$ of $w_1$ which is toggled.
	We apply what we have already proven for statement $(3)$ to $x$.
\end{proof}
\begin{lemma}
	\label{lem:word_prop_B_equality_achieved_in_B_hat}
	Let~$\Gamma$ be a Coxeter graph with all $m_{ij} \geq 4$.
	Let $q_1,q_2 \in Q$ be such that $\nCoxWord{(q_1)}{W}$ and $\nCoxWord{(q_2)}{W}$ have property $\nWordPropB$, and $\nPiCox(q_1) = \nPiCox(q_2)$.
	Then there exists a $\hat{\beta} \in \nBraidSubgrpG$ such that $\hat{\beta} \star q_1 = q_2$.
\end{lemma}
\begin{proof}
	Let $w_1 \coloneq \nCoxWord{(q_1)}{W}$ and $w_2 \coloneq \nCoxWord{(q_2)}{W}$.
	By \Cref{lem:word_ops_relating_B_words} there is a set $h_1,\ldots,h_p$ of disjoint maximal alternating subwords of $w_1$ which are toggled to get $w_2$.
	We prove this claim by induction on $p$.

	For the base case $p=0$, we have $\nCoxWord{(q_1)}{W} = \nCoxWord{(q_2)}{W}$.
	Consider the Coxeter group~$U$ of rank~$\nRank$ for which all $m_{ij}=\infty$, and the canonical epimorphism $\nPiU \colon \nFree \to U$.
	Square-free words are a normal form for $U$, so by \Cref{lem:Q_square_free}, we have $\nPiU(q_1) = \nPiU(q_2)$.
	\cite[Corollary 4.7]{bessis_dual_2006} tells us that $\nPiU$ is injective on $Q$, so $q_1 = q_2$ and the relevant element of $\nBraidSubgrpG$ is the identity.

	We now show the inductive step, taking a running product in $\nBraidSubgrpG$.
	Suppose $p \geq 1$.
	Suppose $h_1$ is a maximal $(\nCoxGenElt_i,\nCoxGenElt_j)$-alternating subword of $w_1$, with corresponding $(\nFreeGenElt_i,\nFreeGenElt_j)$-alternating subword $u$ in $q_1$.
	By statement $(2)$ of \Cref{lem:toggle_applied_universally}, there is a maximal $(\nFreeGenElt_i,\nFreeGenElt_j)$-alternating subword $x$ of $q_1$ with $\ell(x) \geq 4$.
	Let $v_1,\ldots,v_t$ be the collection of all maximal $(\nFreeGenElt_i,\nFreeGenElt_j)$-alternating subwords of $q_1$, with all $\ell(v_k) \geq 3$, including $u$.
	Let $\gamma$ be a minimal simple loop representing $q_1$.
	By \Cref{cor:single_crossing_of_line_in_alternating_subword}, $\gamma_x$ crosses $\mu_{ij}$ and by \Cref{lem:possible_lengths_other_crossings} all other crossings of $\mu_{ij}$ occur in some $\gamma_{v_k}$.
	Thus, $\sigma_{ij}^{\pm m_{ij}}$ acts on $q_1$ by changing each $v_k$ according to~\eqref{eqn:action_of_sigma_ij_on_alternating_words}.

	Now, statement $(3)$ of \Cref{lem:toggle_applied_universally} tells us that every one of $\nCoxWord{(v_1)}{W},\ldots,\nCoxWord{(v_t)}{W}$ is toggled to get from $w_1$ to $w_2$.
	\Cref{lem:crossing_diagonals_and_one_change_of_sign} tells us that all of $v_1,\ldots,v_t$ have changes of sign of one type ($Z_+$ or $Z_-$).
	Let us assume all the changes of sign are of $Z_+$ type.
	Then \Cref{lem:B_word_operations_achieved_in_B_hat} tells us that toggling all of $\nCoxWord{(v_k)}{W}$ in $w_1$ results in the same word as
	\[
		\nCoxWord{\left( \sigma_{ij}^{- m_{ij}} \star q_1 \right)}{W}.
	\]
	If all the changes of sign were instead of $Z_-$ type, then we swap  $\sigma_{ij}^{- m_{ij}}$ for $\sigma_{ij}^{+ m_{ij}}$ and get the same result.
	Let $q^\prime \coloneq \sigma_{ij}^{\pm m_{ij}} \star q_1$ for the appropriate choice of sign.
	Then $\nCoxWord{(q^\prime)}{W}$ and~$\nCoxWord{(q_2)}{W}$ are related by toggling every subword in
	\[
		\Set{h_1,\ldots,h_p} \setminus \Set{\nCoxWord{(v_1)}{W},\ldots,\nCoxWord{(v_t)}{W}},
	\]
	which has cardinality strictly less than $p$.

	It remains to show that $w^\prime \coloneq \nCoxWord{(q^\prime)}{W}$ has property $\nWordPropB$.
	By \Cref{lem:Q_W_properties_and_B1_and_B1_implies_B} it is sufficient to show that $w^\prime$ has \cref{prop:B_odd_length,prop:B_even_length_and_licence}.
	We first focus on the \emph{length requirements} of \cref{prop:B_odd_length,prop:B_even_length_and_licence}, which we recall.
	\begin{enumerate}
		\item For odd $m_{ij}$, every $(\nCoxGenElt_i,\nCoxGenElt_j)$-alternating subword $h$ has $\ell(h) \leq m_{ij}$.
		      \smallskip
		\item For even $m_{ij}$, every $(\nCoxGenElt_i,\nCoxGenElt_j)$-alternating subword $h$ has $\ell(h) \leq m_{ij} + 1$.
	\end{enumerate}
	Let $h^\prime_1,\ldots,h^\prime_p$ be the resulting collection of maximal alternating subwords of $w^\prime$.
	These are the result of toggling each of $h_1,\ldots,h_p$ which is $(\nCoxGenElt_i,\nCoxGenElt_j)$-alternating.
	It is clear that no $h^\prime_k$ contradicts the length requirements.

	Any other alternating subword $z^\prime$ of $w^\prime$ which contradicts the length requirements must have $\ell(z^\prime) \geq 6$ and must overlap some $h^\prime_k$ which was toggled.
	Each $h^\prime_k$ is maximal in $w^\prime$ by \Cref{lem:square_free_maximal_is_conserved}, so by \Cref{rmk:maximal_implies_overlap_at_most_1_and_share_exactly_1_letter} such a $z^\prime$ must overlap $h^\prime_k$ over at most one letter, and have exactly one letter of $\Set{\nCoxGenElt_i, \nCoxGenElt_j}$ in common with it.
	Such a $z^\prime$ would contradict $3$-coherence in $w^\prime$ (which is already known to hold by \Cref{lem:Q_W_properties_and_B1_and_B1_implies_B}).
	We have shown that the length requirements are met in $w^\prime$.

	All that remains to be shown is that $w^\prime$ satisfies the remainder of \cref{prop:B_even_length_and_licence} regarding licences.
	The licences for any $(\nCoxGenElt_x,\nCoxGenElt_y)$-alternating words, where $\Set{i,j} \cap \Set{x,y} = \emptyset$, are unaffected by the toggles in question.
	So, we only need to focus on licences for $(\nCoxGenElt_i,\nCoxGenElt_j)$-alternating subwords of length $m_{ij} + 1$.
	In this case, $\ell(h^\prime_k) = m_{ij} + 1$ only if $\ell(h_k) = m_{ij} -1$.
	Then, by statement~$(1)$ of \Cref{lem:toggle_applied_universally}, there exists some $h_l$ with $\ell(h_l) = m_{ij} + 1$.
	Then $\ell(h^\prime_l) = m_{ij} - 1$ and $h^\prime_l$ is the required licence.
\end{proof}
\begin{theorem}
	\label{thm:final_theorem}
	Let $\Gamma$ be a Coxeter graph with all $m_{ij} \geq 4$.
	Then for each ${q_1,q_2 \in Q}$ such that $\nPiCox(q_1) = \nPiCox(q_2)$, there exists a $\hat{\beta} \in \nBraidSubgrpG$ such that $q_2 = \hat{\beta} \star q_1$, and thus $\nPiArt(q_1) = \nPiArt(q_2)$ and $\nArtG$ is canonically isomorphic to $\nDualArtG$.
\end{theorem}
\begin{proof}
	Let $q_1, q_2 \in Q$ be such that $\nPiCox(q_1) = \nPiCox(q_2)$.
	By \Cref{cor:bhat_B}, there exist $\hat{\delta}_1,\hat{\delta}_2 \in \nBraidSubgrpG$ such that $w_1 \coloneq \nCoxWord{(\hat{\delta}_1 \star q_1)}{W}$ and $w_2 \coloneq \nCoxWord{(\hat{\delta}_2 \star q_2)}{W}$ both have property $\nWordPropB$.

	Let $r_1 \coloneq \hat{\delta}_1 \star q_1$ and $r_2 \coloneq \hat{\delta}_2 \star q_2$.
	Equality in $\nCoxG$ and membership in $Q$ are preserved, i.e.~$\nPiCox(r_1) = \nPiCox(q_1) = \nPiCox(q_2) = \nPiCox(r_2)$ and $r_1,r_2 \in Q$.
	Then, by \Cref{lem:word_prop_B_equality_achieved_in_B_hat}, there exists a $\hat{\beta} \in \nBraidSubgrpG$ such that $\hat{\beta} \star r_1 = r_2$.
	We conclude that $(\hat{\delta}_2)^{-1}\hat{\beta}\hat{\delta}_1 \star q_1 = q_2$.
	The result follows by \Cref{lem:artin_preserving_braid_action}.
\end{proof}

%% file: figs/non_crossing_mu_ij.tex
\definecolor{c819d43}{RGB}{129,157,67}
\definecolor{c979797}{RGB}{151,151,151}

  \path[draw=black,line width=0.0116cm,dash pattern=on 0.0116cm off 0.0464cm,cm={ 1.2928,-0.0,-0.0,1.295,(-1.8972, -4.568)}] (4.8222, 15.8642).. controls (4.8396, 15.6148) and (4.9412, 15.3721) .. (5.1068, 15.1848).. controls (5.2723, 14.9976) and (5.5007, 14.8669) .. (5.746, 14.8191).. controls (6.0641, 14.7571) and (6.4073, 14.838) .. (6.6641, 15.0356).. controls (6.9209, 15.2332) and (7.0871, 15.5442) .. (7.1086, 15.8675);

  \path[draw=black,line width=0.0116cm,dash pattern=on 0.0116cm off 0.0464cm,cm={ 1.2928,-0.0,-0.0,1.295,(-1.8972, -4.568)}] (5.9687, 15.8642).. controls (5.996, 15.5636) and (6.1475, 15.2763) .. (6.3799, 15.0839).. controls (6.6124, 14.8914) and (6.923, 14.7964) .. (7.2233, 14.8257).. controls (7.4871, 14.8514) and (7.7407, 14.972) .. (7.9273, 15.1603).. controls (8.1138, 15.3487) and (8.232, 15.6034) .. (8.2552, 15.8675);

  \node[text=black,line width=0.005cm,anchor=south west] (text57948-6-7) at (4.8283, 14.4769){$\scriptstyle{\mu_{i,j}}$};

  \node[text=black,line width=0.005cm,anchor=south west] (text16) at (7.3683, 14.3711){$\scriptstyle{\mu_{k,l}}$};

  \begin{scope}[shift={(-6.1503, 5.7644)}]
    \path[draw=c819d43,line width=0.02cm] (10.4872, 11.2055) -- (10.4872, 10.2123);

    \path[draw=black,fill=c979797,line width=0.0487cm] (10.4872, 10.2123) circle (0.0458cm);

    \node[text=black,line width=0.005cm,anchor=south west] (text57948-7) at (10.1785, 11.3184){$\scriptscriptstyle\lambda_{i}$};

  \end{scope}
  \begin{scope}[shift={(-4.6681, 5.7644)}]
    \path[draw=c819d43,line width=0.02cm] (10.4872, 11.2055) -- (10.4872, 10.2123);

    \path[draw=black,fill=c979797,line width=0.0487cm] (10.4872, 10.2123) circle (0.0458cm);

    \node[text=black,line width=0.005cm,anchor=south west] (text57948-7-2) at (10.1785, 11.3184){$\scriptscriptstyle\lambda_{k}$};

  \end{scope}
  \begin{scope}[shift={(-3.1945, 5.7686)}]
    \path[draw=c819d43,line width=0.02cm] (10.4872, 11.2055) -- (10.4872, 10.2123);

    \path[draw=black,fill=c979797,line width=0.0487cm] (10.4872, 10.2123) circle (0.0458cm);

    \node[text=black,line width=0.005cm,anchor=south west] (text57948-7-28) at (10.1785, 11.3184){$\scriptscriptstyle\lambda_{j}$};

  \end{scope}
  \begin{scope}[shift={(-1.7123, 5.7686)}]
    \path[draw=c819d43,line width=0.02cm] (10.4872, 11.2055) -- (10.4872, 10.2123);

    \path[draw=black,fill=c979797,line width=0.0487cm] (10.4872, 10.2123) circle (0.0458cm);

    \node[text=black,line width=0.005cm,anchor=south west] (text57948-7-1) at (10.1785, 11.3184){$\scriptscriptstyle\lambda_{l}$};

  \end{scope}
  \path[draw=black,line width=0.02cm] (4.0994, 16.3022) -- (4.4599, 16.3022).. controls (4.4769, 16.3021) and (4.4938, 16.2977) .. (4.5086, 16.2895).. controls (4.5235, 16.2813) and (4.5362, 16.2692) .. (4.5452, 16.2548).. controls (4.5559, 16.2377) and (4.5613, 16.2176) .. (4.5635, 16.1975).. controls (4.5658, 16.1774) and (4.5651, 16.1571) .. (4.5645, 16.1369).. controls (4.563, 16.0864) and (4.5623, 16.0358) .. (4.5624, 15.9852).. controls (4.548, 15.6718) and (4.6722, 15.3546) .. (4.8955, 15.1343).. controls (5.0604, 14.9716) and (5.2769, 14.8619) .. (5.5057, 14.8251).. controls (5.5552, 14.8171) and (5.6026, 14.7969) .. (5.6427, 14.7668).. controls (5.7002, 14.7236) and (5.7412, 14.6614) .. (5.7687, 14.595).. controls (5.7976, 14.5253) and (5.8133, 14.4491) .. (5.8559, 14.3868).. controls (5.8945, 14.3303) and (5.9562, 14.2883) .. (6.0242, 14.2803).. controls (6.0537, 14.2768) and (6.0839, 14.2797) .. (6.1122, 14.2886).. controls (6.5357, 14.3713) and (6.9279, 14.6057) .. (7.2015, 14.9394).. controls (7.444, 15.2352) and (7.5926, 15.6069) .. (7.6209, 15.9884).. controls (7.6278, 16.0807) and (7.5972, 16.1753) .. (7.5374, 16.246).. controls (7.4774, 16.3171) and (7.388, 16.3633) .. (7.295, 16.3682).. controls (7.2387, 16.3711) and (7.1819, 16.3592) .. (7.1309, 16.3352).. controls (7.0799, 16.3111) and (7.0348, 16.275) .. (6.9992, 16.2313).. controls (6.9444, 16.1641) and (6.9122, 16.0787) .. (6.9091, 15.992).. controls (6.9043, 15.8054) and (6.8508, 15.6202) .. (6.7552, 15.4598).. controls (6.6597, 15.2994) and (6.5224, 15.1641) .. (6.3605, 15.071).. controls (6.1946, 14.9756) and (6.0034, 14.9247) .. (5.812, 14.9249).. controls (5.6206, 14.925) and (5.4295, 14.976) .. (5.2634, 15.071).. controls (5.1011, 15.1638) and (4.9629, 15.2985) .. (4.8669, 15.4589).. controls (4.7709, 15.6193) and (4.7174, 15.8051) .. (4.7148, 15.992).. controls (4.7142, 16.0359) and (4.7163, 16.0797) .. (4.7185, 16.1235).. controls (4.7209, 16.1701) and (4.7234, 16.2172) .. (4.7161, 16.2634).. controls (4.7095, 16.3047) and (4.6949, 16.345) .. (4.6717, 16.3797).. controls (4.6253, 16.4491) and (4.5434, 16.4933) .. (4.4599, 16.494) -- (4.0994, 16.494);

%% file: figs/R_24.tex
\definecolor{c819d43}{RGB}{129,157,67}
\definecolor{ce6e6e6}{RGB}{230,230,230}
\definecolor{ca2a2a2}{RGB}{162,162,162}
\definecolor{c979797}{RGB}{151,151,151}

  \path[draw=c819d43,line width=0.02cm] (10.4534, 17.3068) -- (9.3255, 18.8354);

  \path[fill=ce6e6e6,line width=0.015cm] (9.2884, 16.7456).. controls (9.2848, 16.7461) and (9.2725, 16.7476) .. (9.2611, 16.7489).. controls (9.18, 16.7582) and (9.0859, 16.7934) .. (9.0166, 16.8405).. controls (8.9737, 16.8696) and (8.9266, 16.9133) .. (8.8952, 16.9529).. controls (8.8397, 17.0231) and (8.8015, 17.1062) .. (8.784, 17.1949).. controls (8.7766, 17.2319) and (8.7762, 17.2298) .. (8.7921, 17.2356).. controls (8.8091, 17.2418) and (8.8265, 17.2577) .. (8.8348, 17.2747).. controls (8.84, 17.2853) and (8.8408, 17.2895) .. (8.8408, 17.3066).. controls (8.8408, 17.3232) and (8.84, 17.3281) .. (8.8352, 17.3383).. controls (8.8294, 17.3505) and (8.818, 17.3638) .. (8.806, 17.3722) -- (8.7995, 17.3768) -- (9.0623, 18.0943).. controls (9.2069, 18.489) and (9.3254, 18.8116) .. (9.3258, 18.8112).. controls (9.3262, 18.8108) and (9.4459, 18.4878) .. (9.5918, 18.0935) -- (9.857, 17.3764) -- (9.8522, 17.3733).. controls (9.8495, 17.3715) and (9.8423, 17.3651) .. (9.8361, 17.3589).. controls (9.8027, 17.3252) and (9.81, 17.267) .. (9.8506, 17.2426).. controls (9.8572, 17.2387) and (9.8663, 17.2345) .. (9.8708, 17.2333).. controls (9.8804, 17.2307) and (9.8804, 17.2348) .. (9.871, 17.188).. controls (9.8417, 17.0417) and (9.7517, 16.9112) .. (9.6246, 16.8307).. controls (9.5989, 16.8144) and (9.5456, 16.7885) .. (9.5155, 16.7775).. controls (9.4866, 16.7671) and (9.4367, 16.7545) .. (9.4052, 16.7498).. controls (9.3835, 16.7466) and (9.3025, 16.7437) .. (9.2884, 16.7457) -- cycle;

  \path[draw=c819d43,line width=0.02cm] (8.2032, 17.3068) -- (9.3231, 18.8389);

  \path[draw=c819d43,line width=0.02cm] (8.7657, 17.3068) -- (9.3255, 18.8354);

  \path[draw=c819d43,line width=0.02cm] (9.3283, 17.3068) -- (9.3255, 18.8354);

  \path[draw=c819d43,line width=0.02cm] (9.8908, 17.3068) -- (9.3255, 18.8354);

  \path[draw=ca2a2a2,line width=0.02cm,dash pattern=on 0.06cm off 0.02cm,cm={ 0.0,1.0,1.0,0.0,(-29.7, 29.7)}] (-12.3901, 39.0255) circle (1.5255cm);

  \path[draw=black,line width=0.0161cm,dash pattern=on 0.0161cm off 0.0643cm,cm={ 0.9335,-0.0,-0.0,0.9335,(5.2462, -4.3524)}] (3.77, 23.2025).. controls (3.7679, 23.0439) and (3.8313, 22.8854) .. (3.9424, 22.7722).. controls (4.0534, 22.6589) and (4.2106, 22.5924) .. (4.3692, 22.5915).. controls (4.529, 22.5906) and (4.688, 22.6563) .. (4.8005, 22.7697).. controls (4.913, 22.8831) and (4.9775, 23.0427) .. (4.9753, 23.2025);

  \node[text=black,line width=0.005cm,anchor=south west] (text57948) at (9.0763, 16.9496){$\scriptscriptstyle R_{2,4}$};

  \path[draw=black,fill=c979797,line width=0.0487cm,cm={ 0.0,1.0,1.0,0.0,(-29.7, 29.7)}] (-13.9159, 39.0201) circle (0.0528cm);

  \path[draw=black,fill=c979797,line width=0.0487cm,cm={ 0.0,1.0,1.0,0.0,(-29.7, 29.7)}] (-12.3932, 40.1534) circle (0.0528cm);

  \path[draw=black,fill=c979797,line width=0.0487cm,cm={ 0.0,1.0,1.0,0.0,(-29.7, 29.7)}] (-12.3932, 37.9032) circle (0.0528cm);

  \path[draw=black,fill=c979797,line width=0.0487cm,cm={ 0.0,1.0,1.0,0.0,(-29.7, 29.7)}] (-12.3932, 38.4657) circle (0.0528cm);

  \path[draw=black,fill=c979797,line width=0.0487cm,cm={ 0.0,1.0,1.0,0.0,(-29.7, 29.7)}] (-12.3932, 39.0283) circle (0.0528cm);

  \path[draw=black,fill=c979797,line width=0.0487cm,cm={ 0.0,1.0,1.0,0.0,(-29.7, 29.7)}] (-12.3932, 39.5908) circle (0.0528cm);

  \node[text=black,line width=0.005cm,anchor=south west] (text57948-3) at (8.7165, 16.5066){$\scriptstyle \mu_{2,4}$};

%% file: figs/crossing_io_01.tex
\definecolor{c819d43}{RGB}{129,157,67}
\definecolor{c979797}{RGB}{151,151,151}

  \begin{scope}[shift={(0.8956, 6.6924)}]
    \path[draw=c819d43,line width=0.02cm] (10.4872, 11.6817) -- (10.4872, 10.2123);

    \path[draw=black,fill=c979797,line width=0.0487cm] (10.4872, 10.2123) circle (0.0458cm);

    \node[text=black,line width=0.005cm,anchor=south west] (text57948-7) at (10.1785, 11.7946){$\scriptscriptstyle\lambda_{j}$};

  \end{scope}
  \begin{scope}[shift={(-1.0844, 6.6924)}]
    \path[draw=c819d43,line width=0.02cm] (10.4872, 11.6817) -- (10.4872, 10.2123);

    \path[draw=black,fill=c979797,line width=0.0487cm] (10.4872, 10.2123) circle (0.0458cm);

    \node[text=black,line width=0.005cm,anchor=south west] (text17) at (10.1785, 11.7946){$\scriptscriptstyle\lambda_{i}$};

  \end{scope}
  \begin{scope}[shift={(2.9902, 6.6924)}]
    \path[draw=c819d43,line width=0.02cm] (10.4872, 11.6817) -- (10.4872, 10.2123);

    \path[draw=black,fill=c979797,line width=0.0487cm] (10.4872, 10.2123) circle (0.0458cm);

    \node[text=black,line width=0.005cm,anchor=south west] (text1) at (10.1785, 11.7946){$\scriptscriptstyle\lambda_{k}$};

  \end{scope}
  \path[draw=black,line width=0.02cm] (13.7786, 17.325) -- (13.2803, 17.325) -- (12.8733, 17.325).. controls (12.8046, 17.325) and (12.7351, 17.3235) .. (12.6684, 17.307).. controls (12.6017, 17.2904) and (12.5387, 17.259) .. (12.4854, 17.2155).. controls (12.4322, 17.1719) and (12.3891, 17.1173) .. (12.3515, 17.0597).. controls (12.3139, 17.0021) and (12.2812, 16.9414) .. (12.2421, 16.8848).. controls (12.1328, 16.7267) and (11.9686, 16.6034) .. (11.7814, 16.5597).. controls (11.6878, 16.5379) and (11.5893, 16.536) .. (11.4954, 16.5565).. controls (11.4014, 16.577) and (11.3122, 16.6201) .. (11.2393, 16.6827).. controls (11.1315, 16.7752) and (11.061, 16.9131) .. (11.0606, 17.0552).. controls (11.0604, 17.1262) and (11.0774, 17.1974) .. (11.1111, 17.2599).. controls (11.1448, 17.3225) and (11.1952, 17.3761) .. (11.2563, 17.4123).. controls (11.3319, 17.457) and (11.4236, 17.4742) .. (11.5101, 17.459).. controls (11.5967, 17.4438) and (11.6773, 17.3962) .. (11.7318, 17.3273).. controls (11.7811, 17.2648) and (11.8086, 17.1861) .. (11.8128, 17.1066).. controls (11.817, 17.0271) and (11.7984, 16.947) .. (11.7631, 16.8756).. controls (11.6926, 16.7329) and (11.5586, 16.6289) .. (11.4121, 16.5665).. controls (11.1902, 16.4721) and (10.9408, 16.4649) .. (10.7007, 16.4869).. controls (10.5991, 16.4962) and (10.4979, 16.5105) .. (10.3977, 16.5296);

  \path[draw=black,fill=black,line width=0.0169cm,cm={ -0.1845,0.4995,-0.4995,-0.1845,(22.0328, 15.792)}] (8.1902, 16.8198) -- (8.0991, 16.7825) -- (8.1124, 16.8801) -- cycle;

  \path[draw=black,fill=black,line width=0.0169cm,cm={ -0.5318,-0.0262,0.0262,-0.5318,(14.6291, 25.6449)}] (8.1902, 16.8198) -- (8.0991, 16.7825) -- (8.1124, 16.8801) -- cycle;

%% file: figs/crossing_io_02.tex
\definecolor{c819d43}{RGB}{129,157,67}
\definecolor{c979797}{RGB}{151,151,151}

  \begin{scope}[shift={(0.9529, 6.6924)}]
    \path[draw=c819d43,line width=0.02cm] (10.4872, 11.6817) -- (10.4872, 10.2123);

    \path[draw=black,fill=c979797,line width=0.0487cm] (10.4872, 10.2123) circle (0.0458cm);

    \node[text=black,line width=0.005cm,anchor=south west] (text57948-7) at (10.1785, 11.7946){$\scriptscriptstyle\lambda_{k}$};

  \end{scope}
  \begin{scope}[shift={(-1.0844, 6.6924)}]
    \path[draw=c819d43,line width=0.02cm] (10.4872, 11.6817) -- (10.4872, 10.2123);

    \path[draw=black,fill=c979797,line width=0.0487cm] (10.4872, 10.2123) circle (0.0458cm);

    \node[text=black,line width=0.005cm,anchor=south west] (text17) at (10.1785, 11.7946){$\scriptscriptstyle\lambda_{i}$};

  \end{scope}
  \begin{scope}[shift={(2.9902, 6.6924)}]
    \path[draw=c819d43,line width=0.02cm] (10.4872, 11.6817) -- (10.4872, 10.2123);

    \path[draw=black,fill=c979797,line width=0.0487cm] (10.4872, 10.2123) circle (0.0458cm);

    \node[text=black,line width=0.005cm,anchor=south west] (text1) at (10.1785, 11.7946){$\scriptscriptstyle\lambda_{j}$};

  \end{scope}
  \path[draw=black,line width=0.02cm] (11.7397, 17.2438) -- (11.4343, 17.2438) -- (11.2523, 17.2438).. controls (11.0792, 17.2438) and (10.9043, 17.2348) .. (10.7386, 17.1852).. controls (10.5728, 17.1355) and (10.42, 17.0507) .. (10.2715, 16.9619).. controls (10.0667, 16.8395) and (9.8629, 16.7105) .. (9.6372, 16.6331).. controls (9.5308, 16.5965) and (9.4174, 16.5716) .. (9.3059, 16.5868).. controls (9.2501, 16.5943) and (9.1954, 16.612) .. (9.1474, 16.6415).. controls (9.0995, 16.6709) and (9.0584, 16.7124) .. (9.0323, 16.7623).. controls (9.01, 16.8051) and (8.9988, 16.8536) .. (8.9997, 16.9019).. controls (9.0006, 16.9502) and (9.0134, 16.9981) .. (9.0361, 17.0407).. controls (9.0816, 17.1259) and (9.1671, 17.1876) .. (9.2614, 17.2086).. controls (9.3276, 17.2234) and (9.3973, 17.219) .. (9.4631, 17.202).. controls (9.5288, 17.185) and (9.5908, 17.1555) .. (9.649, 17.1205).. controls (9.8551, 16.9965) and (10.0115, 16.8061) .. (10.1952, 16.6507).. controls (10.2876, 16.5725) and (10.3883, 16.5024) .. (10.5005, 16.4569).. controls (10.6503, 16.3962) and (10.8183, 16.3817) .. (10.9762, 16.4158);

  \path[draw=black,fill=black,line width=0.0169cm,cm={ -0.4157,-0.3327,0.3327,-0.4157,(8.1975, 26.2001)}] (8.1902, 16.8198) -- (8.0991, 16.7825) -- (8.1124, 16.8801) -- cycle;

  \path[draw=black,fill=black,line width=0.0169cm,cm={ -0.4412,-0.2981,0.2981,-0.4412,(9.0914, 26.9487)}] (8.1902, 16.8198) -- (8.0991, 16.7825) -- (8.1124, 16.8801) -- cycle;